\documentclass[leqno,11pt]{amsart}
\usepackage{amssymb, amsmath}
\usepackage{amsthm, amsfonts,mathrsfs}
\def\inte#1{
\displaystyle\mathop{#1\kern0pt}^\circ }

\let\pa=\partial
\let\al=\alpha

\let\d=\delta
\let\e=\varepsilon

\let\r=\rho

\let\f=\frac

\let\p=\psi

\let\D=\Delta

\let\Om=\Omega
\let\wt=\widetilde

\def\cB{{\mathcal B}}
\def\cC{{\mathcal C}}

\def\cF{{\mathcal F}}

\def\cL{{\mathcal L}}
\def\cM{{\mathcal M}}

\def\cR{{\mathcal R}}
\def\cS{{\mathcal S}}

\def\dH{\dot{H}}

\def\pa{\partial}
\def\grad{\nabla}
\def\dB{\dot{B}}

\def\dD{\dot{\Delta}}
\def\virgp{\raise 2pt\hbox{,}}
\def\cdotpv{\raise 2pt\hbox{;}}

\def\eqdefa{\buildrel\hbox{\footnotesize def}\over =}

\def\C{\mathop{\mathbb C\kern 0pt}\nolimits}
\def\DD{\mathop{\mathbb D\kern 0pt}\nolimits}
\def\EE{\mathop{{\mathbb E \kern 0pt}}\nolimits}
\def\K{\mathop{\mathbb K\kern 0pt}\nolimits}
\def\N{\mathop{\mathbb N\kern 0pt}\nolimits}
\def\Q{\mathop{\mathbb Q\kern 0pt}\nolimits}
\def\R{\mathop{\mathbb R\kern 0pt}\nolimits}
\def\SS{\mathop{\mathbb S\kern 0pt}\nolimits}
\def\ZZ{\mathop{\mathbb Z\kern 0pt}\nolimits}
\def\TT{\mathop{\mathbb T\kern 0pt}\nolimits}
\def\P{\mathop{\mathbb P\kern 0pt}\nolimits}

\newcommand{\la}{\lambda}

\newcommand{\Z}{{\ZZ}}

\def\dv{\mbox{div}}

\def\dive{\mathop{\rm div}\nolimits}

\def\Supp{\mathop{\rm Supp}\nolimits\ }

\def\no{\noindent}
\def\na{\nabla}
\def\p{\partial}
\def\th{\theta}

\newcommand{\w}[1]{\langle {#1} \rangle}
\newcommand{\beq}{\begin{equation}}
\newcommand{\eeq}{\end{equation}}
\newcommand{\ben}{\begin{eqnarray}}
\newcommand{\een}{\end{eqnarray}}
\newcommand{\beno}{\begin{eqnarray*}}
\newcommand{\eeno}{\end{eqnarray*}}

\newtheorem{defi}{Definition}[section]
\newtheorem{thm}{Theorem}[section]
\newtheorem{lem}{Lemma}[section]
\newtheorem{rmk}{Remark}[section]
\newtheorem{col}{Corollary}[section]
\newtheorem{prop}{Proposition}[section]
\renewcommand{\theequation}{\thesection.\arabic{equation}}

\begin{document}

\title[ Well-posedness of 2-D inhomogeneous Navier-Stokes system]
{On the  global well-posedness of 2-D density-dependent
Navier-Stokes system with variable viscosity}

\author[H.  Abidi]{Hammadi Abidi}
\address[H.  Abidi]{D\'epartement de Math\'ematiques
Facult\'e des Sciences de Tunis Campus universitaire 2092 Tunis,
Tunisia} \email{habidi@univ-evry.fr}
\author[P. Zhang]{Ping Zhang} \address[P. Zhang]{Academy of Mathematics $\&$ Systems Science
and  Hua Loo-Keng Key Laboratory of Mathematics, Chinese Academy of
Sciences, Beijing 100190, CHINA.} \email{zp@amss.ac.cn}

\date{Jan. 10, 2013}

\maketitle
\begin{abstract} Given solenoidal vector $u_0\in H^{-2\d}\cap H^1(\R^2),$ $\r_0-1\in L^2(\R^2),$ and  $\r_0 \in L^\infty\cap\dot{W}^{1,r}(\R^2)$
 with a positive lower bound
 for $\d\in (0,\f12)$ and
 $2<r<\f{2}{1-2\d},$ we prove that 2-D incompressible inhomogeneous
Navier-Stokes system \eqref{1.1} has a unique global solution
provided that the viscous coefficient $\mu(\r_0)$ is close enough to
$1$ in the $L^\infty$ norm  compared to the size of $\d$ and the
norms of the initial data. With smoother initial data, we can prove
the propagation of regularities for such solutions. Furthermore, for
$1<p<4,$ if $(\r_0-1,u_0)$ belongs to the critical Besov spaces
$\dB^{\f2p}_{p,1}(\R^2)\times \bigl(\dB^{-1+\f2p}_{p,1}\cap
L^2(\R^2)\bigr)$ and the $\dB^{\f2p}_{p,1}(\R^2)$ norm of $\r_0-1$
is sufficiently small compared to the exponential of
$\|u_0\|_{L^2}^2+\|u_0\|_{\dB^{-1+\f2p}_{p,1}},$ we  prove the
global well-posedness of \eqref{1.1} in the scaling invariant
spaces. Finally for initial data in the almost critical Besov
spaces, we prove the global well-posedness of \eqref{1.1} under the
assumption that the $L^\infty$ norm of $\r_0-1$ is sufficiently
small.
\end{abstract}

\noindent {\sl Keywords:} Inhomogeneous  Navier-Stokes systems,
Littlewood-Paley Theory, critical\par

 \qquad \quad \ \ \ regularity

\vskip 0.2cm

\noindent {\sl AMS Subject Classification (2000):} 35Q30, 76D03  \\

\renewcommand{\theequation}{\thesection.\arabic{equation}}
\setcounter{equation}{0}

\section{Introduction}

The purpose of this paper is to investigate the global
well-posedness of the following two-dimensional incompressible
inhomogeneous Navier-Stokes equations with variable viscous
coefficient \beq\label{1.1} \left\{\begin{array}{l}
\displaystyle \pa_t \rho + \dv (\rho u)=0,\qquad (t,x)\in\R^+\times\R^2, \\
\displaystyle \pa_t (\rho u) + \dv (\rho u\otimes u) -\dv (2\mu(\rho) d)+\grad\Pi=0, \\
\displaystyle \dv\, u = 0, \\
\displaystyle \rho|_{t=0}=\rho_0,\quad \rho u|_{t=0}=m_0,
\end{array}\right.
\eeq where $\rho, u=(u_1,u_2)$ stand for the density and velocity of
the fluid respectively, and
$d=\bigl(\frac{1}{2}(\pa_{i}u_{j}+\pa_{j} u_{i})\bigr)_{2\times2}$
denotes the deformation tensor, $\Pi$  is a scalar pressure
function,
 and  the
viscous coefficient $\mu(\rho)$ is a smooth, positive and
non-decreasing function on $[0,\infty).$  Such a system describes
for instance a  fluid that is  incompressible but has nonconstant
density owing to  the complex structure of the  flow due to a
mixture (e.g. blood flow) or pollution (e.g. model of rivers).  It
may also describe a fluid containing a melted substance.

When $\mu(\rho)$ is  a positive constant, and  the initial density
has a positive lower bound, Lady\v zenskaja and Solonnikov \cite{LS}
first addressed the question of unique solvability of (\ref{1.1}).
More precisely, they considered the system \eqref{1.1} in a bounded
domain $\Om$ with homogeneous Dirichlet boundary condition for $u.$
Under the assumptions that $u_0\in W^{2-\frac2p,p}(\Om)$ $(p>d)$ is
divergence free and vanishes on $\p\Om$ and that $\r_0\in C^1(\Om)$
is bounded away from zero, then they \cite{LS} proved
\begin{itemize}
\item Global well-posedness in dimension $d=2;$
\item Local well-posedness in dimension $d=3.$ If in addition $u_0$ is small in $W^{2-\frac2p,p}(\Om),$
then global well-posedness holds true.
\end{itemize}
 Danchin
\cite{danchin04} proved similar well-posedness result of \eqref{1.1}
in the whole space  and with the initial data in the almost critical
Sobolev spaces. In particular, in two space dimensions, he proved
the the global well-poseness of \eqref{1.1} with $\mu(\r)=\mu>0$
provided that the initial data $(\r_0,u_0)$ satisfies \beno
&&\r_0-1\in H^{1+\alpha}(\R^2) \ (\nabla \r_0\in
L^\infty(\R^2)\  \mbox{if}\  \alpha=1), \quad  \r_0\geq \underline b>0, \quad \mbox{ and}\\
&& u_0\in H^\beta(\R^2)\quad \mbox{for any}\quad \alpha>0,\quad
\beta\in(0,\alpha)\cap(\alpha-1,\alpha+1).\eeno Very recently,
Paicu, Zhang and Zhang \cite{PZZ1} proved the global well-posedness
of \eqref{1.1} with $\mu(\r)=\mu>0$ for initial data: $\r_0\in
L^\infty(\R^2)$ with a positive lower bound and $u_0\in H^s(\R^2)$
for any $s>0.$ This result improves the former interesting
well-posedness theorem of Danchin and Mucha \cite{DM2} by removing
the smallness assumption on the fluctuation to the initial density
and also with much less regularity for the initial velocity.

In  general,  Lions  \cite{LP} (see also the references therein, and
the monograph \cite{AKM}) proved the global existence of finite
energy weak solutions to \eqref{1.1}.  Yet the uniqueness and
regularities of such weak solutions are big open questions  even in
two space dimensions, as was mentioned by Lions in \cite{LP} (see
page 31-32 of \cite{LP}). Except  under the additional assumptions
that \beq\label{des} \|\mu(\rho_0)-1 \|_{L^\infty({\Bbb T}^2)}\leq
\varepsilon_0\quad \mbox{and}\quad u_0\in H^1({\Bbb T}^2),\eeq
Desjardins \cite{desjardins} proved that the global weak solution,
$(\r,u,\na\Pi),$ constructed in  \cite{LP} satisfies $u\in
L^\infty((0,T); H^1({\Bbb T}^2))$ and $\rho\in
L^\infty((0,T)\times{\Bbb T}^2)$  for any $T<\infty.$ Moreover, with
additional regularity assumptions on the initial data, he could also
prove that $u\in L^2((0,\tau); H^2({\Bbb T}^2))$ for some short time
$\tau$ (see Theorem \ref{Desj} below).

To understand the system \eqref{1.1} further, the second author to
this paper proved the global well-posedness to a modified 2-D model
system, which coincides with the 2-D inhomogeneous Navier-Stokes
equations with $\mu(\r)=\mu>0$ , with general initial data in
\cite{Zhang}. Gui and Zhang \cite{GZ} proved the global
well-posedness of \eqref{1.1} with initial data satisfying
$\|\rho_0-1\|_{H^{s+1}}$ being sufficiently small and $u_0\in
H^s\cap \dot{H}^{-2\d}(\R^2)$ for some $s>2$ and $0<\d<\f12.$ Yet
the exact size of $\|\rho_0-1\|_{H^{s+1}}$ was not presented in
\cite{GZ}. Huang,  Paicu and Zhang \cite{HPZ1} basically proved that
as long as
\begin{equation}\label{thmassume} \eta\eqdefa
\|{\r_0-1}\|_{B^{\f2p}_{p,1}}\exp\Bigl\{C_0\bigl(1+\mu^2(1)\bigr)\exp\bigl(\f{C_0}{\mu^2(1)}\|u_0\|_{B^{-1+\f2p}_{p,1}}^2\bigr)\Bigr\}
\leq \f{c_0\mu(1)}{1+\mu(1)},
\end{equation} for some sufficiently small $c_0,$
\eqref{1.1} has a  global solution so that $\r-1\in
\cC_b([0,\infty);B^{\f2p}_{p,1}(\R^2))$ and $u \in
\cC_b([0,\infty);B^{-1+\f2p}_{p,1}(\R^2))\cap
L^1(\mathbb{R}^+;B^{1+\f2p}_{p,1}(\R^2))$ for $1<p<4.$ In a recent
preprint  \cite{HP1},  Huang and Paicu can prove the global
existence of solution of \eqref{1.1} with much weaker assumption
than \eqref{thmassume}. Yet as there is no $L^1((0,T);Lip(\R^2))$
estimate for the velocity field, the uniqueness of such solutions is
not clear in \cite{HP1}.

Let $\mathcal{R}$ be the usual Riesz transform,
 $\mathcal{Q}\eqdefa \nabla(\Delta)^{-1}\dv,$  and
$\mathbb{P}\eqdefa I-\mathcal{Q}$ be the Leray projection operator
on the space of divergence-free vector fields, we first recall the
following result from Desjardins \cite{desjardins}:

\begin{thm}\label{Desj}
{\sl Let $\r\in L^\infty({\Bbb T}^2),$ $u_0\in H^1({\Bbb T}^2)$ with
$\dv\, u_0 = 0.$  Then there exists a positive constant
$\varepsilon_0$ such that under the assumption of \eqref{des}, Lions
weak solutions (\cite{LP}) to \eqref{1.1} satisfy the following
regularity properties  for all $T>0:$
\begin{itemize}
  \item $u\in L^\infty((0,T);\,H^1({\Bbb T}^2)) $ and $\sqrt{\rho}\partial_tu\in L^2((0,T)\times {\Bbb T}^2),$
  \item $\rho$ and $\mu(\rho)\in L^\infty((0,T)\times {\Bbb T}^2)\cap
  \cC([0,T];\,L^p({\Bbb T}^2))$ for all $p\in[1,\infty),$
  \item $\nabla(\Pi-\mathcal{R}_i\mathcal{R}_j(2\mu d_{ij}))$ and
  $\nabla(\mathbb{P}\bigotimes\mathcal{Q}(2\mu d))_{ij}\in L^2((0,T)\times {\Bbb T}^2),$
  \item $\Pi$ may be renormalized in such a way that for some
universal constant $C_0>0,$ \beq\label{upi} \Pi\quad\mbox{and}\quad
\nabla u\in L^2((0,T);\, L^p({\Bbb T}^2))\quad\mbox{for all}\,\;
p\in [4,p^*], \eeq where \beq\label{pstar}
\frac{1}{p^*}=2C_0\|\mu(\rho_0)-1\|_{L^\infty}. \eeq
\end{itemize}
Moreover, if $\mu(\r_0)\geq \underline{\mu}$ and $\log(\mu(\r_0))\in
W^{1,r}({\Bbb T}^2)$ for some $r>2,$ there exists some positive time
$\tau$ so that $u\in L^2((0,\tau); H^2({\Bbb T}^2))$ and $\mu(\r)\in
C([0,\tau];W^{1,\bar{r}}({\Bbb T}^2))$ for any $\bar{r}<r.$ }
\end{thm}

In what follows, we shall always assume that
$$
0<\underline\mu\leq\mu(\r_0),\quad\mu(\cdot)\in
W^{2,\infty}(\R^+)\quad \mbox{and}\quad \mu(1)=1.
$$

\no{\bf Notations:} In the rest of this paper, we always denote
$a_+$ to be any number strictly bigger than $a$ and $a_-$  any
number strictly less then $a.$ We shall denote $[Y]$ the integer
part of $Y,$ and $\bar{C}$ to be a uniform constant depending only
$m,M$ in \eqref{thma.1} below and $\|\mu'\|_{L^\infty},$ which may
change from line to line.

Our first purpose in this paper is to prove the following global
well-posedness result for \eqref{1.1}.

\begin{thm}\label{thm1.1}
{\sl Let  $m, M$ be two positive constants and $\d\in (0,\f12),$
$2<r<\f2{1-2\d}.$ Let $u_0\in H^{-2\d}\cap H^1(\R^2)$ be a
solenoidal vector filed, and $\rho_0-1\in L^2\cap L^\infty\cap
\dot{W}^{1,r}(\R^2)$ satisfy \beq \label{thma.1} m\leq \r_0\leq
M,\quad \|\mu(\r_0)-1\|_{L^\infty}\leq \e_0, \eeq and for some $q\in
(1/\d,p^\ast],$ \beq \label{p.2qwt} C_0\eqdefa \|u_0\|_{H^{-2\d}}^2+
\|\r_0-1\|_{L^2}^4+\| u_0\|_{L^2}^4+\|\na
u_0\|_{L^2}^2\exp\bigl(C\|u_0\|_{L^2}^4\bigr). \eeq  there holds
 \beq \label{thm1.1av}
 \|\mu(\r_0)-1\|_{L^\infty}\Bigl(\f1\d+4\bar{C}^2C_0\bigl(1+\|\r_0\|_{B^{(2/q)_+}_{\infty,\infty}}\bigr)\exp\bigl(\bar{C}C_0\bigr)\Bigr) \leq \e_0  \eeq
for some sufficiently small $\e_0$. Then  \eqref{1.1} has a unique
global solution $(\rho,u,\na\Pi)$  with $\r-1 \in C_b([0,\infty);
L^2\cap L^\infty\cap\dot W^{1,r}(\R^2)),$ $u\in
\cC_b([0,\infty);H^1(\R^2))\cap L^1(\R^+; H^2(\R^2)),$ $\p_tu,
\na\Pi\in L^1(\R^+;L^2(\R^2)),$ and \beq \label{thm1.1pql}
\begin{split} \|\na u\|_{L^1(\dot{B}^0_{\infty,1})}\leq
2\bar{C}C_0\bigl(1+\|\r_0\|_{B^{(2/q)_+}_{\infty,\infty}}\bigr)\exp\bigl(\bar{C}C_0\bigr).
\end{split}
\eeq If in addition, $\mu(\cdot)\in W^{2+[s],\infty}(\R^+),$
$\r_0-1\in H^{1+s}(\R^2)$ and $u_0\in H^s(\R^2)$ for some $s>1,$
then  the global solution $\r-1\in \cC([0,\infty);H^{1+s}(\R^2)),$
$u\in \cC([0,\infty);H^s(\R^2))\cap
\wt{L}^1_{loc}(\R^+;\dot{H}^{2+s}(\R^2)).$
 }
\end{thm}

\begin{rmk} Without the assumptions that $\r_0-1\in L^2(\R^2)$ and
$u_0\in H^{-2\d}(\R^2)$ in Theorem \ref{thm1.1}, our proof of
Theorem \ref{thm1.1} ensures that \eqref{1.1} has a unique solution
$(\r,u)$ on a time interval $[0,\frak{T}]$ with $\frak{T}$ being
determined by \beno \frak{T}\geq
\bigl(C(m,M,\|\na\r_0\|_{L^r},\|u_0\|_{H^1})\|\mu(\r_0)-1\|_{L^\infty}\bigr)^{-1}
\eeno and $\r \in L^\infty((0,\frak{T}); L^\infty\cap\dot
W^{1,r}(\R^2)),$ $u\in L^\infty((0,\frak{T}); H^1(\R^2))\cap
L^2((0,\frak{T}); H^2(\R^2)).$
\end{rmk}

\begin{rmk}\label{rmk1.1}
We should point out that the reason why Desjardin \cite{desjardins}
only proved \eqref{upi} for $p\in [2,p^\ast]$ with $p^\ast$ being
determined by \eqref{pstar} is because of the fact that the Riesz
transform $\cR$ maps continuously from $L^p(\R^d)$ to $L^p(\R^d)$
with the operator norm (see Theorem 3.1.1. of \cite{Ch1} for
instance): \beno \|\cR\|_{\cL(L^p\to L^p)}\leq C_0 p \eeno for some
uniform constant $C_0.$ Our main observation used in the proof of
Theorem \ref{thm1.1} is that: Riesz transform $\cR$ maps
continuously from homogeneous Besov spaces $\dot B^s_{p,r}(\R^d)$
(see Definition \ref{def1.1}) to $\dot B^s_{p,r}(\R^d)$ with the
operator norm \beno \|\cR\|_{\cL(\dot B^s_{p,r}\to \dot
B^s_{p,r})}\leq C_0, \eeno which enables us to prove the {\it a
priori} estimate for $\|\na u\|_{L^1_{T}(L^\infty)}.$  This is in
fact the most important ingredient used in the proof of Theorem
\ref{thm1.1}.

The other important ingredient used in the proof of
\eqref{thm1.1pql} is the time decay estimates \eqref{p.2dfg} and
\eqref{p.2a}, which is a slight generalization of the decay
estimates obtained by Huang and Paicu in \cite{HP1}. The proof of
such decay estimates is a direct application of Schonbek's frequency
splitting method as well as the strategy of Wiegner \cite{Wie} to
prove the time decay estimate for classical 2-D Navier-Stokes
system.
\end{rmk}

In the particular case when  $\mu(\r)$ is a positive constant, the
proof of Theorem \ref{thm1.1} yields the following corollary, which
does not require any low frequency assumption on $u_0.$

\begin{col}\label{cst}
{\sl  Let $\alpha\in (0,1)$ and $m, M$ be positive constants. Let
$u_0\in \dot B^0_{2,1}(\R^2)$ be a  solenoidal vector filed and
$\rho_0-1\in \dot B^1_{2,1}\cap\dot
B^{\alpha}_{\infty,\infty}(\R^2)$ with $ m\leq \r_0 \leq M.$ Then
\eqref{1.1} with $\mu(\r)=1$ has a unique global solution $(\rho,u)$
so that $\r-1 \in \cC([0,\infty);\dot B^1_{2,1}\cap \dot
B^{\alpha}_{\infty,\infty}(\R^2)),$ $u\in \cC([0,\infty);\dot
B^0_{2,1}(\R^2))\cap L^1_{loc}(\R^+; \dot B^2_{2,1}(\R^2)).$ }
\end{col}

Another important feature of \eqref{1.1} is the scaling invariant
property, namely,  if $(\r,u)$ is a solution of \eqref{1.1}
associated to the initial data $(\r_0, u_0)$, then
\beq\label{scalingt} \bigl(\r_\la(t,x),u_\la(t,x)\bigr)\eqdefa
\bigl(\r(\la^2t,\la x), \la u(\la^2t,\la x)\bigr)\quad
\bigl(\r_{0,\la}(x),u_{0,\la}(x)\bigr)\eqdefa \bigl(\r_0(\la x), \la
u_0(\la x)\bigr),\eeq $(\r_\la(t,x),u_\la(t,x))$ is a also a
solution of \eqref{1.1} associated with the initial data
$(\r_{0,\la}(x),u_{0,\la}(x)).$ A functional space for the data
$(\r_0,u_0)$ or for the solution $(\r,u)$ is said to be at the
scaling of the equation if its norm is invariant under the
transformation \eqref{scalingt}. In the very interesting paper
\cite{DM1}, Danchin and Mucha proved the global well-posedness of
\eqref{1.1} with $\mu(\r)=\mu>0$ in $d$ space dimensions and with
small initial data in the critical spaces
$\r_0-1\in\dot{B}^{\f{d}p}_{p,1}(\R^d)$ and $u_0\in
\dot{B}^{-1+\f{d}p}_{p,1}(\R^d)$ for $p\in [1,2d).$ In fact, they
\cite{DM1} only require $\r_0-1$ to be small in the multiplier space
of $ \dot{B}^{-1+\f{d}p}_{p,1}(\R^d).$ One may check \cite{DM1} and
the references therein for more details in this direction.

It is easy to check that
$\dB^{\frac{2}{p}}_{p,1}(\R^2)\times\bigl(\dB^{-1+\frac{2}{p}}_{p,1}\cap
L^2(\R^2)\bigr)$ is at the scaling of \eqref{1.1}. When $\rho_0-1$
is small enough in the critical space
$\dB^{\frac{2}{p}}_{p,1}(\R^2),$ we have the following global
well-posedness result for \eqref{1.1}, which in particular improves
the smallness condition \eqref{thmassume} in \cite{HPZ1} to
\eqref{th1.2a} below (with only one exponential), and completes the
uniqueness gap for $p\in (2,4)$ in \cite{HPZ1}.

\begin{thm}\label{thm1.2}
{\sl  Let $1< p<4,$ $\rho_0-1\in \dot B^{\frac{2}{p}}_{p,1}(\R^2)$
and $u_0\in \dot B^{-1+\frac{2}{p}}_{p,1}\cap L^2(\R^2)$ which
satisfy $\dv\,u_0=0$ and \beq\label{th1.2a} \|\r_0-1\|_{\dot
B^{\frac{2}{p}}_{p,1}}\exp\bigl\{C_0(\|u_0\|_{\dot
B^{-1+\frac{2}{p}}_{p,1}}+\|u_0\|_{L^2}^2)\bigr\} \le \e_0 \eeq for
some uniform constant $C_0$ and $\e_0$ being sufficiently small.
Then \eqref{1.1} has a unique global solution $(\rho,u,\na\Pi)$  so
that $\r \in \cC_b([0,\infty);\dot B^{\frac{2}{p}}_{p,1}(\R^2) ),$
$u\in \cC_b([0,\infty);\dot B^{-1+\frac{2}{p}}_{p,1}\cap
L^2(\R^2))\cap L^1(\R^+; \dot B^{1+\frac{2}{p}}_{p,1}(\R^2)),$ and
$\p_tu, \na\Pi\in L^1(\R^+; \dot B^{-1+\frac{2}{p}}_{p,1}(\R^2)).$ }
\end{thm}

Finally, in the case when the initial data is in the almost scaling
invariant spaces  and $\|\r_0-1\|_{L^\infty}$ is sufficiently small,
we have the following global well-posedness result for \eqref{1.1}:

\begin{thm}\label{thm1.3}
{\sl  Let $1< p<4$ and $0<\e<\f4p-1.$ Let $\rho_0-1\in \dot
B^{\frac{2}{p}}_{p,1}\cap \dot B^{\frac{2}{p}+\e}_{p,1}(\R^2)$ and
$u_0\in \dot B^{-1+\frac{2}{p}}_{p,1}\cap \dot
B^{-1+\frac{2}{p}-\e}_{p,1}\cap L^2(\R^2)$ be a solenoidal vector
filed. Then \eqref{1.1} has a unique global solution $(\rho,u,
\na\Pi)$ so that $\r-1 \in \cC_b([0,\infty);\dot
B^{\frac{2}{p}}_{p,1}\cap\dot B^{\frac{2}{p}+\e}_{p,1}(\R^2) ),$
$u\in \cC_b([0,\infty);\dot B^{-1+\frac{2}{p}-\e}_{p,1}\cap\dot
B^{-1+\frac{2}{p}}_{p,1}\cap L^2(\R^2))\cap L^1(\R^+; \dot
B^{1+\frac{2}{p}}_{p,1}\cap \dot B^{1+\frac{2}{p}-\e}_{p,1}(\R^2)),$
and $\p_tu, \na\Pi\in  L^1(\R^+; \dot B^{-1+\frac{2}{p}}_{p,1}\cap
\dot B^{-1+\frac{2}{p}-\e}_{p,1}(\R^2))$ provided that
\beq\label{thm1.3a} \|\r_0-1\|_{L^\infty}\leq\e_0\eeq for some small
enough $\e_0.$ }
\end{thm}

\begin{rmk}\label{rmk1.3}
One may check \eqref{h.10} for the exact size of $\e_0$ in
\eqref{thm1.3a}.
\end{rmk}

\no{\bf Scheme of the proof and organization of the paper.}\
 In the
second section, we shall present the {\it a priori}  time decay
estimate for $\|u(t)\|_{L^2}$ and $\|\na u(t)\|_{L^2}$ which leads
to the crucial estimate for $\|\na u\|_{L^1(\R^+;L^q)}$ for $q$
satisfying $q\d>1.$  Based on these estimates and the observation in
Remark \ref{rmk1.1}, in Section 3, we shall present the {\it a
priori}  $L^1(\R^+;\dot B^1_{\infty,1})$ estimate for velocity
field.  In Section 4, we present a blow-up criterion for smooth
enough solutions of \eqref{1.1}. We then present the proofs of
Theorem \ref{thm1.1} in Section \ref{sect4} and Corollary \ref{cst}
in Section 6. Finally we present the proofs of Theorem \ref{thm1.2}
in Section 7 and Theorem \ref{thm1.3} in Section 8. For the
convenience of the readers, we collect some basic facts on
Littlewood-Paley analysis, which has been used throughout the paper,
in the Appendix \ref{abbendixb}.

Let us complete this section with the notations we are going to use
in this context.

\medbreak \noindent{\bf Notations:} Let $A, B$ be two operators, we
denote $[A,B]=AB-BA,$ the commutator between $A$ and $B$. For
$a\lesssim b$, we mean that there is a uniform constant $C,$ which
may be different on different lines, such that $a\leq Cb$. We shall
denote by $(a|b)$ (or $(a | b)_{L^2}$) the $L^2(\R^2)$ inner product
of $a$ and $b.$

For $X$ a Banach space and $I$ an interval of $\R,$ we denote by
${\mathcal{C}}(I;\,X)$ the set of continuous functions on $I$ with
values in $X,$ and by ${\mathcal{C}}_b(I;\,X)$ the subset of bounded
functions of ${\mathcal{C}}(I;\,X).$ For $q\in[1,+\infty],$ the
notation $L^q(I;\,X)$ stands for the set of measurable functions on
$I$ with values in $X,$ such that $t\longmapsto\|f(t)\|_{X}$ belongs
to $L^q(I).$ For any vector field $v=(v_1,v_2),$ we denote
$d(v)=\f12\bigl(\p_iv_j+\p_jv_i\bigr)_{i,j=1,2}.$ Finally,
$(d_{j})_{j\in\Z}$ (resp. $(c_j)_{j\in\Z}$) will be a generic
element of $\ell^1(\Z)$ (resp. $\ell^2(\Z))$ so that
$\sum_{j\in\Z}d_{j}=1$ (resp. $\sum_{j\in\Z}c_j^2=1).$

\renewcommand{\theequation}{\thesection.\arabic{equation}}
\setcounter{equation}{0}

\section{Basic Estimates}\label{sect2}

In this section, we shall improve the {\it a priori} estimate of
$\|\na u\|_{L^2(\R^+;L^p)}$, which was obtained by Desjardins
\cite{desjardins} in the case of  ${\Bbb{T}}^2,$ to be that of
$\|\na u\|_{L^1(\R^+;L^p)}$ for any $p\in (1/\d, p^\ast]$ with
$p^\ast$ being determined by \eqref{pstar}. This will be one of the
crucial ingredient for us to prove the
$L^1(\R^+;\dot{B}^1_{\infty,1}(\R^2))$ estimate of the velocity
field in Section \ref{sect3}. The main idea to achieve the estimate
of $\|\na u\|_{L^1(\R^+;L^p)}$ is to use the decay estimate for
velocity field in \cite{HP1, Sch, Wie} and the energy  method in
\cite{desjardins}.

\begin{prop}\label{propb.1}
{\sl Let $f(t)$ be a positive smooth function, let $(\r, u)$ be a
smooth enough solution of \eqref{1.1} on $[0,T^\ast)$ for some
positive time $T^\ast.$ Then under the assumption \eqref{thma.1},
one has \beq\label{p.1dfg}
\begin{split}
\f{d}{dt}&\Bigl(f(t)\int_{\R^2}\mu(\r)d:d\,dx\Bigr)+f(t)\int_{\R^2}|\p_tu|^2\,dt'\\
&\leq
4f'(t)\int_{\R^2}\mu(\r)d:d\,dx+C_{m,M}\Bigl(f(t)(1+\|u\|_{L^2}^2)\|\na
u\|_{L^2}^4\Bigr)\quad \mbox{for}\ \ t\in [0,T^\ast),
\end{split}
\eeq where $C_{m,M}$ is a positive constant depending on $m, M$ in
\eqref{thma.1}.}
\end{prop}

\begin{proof} The proof of this proposition basically follows from
that of Theorem 1 in \cite{desjardins}. For completeness, we outline
its proof here. Indeed thanks to \eqref{thma.1}, one has
\beq\label{b.3} m\leq \r(t,x)\leq M\quad \mbox{for}\quad t\in
[0,T^\ast). \eeq In what follows, the uniform constant $C$ always
depends on $m,M$ and sometimes on $\|\mu'\|_{L^\infty}$ also, yet we
neglect the subscripts $m, M$ for simplicity.

By taking $L^2$ inner product of the momentum equation of
\eqref{1.1} with $\p_tu$ and  using integration by parts, we deduce
from the derivation of (29) in \cite{desjardins} that
$$
\begin{aligned}
f&(t)\int_{\R^2}\rho|\partial_tu|^2\,dx+\frac{d}{dt}\Bigl(f(t)\int_{\R^2}\mu(\rho) d:d \,dx\Bigr)\\
&=f'(t)\int_{\R^2}\mu(\rho) d:d \,dx-f(t)\int_{\R^2}\partial_tu\ |\
\bigl(\rho u\cdot\nabla u\bigr)\,dx -f(t)\int_{\R^2}(u\cdot\nabla)u\
|\ \dv\bigl(2\mu(\rho)d\bigr)\,dx
\\&
=f'(t)\int_{\R^2}\mu(\rho) d:d \,dx-2f(t)\int_{\R^2}\partial_tu\ |\
\bigl(\rho u\cdot\nabla u\bigr)\,dx-f(t)\int_{\R^2}\r|u\cdot\na
u|^2\,dx\\
&\quad-f(t)\int_{\R^2}u\cdot\na u\ |\ \na\Pi\,dx,
\end{aligned}
$$ where in the last step we used the momentum equation of \eqref{1.1} so that
$\dv\bigl(2\mu(\rho)d\bigr)=\r\p_tu+\quad$ $\r u\cdot\na u+\na\Pi.$
This gives rise to  \beq\label{b.4}
\begin{aligned}
f(t)&\int_{\R^2}\rho|\partial_tu|^2\,dx+\frac{d}{dt}\Bigl(f(t)\int_{\R^2}\mu(\rho)
d:d \,dx\Bigr)\\
& \leq 2\Bigl( f'(t)\int_{\R^2}\mu(\rho) d:d \,dx+ f(t)\|\sqrt{\rho}
u\cdot\nabla u\|_{L^2}^2
-f(t)\int_{\R^2}u\cdot\nabla u \ |\  \nabla\Pi\,dx\Bigr)\\
&\leq 2f'(t)\int_{\R^2}\mu(\rho) d:d \,dx+C f(t)\Bigl(\|u\|_{L^4}^2
\|\nabla u\|_{L^4}^2+\Bigl|\sum_{i,k=1}^2\int_{\R^2}
\Pi\partial_iu^k\partial_ku^i\,dx\Bigr|\Bigr)
\end{aligned}
\eeq
 To deal with the pressure function $\Pi,$ we get, by taking
space divergence  to the momentum equation of \eqref{1.1}, that
\beq\label{b.4a}
\begin{aligned}
\Pi& =(-\Delta)^{-1}\dv\bigl(\rho\partial_tu+\rho u\cdot\nabla
u\bigr) -(-\Delta)^{-1}\dv\otimes\dv\bigl(2\mu(\rho)d\bigr),
\end{aligned}
\eeq from which, we deduce
$$
\begin{aligned}
\Bigl|\sum_{i,k=1}^2&\int_{\R^2}
\Pi\partial_iu^k\partial_ku^i\,dx\Bigr| \lesssim \|\nabla
u\|_{L^2}\|\nabla u\|_{L^4}^2
\\&\qquad\qquad
+\|(-\Delta)^{-1}\dv\bigl(\rho\partial_tu
+\rho(u\cdot\nabla)u\bigr)\|_{BMO}
\bigl\|\sum_{i,k=1}^2\partial_iu^k\partial_ku^i\bigr\|_{\mathscr{H}^1},
\end{aligned}
$$ where $\|f\|_{\mathscr{H}^1}$ denotes the Hardy norm of $f.$
Yet as  $\dv\,u=0,$  it follows from \cite{Coifman} that \beno
\bigl\|\sum_{i,k=1}^2\partial_iu^k\partial_ku^i\bigr\|_{\mathscr{H}^1}\lesssim
\|\nabla u\|_{L^2}^2,\eeno and $\|f\|_{BMO(\R^2)}\lesssim\|\nabla
f\|_{L^2(\R^2)},$ we obtain
$$\begin{aligned}
\Bigl|\sum_{i,k=1}^2\int_{\R^2}
\Pi\partial_iu^k\partial_ku^i\,dx\Bigr| \lesssim& \|\nabla
u\|_{L^2}\|\nabla u\|_{L^4}^2
 +\|\rho\partial_tu +\rho(u\cdot\nabla)u\|_{L^2}
\|\nabla u\|_{L^2}^2,
\end{aligned}
$$
which along with $\|u\|_{L^4}^2\lesssim \|u\|_{L^2}\|\na u\|_{L^2}$
and \eqref{b.4} ensures that \beq\label{b.5}
\begin{split}
f(t)&\int_{\R^2}\rho|\partial_tu|^2\,dx+\frac{d}{dt}\Bigl(f(t)\int_{\R^2}\mu(\rho)
d:d \,dx\Bigr)\\
& \leq 3f'(t)\int_{\R^2}\mu(\rho) d:d \,dx+ C\Bigl(f(t) \|\nabla
u\|_{L^2}^4+f(t)(1+\|u\|_{L^2})\|\nabla u\|_{L^2} \|\nabla
u\|_{L^4}^2\Bigr).
\end{split} \eeq To handle $\|\nabla u\|_{L^4},$ we write
\beq\label{b.5af} \na
u=\na(-\D)^{-1}\mathbb{P}\dv\bigl(2(\mu(\rho)-1)d\bigr)
-\na(-\D)^{-1}\mathbb{P}\dv\bigl(2\mu(\rho)d\bigr), \eeq which
together with the following interpolation inequality from \cite{CX}
\beq \|f\|_{L^{r}(\R^2)}\leq
C\sqrt{r}\|f\|_{L^2(\R^2)}^{\frac2r}\|\na
f\|_{L^2(\R^2)}^{1-\frac2r},\qquad 2\leq r<\infty, \label{f.2} \eeq
ensures that for any $p\in[2,\infty)$
$$
\|\nabla u\|_{L^p} \le C_0 p\|\mu(\rho_0)-1\|_{L^\infty}\|\nabla
u\|_{L^p} +C\sqrt{p}\|\nabla u\|_{L^2}^{\frac{2}{p}}
\|\mathbb{P}\dv\bigl(2\mu(\rho)d\bigr)\|_{L^2}^{1-\frac{2}{p}}
$$
with $C_0>0$ being a universal constant. Taking $\e_0$ sufficiently
small in \eqref{thma.1}, we obtain for $2\leq p\leq
p^*=\frac{1}{2C_0\|\mu(\rho_0)-1\|_{L^\infty}}$ that \beq\label{b.7}
\begin{aligned}
\|\nabla u\|_{L^p} &\le C\sqrt{p}\|\nabla u\|_{L^2}^{\frac{2}{p}}
\|\rho\partial_tu+\rho(u\cdot\nabla)u\|_{L^2}^{1-\frac{2}{p}}
\\&
\le C\sqrt{p}\|\nabla u\|_{L^2}^{\frac{2}{p}}
\big(\|\partial_tu\|_{L^2}^{1-\frac{2}{p}}
+\|u\|_{L^4}^{1-\frac{2}{p}}\|\nabla u\|_{L^4}^{1-\frac{2}{p}}\big).
\end{aligned}
\eeq In particular taking $p=4$ in \eqref{b.7} results in
\beq\label{b.10} \|\nabla u\|_{L^4}^2 \leq C\bigl( \|\nabla
u\|_{L^2}\|\partial_tu\|_{L^2}+\|u\|_{L^2}\|\nabla
u\|_{L^2}^3\bigr). \eeq Substituting the above inequality into
\eqref{b.5}, we obtain \eqref{p.1dfg}. This completes the proof of
the proposition.
\end{proof}

\begin{col}\label{colb.1}
{\sl Under the assumption of Proposition \ref{propb.1}, we have
\beq\label{thma.2} \begin{split} &\|u\|_{L^\infty_t(L^2)}^2+\|\na
u\|_{L^2_t(L^2)}^2\leq C \|u_0\|_{L^2}^2,\\
& \|\w{t'}^{\f12}\na
u\|_{L^\infty_{t}(L^2)}^2+\|\w{t'}^{\f12}\p_tu\|_{L^2_t(L^2)}^2 \leq
C\|\na u_0\|_{L^2}^2\exp\bigl(C\|u_0\|_{L^2}^4\bigr),
\end{split}\eeq
 for all
$t\in [0, T^\ast)$ and where $\w{t}\eqdefa e+t$.}
\end{col}
\begin{proof} We first get, by using
standard energy estimate to \eqref{1.1}, that \beq\label{b.3a}
\frac12\frac{d}{dt}\int_{\R^2}\r|u|^2\,dx+\int_{\R^2}\mu(\r)d:d\,dx=0,
\eeq which implies the first inequality of \eqref{thma.2}.

 Whereas taking $f(t')=\w{t'}$ in \eqref{p.1dfg} and integrating the resulting inequality over $[0,t],$ we
 obtain
$$
\int_{0}^t\int_{\R^2}\w{t'}|\partial_tu|^2\,dx\,dt'+\|\w{t'}^{\f12}\nabla
u\|_{L^\infty_t(L^2)}^2 \leq C\bigl( \|\nabla
u_0\|_{L^\infty_t(L^2)}^2+(1+\|u_0\|_{L^2}^2)\int_{0}^t\w{t'}\|\nabla
u\|_{L^2}^4\,dt'\bigr),
$$ Applying  Gronwall's inequality and using the first inequality of  \eqref{thma.2}  gives rise
to \beno \begin{split}
\int_{0}^t\int_{\R^2}\w{t'}|\partial_tu|^2\,dx\,dt'
+\|\w{t'}^{\f12}\nabla u\|_{L^\infty_t(L^2)}^2 \leq& C\|\nabla
u_0\|_{L^2}^2
\exp\bigl\{C(1+\|u_0\|_{L^2}^2)\|\nabla u\|_{L^2_t(L^2)}^2\bigr\}\\
\leq& C \|\nabla u_0\|_{L^2}^2\exp\bigl(C\|u_0\|_{L^2}^4\bigr).
\end{split}\eeno
This completes the proof of \eqref{thma.2}. \end{proof}

\begin{prop}\label{propb.2}
{\sl With the additional assumption that $\r_0-1\in L^2(\R^2),$
$u_0\in H^{-2\d}(\R^2)$ for $\d\in (0,\f12),$ then under the
assumption of Proposition \ref{propb.1}, we have \beq\label{p.2dfg}
\begin{split}
&\|\w{t'}^{\d}u\|_{L^2}+\|\w{t'}^{\d_-}\na u\|_{L^2_t(L^2)}\leq
C\sqrt{C_0}\exp\bigl(CC_0\bigr),
\end{split}
\eeq and \beq\label{p.2a}
\begin{split} &\|\w{t'}^{(\f{1}2+\d)_-}\na
u\|_{L^\infty_t(L^2)}+\|\w{t'}^{(\f{1}2+\d)_-}u_t\|_{L^2_t(L^2)}\leq
C\sqrt{C_0}\exp\bigl(CC_0\bigr),
\end{split}
\eeq for any $t\in [0,T^\ast)$ and $C_0$ being determined by
\eqref{p.2qwt}.
 } \end{prop}

\begin{rmk} Large time decay estimates for $\|u(t)\|_{L^2}$ and $\|\na
u(t)\|_{L^2}$ were obtained by Gui and the authors in \cite{AGZ1}
for 3-D inhomogeneous Navier-Stokes system with constant viscosity.
Gui and the second author proved the time decay estimate for
$\|u(t)\|_{L^2}$ in \eqref{p.2dfg} for 2-D inhomogeneous
Navier-Sttokes system with variable density in \cite{GZ}. Similar
time decay estimates as \eqref{p.2dfg} and \eqref{p.2a} were
obtained by Huang and Paicu in \cite{HP1}. Note that for $p\in
[1,2)$ and $\d=\f1p-\f12,$ $L^p(\R^2)$ can be continuously imbedded
into $H^{-2\d}(\R^2)$, the decay estimates \eqref{p.2dfg} and
\eqref{p.2a} are slightly general than that in \cite{HP1}, where the
authors require the low frequency assumption for $u_0$ that $u_0\in
L^p(\R^2)$ for $p\in [1,2).$ For completeness, here we shall outline
the proof which basically follows from the corresponding argument in
\cite{Wie} for the classical 2-D Navier-Stokes system.
\end{rmk}

According to \cite{Wie} for classical Navier-Stokes system, the key
ingredient used in the proof of the decay estimate for
$\|u(t)\|_{L^2}$ in \eqref{p.2dfg} is the following Lemma:

\begin{lem}\label{lemb.1}
{\sl Under the assumption of Proposition \ref{propb.2}, we have \beq
\label{p.3qpo} \|u(t)\|_{L^2}\leq C\sqrt{C_0} \frac1{\ln\w{t}}\quad
\mbox{ for any}\quad t\in [0,T^\ast).\eeq}
\end{lem}

\begin{proof} Following the proofs of Theorem 3.1
of \cite{HP1} and Lemma 4.4 of \cite{GZ}, we first deduce from
\eqref{b.3a} that \beq\label{p.3a}
\f{d}{dt}\|\sqrt{\r}u\|_{L^2}^2+2\underline{\mu}\|\na
u\|_{L^2}^2\leq 0. \eeq Applying Schonbek's strategy in \cite{Sch},
by splitting the phase space $\R^2$ into two time-dependent domain:
$\R^2=S(t)\cup S(t)^c,$ where $S(t)\eqdefa \bigl\{\xi:\ |\xi|\leq
\sqrt{\f{M}{2\underline{\mu}}}g(t)\ \bigr\}$ for some $g(t),$ which
will be chosen later on. Then we deduce from \eqref{p.3a} that
\beq\label{p.4thy}
\f{d}{dt}\|\sqrt{\r}u(t)\|_{L^2}^2+g^2(t)\|\sqrt{\r}u(t)\|_{L^2}^2\leq
Mg^2(t)\int_{S(t)}|\hat{u}(t,\xi)|^2\,d\xi. \eeq To deal with the
low frequency part of $u$ on the right hand side of \eqref{p.4thy},
we write \beno
\begin{split}
u(t)=e^{t\D}u_0+\int_0^te^{(t-t')\D}\Bbb{P}\Bigl(\dive\bigl((\mu(\r)-1)d(u)\bigr)+(1-\r)(u_t+u\cdot\na
u)-u\cdot\na u\Bigr)(s)\,dt'.
\end{split}
\eeno Taking Fourier transform with respect to $x$ variables gives
rise to \beno
\begin{split}
|\hat{u}(t,\xi)|\lesssim &
e^{-t|\xi|^2}|\hat{u}_0(\xi)|+\int_0^te^{-(t-t')|\xi|^2}\Bigl(|\xi|\bigl(\bigl|\cF_x[(\mu(\r)-1)d(u)]\bigr|+\bigl|
\cF_x(u\otimes u)\bigr|\bigr)\\
&+\bigl|\cF_x[(1-\r)(u_t+u\cdot\na u)]\bigr|\Bigr)(t')\,dt',
\end{split}
\eeno so that \beq\label{p.5yui}
\begin{split}
\int_{S(t)}&|\hat{u}(t,\xi)|^2\,d\xi\lesssim
\int_{S(t)}e^{-t|\xi|^2}|\hat{u}_0(\xi)|^2\,d\xi+g^4(t)\Bigl(\int_0^t\bigl(\bigl\|\cF_x[(\mu(\r)-1)d(u)]\bigr\|_{L^\infty_\xi}\\
&+\bigl\| \cF_x(u\otimes
u)\bigr\|_{L^\infty_\xi}\bigr)\,dt'\Bigr)^2+g^2(t)\Bigl(\int_0^t\bigl\|\cF_x[(1-\r)(u_t+u\cdot\na
u)]\bigr\|_{L^\infty_\xi}\,dt'\Bigr)^2.
\end{split}
\eeq It is easy to observe that \beno
\int_{S(t)}e^{-t|\xi|^2}|\hat{u}_0(\xi)|^2\,d\xi\leq
\w{t}^{-2\d}\|u_0\|_{H^{-2\d}}^2, \eeno and \beno
\begin{split}
\Bigl(\int_0^t\bigl(&\bigl\|\cF_x[(\mu(\r)-1)d(u)]\bigr\|_{L^\infty_\xi}+\bigl\|
\cF_x(u\otimes u)\bigr\|_{L^\infty_\xi}\bigr)\,dt'\Bigr)^2\\
&\lesssim \Bigl(\int_0^t\bigl(\|(\mu(\r)-1)\cM(u)\|_{L^1}+\|
u\otimes u\|_{L^1}\bigr)\,dt'\Bigr)^2\\
&\lesssim \|\mu(\r)-1\|_{L^\infty_t(L^2)}^2\|\na
u\|_{L^2_t(L^2)}^2t+\Bigl(\int_0^t\|u(t')\|_{L^2}^2\,dt'\Bigr)^2,
\end{split}
\eeno Finally thanks to \eqref{b.10} and \eqref{thma.2},  we have
\beno \Bigl(\int_0^t\|u_t(t')\|_{L^2}\,dt'\Bigr)^2\leq
C\ln\w{t}\int_0^t\w{t'}\|u_t(t')\|_{L^2}^2\,dt'\leq C\|\na
u_0\|_{L^2}^2\exp\bigl(C\|u_0\|_{L^2}^4\bigr)\ln\w{t}, \eeno and
\beno
\begin{split}
\int_0^t\|u\|_{L^4}\|\na u\|_{L^4}\,dt'\lesssim &
\int_0^t\bigl(\|u\|_{L^2}^{\f12}\|\na
u\|_{L^2}\|u_t\|_{L^2}^{\f12}+\|u\|_{L^2}\|\na
u\|_{L^2}^2\bigr)\,dt'\\
\lesssim &\|u\|_{L^\infty_t(L^2)}^{\f12}\|\na
u\|_{L^2_t(L^2)}\|\w{t}^{\f12}u_t\|_{L^2}^{\f12}\ln^{\f14}\w{t}+\|u\|_{L^\infty_t(L^2)}\|\na
u\|_{L^2_t(L^2)}^2\\
\leq & C \|\na
u_0\|_{L^2}^{\f12}\exp\bigl(C\|u_0\|_{L^2}^4\bigr)\ln^{\f14}\w{t},
\end{split}
\eeno which leads to \beno
\begin{split}
\Bigl(\int_0^t\bigl\|&\cF_x[(1-\r)(u_t+u\cdot\na
u)]\bigr\|_{L^\infty_\xi}(t')\,dt'\Bigr)^2\\ \leq
&\|(1-\r)\|_{L^\infty_t(L^2)}^2\Bigl[\Bigl(\int_0^t\|u_t(t')\|_{L^2}\,dt'\Bigr)^2+\Bigl(\int_0^t\|u\|_{L^4}\|\na
u\|_{L^4}\,dt'\Bigr)^2\Bigr]\\
\leq &C\|\r_0-1\|_{L^2}^2\|\na
u_0\|_{L^2}^2\exp\bigl(C\|u_0\|_{L^2}^4\bigr)\ln\w{t}.
\end{split}
\eeno Resuming the above estimates into \eqref{p.5yui} and then
using \eqref{p.4thy}, we obtain \beq\label{p.6qwp}
\begin{split}
\f{d}{dt}\|\sqrt{\r}u(t)\|_{L^2}^2+g^2(t)\|\sqrt{\r}u(t)\|_{L^2}^2\leq&
M g^6(t)\Bigl(\int_0^t\|u(t')\|_{L^2}^2\,dt'\Bigr)^2\\
&+CC_0\Bigl(g^2(t)\w{t}^{-2\d}+g^6(t)\w{t}+g^4(t)\ln\w{t}\Bigr),
\end{split}
\eeq for $C_0$ given by \eqref{p.2qwt}. Taking
$g^2(t)=\f3{\w{t}\ln\w{t}}$ in the above inequality and then
integrating the resulting inequality over $[0,t]$  resulting
\eqref{p.3qpo}.
\end{proof}

We now turn to the proof of Proposition \ref{propb.2}.

\begin{proof}[Proof of Proposition \ref{propb.2}] With Lemma
\ref{lemb.1} and \eqref{p.6qwp}, the decay estimate of
$\|u(t)\|_{L^2}$ in \eqref{p.2dfg} follows by an standard argument
as \cite{Wie} for the classical 2-D Navier-Stokes system (One may
check page 310-311 of \cite{Wie} for details). Whereas multiplying
\eqref{p.3a} by $\w{t}^{(2\d)_-}$ and then integrating the resulting
inequality over $[0,t],$ we obtain \beq \label{p.7}
\begin{split}
\|\w{t}^{\d_-}u\|_{L^2}^2+2\underline{\mu}\|\w{t'}^{\d_-}\na
u\|_{L^2_t(L^2)}^2\leq & C\bigl(
\|u_0\|_{L^2}^2+\int_0^t\w{t'}^{(2\d-1)_-}\|u(t')\|_{L^2}^2\,dt'\bigr)\\
\leq &CC_0\exp\bigl(CC_0),
\end{split}
\eeq for $C_0$ given by \eqref{p.2qwt}. This proves \eqref{p.2dfg}.

On the other hand, taking $f(t)=\w{t}^{(1+2\d)_-}$ in
\eqref{p.1dfg}, and then using \eqref{p.7} and Gronwall's
inequality, we obtain \eqref{p.2a}. This completes the proof of the
Proposition.
\end{proof}

\no{\bf Notation:} In all that follows, for $C_0$ given by
\eqref{p.2qwt},  we already denote \beq\label{p.7a} C_1\eqdefa
C\sqrt{C_0}\exp\bigl(CC_0\bigr).\eeq

We now present the key estimate in this section:

\begin{prop}\label{propb.3}
{\sl Under the assumptions of Proposition \ref{propb.2}, for $p\in
[2, p^\ast]$ with $p^\ast$ being determined by \eqref{pstar}, we
have for any $t\in [0, T^\ast)$ \beq \label{p.8}
\|\w{t'}^{(\f12+\d-\f1p)_-}\nabla u\|_{L^2_t(L^p)} \leq
\sqrt{p}C_1^{2-\f2p}.\eeq }
\end{prop}

\begin{proof} We first, get by resuming \eqref{b.10} into
\eqref{b.7},
that \beno
\begin{split}
\|\na u\|_{L^p}\leq C\sqrt{p}\Bigl(\|\na
u\|_{L^2}^{\f2p}\|u_t\|_{L^2}^{1-\f2p}+\|u\|_{L^2}^{\f12(1-\f2p)}\|\na
u\|_{L^2}\|u_t\|_{L^2}^{\f12(1-\f2p)}+\|u\|_{L^2}^{1-\f2p}\|\na
u\|_{L^2}^{2-\f2p}\Bigr).
\end{split}
\eeno Notice that $p\geq 2,$ multiplying $\w{t'}^{(\f12+\d-\f1p)_-}$
to the above inequality and then taking $L^2(0,t)$ norm of the
resulting inequality, we obtain \beno
\begin{split}
&\|\w{t'}^{(\f12+\d-\f1p)_-}\na u\|_{L^2_t(L^p)}\leq
C\sqrt{p}\Bigl(\|\w{t'}^{\d_-}\na
u\|_{L^2_t(L^2)}^{\f2p}\|\w{t'}^{(\f12+\d)_-}u_t\|_{L^2_t(L^2)}^{1-\f2p}\\
&\qquad\qquad+\|\w{t'}^\d
u\|_{L^\infty_t(L^2)}^{\f12-\f1p}\|\w{t'}^{(\f12+\d)_-}\na
u\|_{L^\infty_t(L^2)}^{\f12-\f1p}\|\w{t'}^{\d_-}\na
u\|_{L^2_t(L^2)}^{\f12+\f1p}\|\w{t'}^{(\f12+\d)_-}u_t\|_{L^2_t(L^2)}^{\f12-\f1p}\\
&\qquad\qquad+\|
u\|_{L^\infty_t(L^2)}^{1-\f2p}\|\w{t'}^{(\f12+\d)_-}\na
u\|_{L^\infty_t(L^2)}^{1-\f2p}\|\w{t'}^{\d_-}\na
u\|_{L^2_t(L^2)}\Bigr).
\end{split}
\eeno Then  we get, by resuming \eqref{p.2dfg} and \eqref{p.2a} into
the above inequality, that \beno \|\w{t'}^{(\f12+\d-\f1p)_-}\nabla
u\|_{L^2_t(L^p)} \leq
\sqrt{p}C_1^{2-\f2p}\bigl(1+\|u_0\|_{L^2}^{1-\f2p}\bigr), \eeno
which together with \eqref{p.2qwt} and \eqref{p.7a} leads to
\eqref{p.8}.
\end{proof}

\renewcommand{\theequation}{\thesection.\arabic{equation}}
\setcounter{equation}{0}
\section{The $L^1(\R^+;\dot B^1_{\infty,1})$ estimate for the  velocity
field}\label{sect3}

The goal of the this section is to present the {\it a priori}
 $L^1(\R^+;\dot B^1_{\infty,1})$ estimate for the  velocity field, which
 is the most important ingredient used in the proof of Theorem
 \ref{thm1.1}.

\begin{lem}\label{lema.1.1}
{\sl Let $q\in (1/\d,p^\ast]$ with  $p^\ast$ being determined by
\eqref{pstar} and $\e>0$ such that $\frac{2}{q}+\e<1.$ Let
$(\r,u,\na\Pi)$ be a  smooth enough solution of \eqref{1.1} on
$[0,T^\ast).$ Then under the assumptions Proposition \ref{propb.2},
one has \beq \label{a.1rf} \|u\|_{\widetilde L^1_t(\dot
B^{1+\frac{2}{q}+\varepsilon}_{q,\infty})} \leq\|u_0\|_{\dot
B^{-1+\frac{2}{q}+\varepsilon}_{q,\infty}}
+CC_1^2\bigl(1+\|\rho\|_{L^\infty_t(B^{\frac{2}{q}+\varepsilon}_{\infty,\infty})}\bigr)\quad
\mbox{for \ any }\quad t<T^\ast,\eeq where the norm
$\|u\|_{\widetilde L^1_t(\dot
B^{1+\frac{2}{q}+\varepsilon}_{q,\infty})} $ is given by Definition
\ref{chaleur+} and the constant $C_1$ by \eqref{p.7a}. }
\end{lem}

\begin{proof} Let $\mathbb{P}\eqdefa I-\na(\D)^{-1}\dv$ be Leray
projection operator. We get, by first dividing the momentum equation
of \eqref{1.1} by $\rho$ and then applying the resulting equation by
$\mathbb{P},$ that
$$
\partial_t u+\mathbb{P}\bigl\{u\cdot\nabla u\bigr\}
-\mathbb{P}\bigl\{\frac{1}{\rho}\bigl(\dv(2\mu(\rho)d)-\nabla\Pi\bigr)\bigr\}=0.
$$
Applying $\dot\Delta_j$ to the above equation and using a standard
commutator's process yields \beq\label{a.2}
\begin{aligned}
\rho&\partial_t\dot\Delta_ju+\rho
u\cdot\nabla\dot\Delta_ju-\D\dot\D_ju-
2\dv\bigl((\mu(\rho)-1)\mathbb{P}d(\dot\Delta_ju)\bigr)\\
&=-\rho[\dot\Delta_j\mathbb{P}; u\cdot\nabla]u
+\rho[\dot\Delta_j\mathbb{P};
\frac{1}{\rho}]\bigl(\dv(2\mu(\rho)d)-\nabla\Pi\bigr)
+2\dv[\dot\Delta_j\mathbb{P}; \mu(\rho)]d.
\end{aligned}
\eeq Throughout this paper, we always denote
$d(v)\eqdefa\bigl(\f12(\p_iv_j+\p_jv_i)\bigr)_{2\times2},$ and
abbreviate $d(u)$ as $d.$

Taking $L^2$ inner product of \eqref{a.2} with
$|\dot\Delta_ju|^{q-2}\dot\Delta_ju,$ we obtain \beq\label{a.2da}
\begin{aligned}
&\frac{1}{q}\frac{d}{dt}\int_{\R^2}\rho|\dot\Delta_ju|^q\,dx-
\int_{\R^2}\D\Delta_ju\ |\ |\dot\Delta_ju|^{q-2}\dot\Delta_j u\,dx
\\&
\leq \|\dot\Delta_ju\|_{L^q}^{q-1}
\Bigl\{C(q-1)2^j\|(\mu(\rho)-1)\mathbb{P} d(\dot\Delta_ju)\|_{L^q}
+\|\rho[\dot\Delta_j\mathbb{P};u\cdot\nabla]u\|_{L^q}\\
& \quad+\|\rho[\dot\Delta_j\mathbb{P};
\frac{1}{\rho}]\bigl(\dv(2\mu(\rho)d) -\nabla\Pi\bigr)\|_{L^q}
+c(q-1)2^j\|[\dot\Delta_j\mathbb{P},\mu(\rho)](2d)\|_{L^q}\Bigr\}.
\end{aligned}
\eeq However as $\dv\, u=0,$ one gets, by using integration by parts
and  Lemma A.5 of \cite{DAN}, that \beno \begin{split} -
\int_{\R^2}\D\dot\Delta_ju\ |\ |\dot\Delta_ju|^{q-2}\dot\Delta_j
u\,dx =&\int_{\R^2}|\dot\D_j\na
u|^2|\dot\D_ju|^{q-2}\,dx\\&+(q-2)\int_{\R^2}|\dot\D_ju|^{q-2}\bigl(\na|\dot\D_ju|\bigr)^2\,dx\geq
c2^{2j}\|\dot\Delta_ju\|_{L^q}^q ,
\end{split}
\eeno for some positive constant $c.$

Whereas it follows from Lemma \ref{lem2.1} that \beno
\|(\mu(\rho)-1)\mathbb{P} d(\dot\Delta_ju)\|_{L^q}\lesssim 2^j
\|\mu(\rho)-1\|_{L^\infty}\|\dot\D_ju\|_{L^q}\lesssim 2^j
\|\mu(\rho_0)-1\|_{L^\infty}\|\dot\D_ju\|_{L^q}. \eeno Therefore
taking $\e_0$ sufficiently small in \eqref{thma.1} and using
\eqref{b.3}, we deduce from \eqref{a.2da} that \beno
\begin{aligned}
\frac{d}{dt}\|\rho^{\frac1q}\dot\Delta_ju\|_{L^q}&
+c2^{2j}\|\rho^{\frac1q}\dot\Delta_ju\|_{L^q} \lesssim
\|\rho[\dot\Delta_j\mathbb{P};
u\cdot\nabla]u\|_{L^q}\\
& \quad+\|\rho[\dot\Delta_j\mathbb{P};
\frac{1}{\rho}]\bigl(\dv(2\mu(\rho)d) -\nabla\Pi\bigr)\|_{L^q}
+2^j\|[\dot\Delta_j\mathbb{P},\mu(\rho)](2d)\|_{L^q},
\end{aligned}
\eeno  which gives rise to \beno
\begin{aligned}
\|\rho^{\frac{1}{q}}\dot\Delta_ju(t)\|_{L^q} \lesssim&
e^{-c2^{2j}t}\|\rho_0^{\frac{1}{q}}\dot\Delta_ju_0\|_{L^q}
+\int_0^te^{-c2^{2j}(t-t')}\Bigl\{\|\rho[\dot\Delta_j\mathbb{P};
u\cdot\nabla]u\|_{L^q}\\
&+\|\rho[\dot\Delta_j\mathbb{P};
\frac{1}{\rho}]\bigl(\dv(2\mu(\rho)d) -\nabla\Pi\bigr)\|_{L^q}
+2^j\|[\dot\Delta_j\mathbb{P},\mu(\rho)](2d)\|_{L^q}\Bigr\}\,dt'.
\end{aligned}
\eeno As a consequence,  thanks to \eqref{b.3} and Definition
\ref{chaleur+}, we conclude, for $q\in (2,p^\ast],$ that
\beq\label{a.3}
\begin{aligned}
\|u\|_{\widetilde L^1_t(\dot
B^{1+\frac{2}{q}+\varepsilon}_{q,\infty})} \lesssim &\|u_0\|_{\dot
B^{-1+\frac{2}{q}+\varepsilon}_{q,\infty}}
+\sup_j2^{(-1+\frac{2}{q}+\varepsilon)j}\|[\dot\Delta_j\mathbb{P};
u\cdot\nabla]u\|_{L^1_t(L^q)}
\\&
+\sup_j2^{(-1+\frac{2}{q}+\varepsilon)j}
\|[\dot\Delta_j\mathbb{P};\frac{1}{\rho}]\bigl(\dv(2\mu(\rho)d)
-\nabla\Pi\bigr)\|_{L^1_t(L^q)}\\
&+\sup_j2^{(\frac{2}{q}+\varepsilon)j}\|[\dot\Delta_j\mathbb{P};
\mu(\rho)]d\|_{L^1_t(L^q)}.
\end{aligned}
\eeq In what follows, we shall handle term by term the right-hand
side of \eqref{a.3}. Firstly appplying Bony's decomposition
\eqref{bony}, one has \beno [\dot\Delta_j\mathbb{P};
u\cdot\nabla]u=[\dot\Delta_j\mathbb{P};
T_u\cdot\nabla]u+\dot\D_j\mathbb{P}R(u,\na u)-R(u,\na\dot\D_ju).
\eeno Applying Lemma 1 of \cite{Plan} gives \beno \begin{split}
2^{(-1+\frac{2}{q}+\varepsilon)j} \|[\dot\Delta_j\mathbb{P};
T_u\cdot\nabla]u\|_{L^q} \lesssim&
2^{(-1+\frac{2}{q}+\varepsilon)j}\sum_{|j-\ell|\leq 4} \|\na \dot
S_{\ell-1}u\|_{L^\infty}\|\dot\D_\ell u\|_{L^q}\\
 \lesssim& \|\nabla
u\|_{L^q} \|u\|_{\dot H^{\frac{2}{q}+\varepsilon}} \lesssim \|\nabla
u\|_{L^q} \|\nabla u\|_{L^2}^{\frac{2}{q}+\varepsilon}
\|u\|_{L^2}^{1-\frac{2}{q}-\varepsilon}.
\end{split}
\eeno Whereas applying Lemma \ref{lem2.1}, one has
\beno\begin{split} 2^{(-1+\frac{2}{q}+\varepsilon)j}
\|\dot\D_j\mathbb{P}R(u,\na u)\|_{L^q} \lesssim
2^{(\frac{2}{q}+\varepsilon)j}\sum_{\ell\geq j-3}\|\dot\D_\ell
u\|_{L^2}\|\dot S_{\ell+2}\na u\|_{L^q} \lesssim \|u\|_{\dot
H^{\frac{2}{q}+\varepsilon}}\|\na u\|_{L^q}.
\end{split} \eeno
The same estimate holds for $R(u,\na\dot\D_ju).$ which together with
\eqref{thma.2} and  Proposition \ref{propb.3} implies
\beq\label{a.4}
\begin{split}
\sup_j2^{(-1+\frac{2}{q}+\varepsilon)j}&
\|[\dot\Delta_j\mathbb{P},u\cdot\nabla]u\|_{L^1_t(L^q)} \leq C
\|u\|_{L^\infty_t(L^2)}^{1-\frac{2}{q}-\varepsilon} \|\nabla
u\|_{L^\infty_t(L^2)}^{\frac{2}{q}+\varepsilon}\|\nabla
u\|_{L^1_t(L^q)} \\
&\leq C\bigl(\|u_0\|_{L^2}+\|\na
u_0\|_{L^2}\exp\bigl(C\|u_0\|_{L^2}^4\bigr)\bigr)\|\w{t'}^{(\f12+\d-\f1q)_-}\na
u\|_{L^2_t(L^q)} \leq CC_1^{2},
\end{split}
 \eeq
where in the last step, we used the assumption that $q\d>1$ so that
$\|\w{t'}^{-(\f12+\d-\f1q)_-}\|_{L^2_t}\leq C.$

Exactly along the same line to the proof of \eqref{a.4}, we  get, by
applying Bony's decomposition \eqref{bony}, that \beno
[\dot\Delta_j\mathbb{P},\frac{1}{\rho}]f
=[\dot\Delta_j\mathbb{P},T_{\frac{1}{\rho}}]f
+\dot\D_j\mathbb{P}R(\frac1\r,f)
-R(\frac1\r,\dot\Delta_j\mathbb{P}f). \eeno It follows from Lemma 1
of \cite{Plan}  that
$$
\begin{aligned}
2^{(-1+\frac{2}{q}+\varepsilon)j}
\|[\dot\Delta_j\mathbb{P},T_{\frac{1}{\rho}}]f\|_{L^q} &\lesssim
2^{j\varepsilon}
\|[\dot\Delta_j\mathbb{P},T_{\frac{1}{\rho}}]f\|_{L^2}
\\&\lesssim
\|\na\frac1\r\|_{\dot B^{-1+\varepsilon}_{\infty,\infty}}\|f\|_{L^2}
\lesssim (1+\|\r\|_{L^\infty}) \|\r\|_{
B^{\varepsilon}_{\infty,\infty}}\|f\|_{L^2},
\end{aligned}
$$
and applying Lemma \ref{lem2.1}  leads to  \beno
\begin{split}
2^{(-1+\frac{2}{q}+\varepsilon)j}
\|\dot\D_j\mathbb{P}R(\frac1\r,f)\|_{L^q} &\lesssim2^{j\varepsilon}
\|\dot\D_j\mathbb{P}R(\frac1\r,f)\|_{L^2}
\\&\lesssim2^{j\varepsilon}\sum_{\ell\geq j-3}
\|\dot\D_\ell(\frac1\r)\|_{L^\infty}\|\dot S_{\ell+2}f\|_{L^2}
\lesssim  (1+\|\r\|_{L^\infty})
\|\r\|_{B^{\varepsilon}_{\infty,\infty}}\|f\|_{L^2}.
\end{split}
\eeno The same estimate holds for
$R(\frac1\r,\dot\Delta_j\mathbb{P}f),$ so we obtain \beno
2^{(-1+\frac{2}{q}+\varepsilon)j}\bigl\|[\dot\Delta_j\mathbb{P},\frac{1}{\rho}]f\bigr\|_{L^1_t(L^q)}\leq
C\|\r\|_{L^\infty_t(
B^{\varepsilon}_{\infty,\infty})}\|f\|_{L^1_t(L^2)}, \eeno from
which  and
$\dv(2\mu(\rho)d)-\nabla\Pi=\r\partial_tu+\r(u\cdot\nabla)u$, we
deduce that \beno
\begin{split}
\sup_j2^{(-1+\frac{2}{q}+\varepsilon)j}&
\|[\dot\Delta_j\mathbb{P},\frac{1}{\rho}]\bigl(\dv(2\mu(\rho)d)
-\nabla\Pi\bigr)\|_{L^1_t(L^q)} \\
&\leq C \|\r\|_{L^\infty_t( B^{\varepsilon}_{\infty,\infty})}
\|(\dv(2\mu(\rho)d)-\nabla\Pi)\|_{L^1_t(L^2)}
\\&
\leq C
\|\rho\|_{L^\infty_t(B^{\varepsilon}_{\infty,\infty})}\bigl(\|\p_tu\|_{L^1_t(L^2)}+\int_0^t\|u\|_{L^4}\|\na
u\|_{L^4}\,dt'\bigr).
\end{split}
\eeno However, notice from \eqref{b.10} and Proposition
\ref{propb.2} that \beq\label{a.4a}
\begin{split}
\int_0^t\|u\|_{L^4}\|\na u\|_{L^4}\,dt'\leq
&C\int_0^t\bigl(\|u\|_{L^2}^{\f12}\|\na
u\|_{L^2}\|u_t\|_{L^2}^{\f12}+\|u\|_{L^2}\|\na
u\|_{L^2}^2\bigr)\,dt'\\
\lesssim & \|u\|_{L^\infty_t(L^2)}^{\f12}\|\na
u\|_{L^2_t(L^2)}\|\langle t'\rangle^{\bigl(\f{1}2+\d)_-}
u_t\|_{L^2_t(L^2)}^{\f12}\|\langle t'\rangle^{-(\f{1}4+\f{\d}2)_-}
\|_{L^4_t}\\
&+\|u\|_{L^\infty_t(L^2)}\|\na u\|_{L^2_t(L^2)}^2\leq
C\bigl(C_1^{\f12}\|u_0\|_{L^2}^{\f32}+\|u_0\|_{L^2}^3\bigr),
\end{split}
\eeq and \beno \|\p_tu\|_{L^1_t(L^2)}\leq
C\|\w{t'}^{(\f12+\d)_-}\p_tu\|_{L^2_t(L^2)}\leq C_1, \eeno so that
we obtain \beq\label{a.5} \sup_j2^{(-1+\frac{2}{q}+\varepsilon)j}
\|[\dot\Delta_j\mathbb{P},\frac{1}{\rho}]\bigl(\dv(2\mu(\rho)d)
-\nabla\Pi\bigr)\|_{L^1_t(L^q)} \leq CC_1
\|\rho\|_{L^\infty_t(B^{\varepsilon}_{\infty,\infty})}. \eeq
 As $\f2q+\e<1,$ the same
process also ensures \beq\label{a.6}
\begin{split}
\sup_j2^{(\frac{2}{q}+\varepsilon)j}&
\|[\dot\Delta_j\mathbb{P},\mu(\rho)](2d)\|_{L^1_t(L^q)} \leq
C\|\rho\|_{L^\infty_t( B^{\frac{2}{q}+\varepsilon}_{\infty,\infty})}
\|\nabla u\|_{L^1_t(L^q)}\\
&\qquad\leq C\|\rho\|_{L^\infty_t(
B^{\frac{2}{q}+\varepsilon}_{\infty,\infty})}\|\w{t'}^{(\f12+\d-\f1q)_-}\na
u\|_{L^2_t(L^q)}
 \leq C
C_1^{2} \|\rho\|_{L^\infty_t(
B^{\frac{2}{q}+\varepsilon}_{\infty,\infty})}.
\end{split}
\eeq Substituting \eqref{a.4}, \eqref{a.5} and \eqref{a.6} into
\eqref{a.3} results in \eqref{a.1rf}, and we complete the proof of
Lemma \ref{lema.1.1}.
\end{proof}

With Lemma \ref{lema.1.1}, we can prove the {\it a priori} $
L^1(\R^+;\,\dot B^1_{\infty,1}(\R^2))$ estimate for $u.$

\begin{prop}\label{propa.lip}
{\sl Under the assumptions of Lemma \ref{lema.1.1}, there exists a
positive constant $C$ which depends on $m, M$ and $
\|\mu'\|_{L^{\infty}}$ such that if \beq\label{a.7}
4C^2C_1^2\bigl(1+\|\r_0\|_{B^{\f2q+\e}_{\infty,\infty}}\bigr)\|\mu(\r_0)-1\|_{L^\infty}\leq
1, \eeq  for $C_1$ given by \eqref{p.7a}, one has \beq\label{a.8}
\|u\|_{L^1_t(\dot B^{1}_{\infty,1})} \le
2CC_1^2\bigl(1+\|\r_0\|_{B^{\f2q+\e}_{\infty,\infty}}\bigr). \eeq }
\end{prop}
\begin{proof}
 Bony's decomposition \eqref{bony} for $(\mu(\rho)-1)d$ reads \beno
(\mu(\rho)-1)d=T_{\mu(\rho)-1}d+T_d(\mu(\rho)-1)+\cR(\mu(\rho)-1,d).
\eeno Applying para-product estimates (\cite{BCD}) gives
\beq\label{a.9}\begin{split} \|T_{\mu(\rho)-1}d\|_{L^1_t(\dot
B^0_{\infty,1})} \lesssim& \|\mu(\r)-1\|_{L^\infty_t(L^\infty)}
\|u\|_{L^1_t(\dot B^1_{\infty,1})}\\
 \lesssim&
\|\mu(\r_0)-1\|_{L^\infty} \|u\|_{L^1_t(\dot B^1_{\infty,1})}.
\end{split} \eeq
To deal with $\cR(\mu(\rho)-1,d),$ for any integer $N,$ we decompose
it as
$$
\begin{aligned}
\|\cR(\mu(\rho)-1,d)\|_{L^1_t(\dot B^0_{\infty,1})} &\lesssim
\sum_{\ell\leq0}\|\dot\D_\ell(\cR(\mu(\rho)-1,d))\|_{L^1_t(L^\infty)}
\\&
+\sum_{0\leq\ell\leq
N}\|\dot\D_\ell(\cR(\mu(\rho)-1,d))\|_{L^1_t(L^\infty)}
\\&
+\sum_{N\leq\ell}\|\dot\D_\ell(\cR(\mu(\rho)-1,d))\|_{L^1_t(L^\infty)}
\eqdefa \mbox{I}+\mbox{II}+\mbox{III}.
\end{aligned}
$$
Let $q$ be as in Lemma \ref{lema.1.1} and $\bar{q}=\f{2q}{2+q}.$
Then by virtue of Lemma \ref{lem2.1} and para-product estimates
(\cite{BCD}), we have
$$
\begin{aligned}
\mbox{I} \lesssim& \|\cR(\mu(\rho)-1,d)\|_{L^1_t(\dot
B^0_{\bar{q},\infty})} \\
\lesssim& \|(\mu(\rho)-1)d\|_{L^1_t(L^{\bar{q}})}
+\|T_{\mu(\rho)-1}d\|_{L^1_t(\dot B^0_{\bar{q},\infty})}
+\|T_d(\mu(\rho)-1)\|_{L^1_t(\dot B^0_{\bar{q},\infty})}
\\
\lesssim&\|\mu(\r)-1\|_{L^\infty_t(L^{2})}\|\na
u\|_{L^1_t(L^q)}\lesssim C_1^2\|\r_0-1\|_{L^2},
\end{aligned}
$$
where in the last step, we used \eqref{p.8}. Along the same line,
one has
$$
\begin{aligned}
\mbox{II} \lesssim N\|\cR(\mu(\rho)-1,d)\|_{L^1_t(\dot
B^0_{\infty,\infty})} \lesssim
N\|\mu(\r_0)-1\|_{L^\infty}\|u\|_{L^1_t(\dot B^1_{\infty,1})},
\end{aligned}
$$
and
$$
\begin{aligned}
\mbox{III} \lesssim& \sum_{\ell\geq N}\sum_{j\geq
\ell-N_0}\|\dD_j(\mu(\r)-1)\|_{L^\infty_t(L^\infty)}\|\dD_j(\na
u)\|_{L^1_t(L^\infty)}\\
\lesssim &\sum_{\ell\geq
N}2^{-\ell\varepsilon}\|\mu(\r)-1\|_{L^\infty_t(L^\infty)}
\|u\|_{\wt{L}^1_t(\dot
B^{1+\varepsilon}_{\infty,\infty})}\\
\lesssim &2^{-N\varepsilon}\|\mu(\r_0)-1\|_{L^\infty}
\|u\|_{\wt{L}^1_t(\dot B^{1+\f2q+\varepsilon}_{q,\infty})}
\end{aligned}
$$
for any $q\in (2, p^\ast)$ with $q\e>1.$  Hence we obtain
\beq\label{a.10}
\begin{aligned}
\|\cR(\mu(\rho)-1,d)\|_{L^1_t(\dot B^0_{\infty,1})} \leq& CC_1^2+ C
\|\mu(\r_0)-1\|_{L^\infty}\\
&\quad\times \Bigl(N\|u\|_{L^1_t(\dot
B^1_{\infty,1})}+2^{-N\varepsilon} \|u\|_{L^1_t(\dot
B^{1+\frac{2}{q}+\varepsilon}_{q,\infty})}\Bigr).
\end{aligned}
\eeq The same process leads to \beq\label{a.11}
\begin{aligned}
\|T_d(\mu(\rho)-1)\|_{L^1_t(\dot B^0_{\infty,1})} &\leq CC_1^2+
C\Bigl(\|\mu(\r_0)-1\|_{L^\infty}\|\nabla u\|_{L^1_t(L^\infty)}N
\\&\qquad\qquad
+2^{-N\varepsilon} \|\mu(\r)-1\|_{L^\infty_t(\dot
B^{\varepsilon}_{\infty,\infty})} \|\nabla
u\|_{L^1_t(L^\infty)}\Bigr).
\end{aligned}
\eeq Notice that \beno \begin{split}
&\|u_0\|_{\dot{B}^{-1+\f2q+\e}_{q,\infty}}\lesssim
\|u_0\|_{\dot{B}^{\e}_{2,\infty}}\lesssim
\|u_0\|_{H^1},\\
&\|\r\|_{L^\infty_t(B^{\f2q+\e}_{\infty,\infty})}\leq
\|\r_0\|_{B^{\f2q+\e}_{\infty,\infty}}\exp\bigl(C\|\nabla
u\|_{L^1_t(L^\infty)}\bigr),\\
&\|\mu(\r)-1\|_{L^\infty_t(B^\e_{\infty,\infty})}\leq
\|\mu(\r_0)-1\|_{B^\e_{\infty,\infty}}\exp\bigl(C\|\nabla
u\|_{L^1_t(L^\infty)}\bigr), \end{split} \eeno and  Riesz transform
maps continuously from $\dot B^0_{\infty,1}$ from $\dot
B^0_{\infty,1}$ with uniform bound, we get,
 by summing up \eqref{a.9}, \eqref{a.10}, \eqref{a.11} and  Lemma
\ref{lema.1.1}, that \beno\begin{split}
\|\mathcal{R}\mathbb{P}\mathcal{R}
\cdot\bigl(2(\mu(\rho)-1)d\bigr)\|_{L^1_t(\dot B^0_{\infty,1})} \leq
& CC_1^2+
C\Bigl(\|\mu(\r_0)-1\|_{L^\infty}\|u\|_{L^1_t(\dot{B}^1_{\infty,1})}N\\
&+C_1^22^{-N\e}\bigl(1+\|\r_0\|_{B^{\f2q+\e}_{\infty,\infty}}\exp\bigl(C\|\nabla
u\|_{L^1_t(L^\infty)}\bigr)\bigr)\\
&+2^{-N\varepsilon} \|\mu(\r_0)-1\|_{
B^{\varepsilon}_{\infty,\infty}}\exp\bigl(C\|\nabla
u\|_{L^1_t(L^\infty)}\bigr) \Bigr).
\end{split}
\eeno Let $[Y]$ be the integer part of $Y.$ Then choosing
$N=\Bigl[\f{C}{\e\ln 2}\|\nabla u\|_{L^1_t(L^\infty)}\Bigr]$ so that
\beno C2^{-N\varepsilon} \exp\bigl(C\|\nabla
u\|_{L^1_t(L^\infty)}\bigr) \leq 1 \eeno in the above inequality
results in \beq\label{a.12}\begin{split}
\|\mathcal{R}\mathbb{P}\mathcal{R}
\cdot\bigl(2(\mu(\rho)-1)d\bigr)\|_{L^1_t(\dot B^0_{\infty,1})}
\leq&CC_1^2\bigl(1+\|\r_0\|_{B^{\f2q+\e}_{\infty,\infty}}\bigr)+C\|\mu(\r_0)-1\|_{L^\infty}\|u\|_{L^1_t(\dot{B}^1_{\infty,1})}^2.
\end{split}
\eeq

To handle
$\mathcal{R}(-\Delta)^{-\frac{1}{2}}\mathbb{P}\dv\bigl(2\mu(\rho)d\bigr),$
for any integer $L,$ we get, by applying Lemma \ref{lem2.1}, that
$$
\begin{aligned}
\|\mathcal{R}&(-\Delta)^{-\frac{1}{2}}\mathbb{P}\dv
\bigl(2\mu(\rho)d\bigr)\|_{L^1_t(\dot B^0_{\infty,1})}\\
 &\lesssim
\sum_{\ell\leq0}2^{\f{2\ell}q}\|\mu(\rho)d\|_{L^1_t(L^q)}
+\sum_{0\leq\ell\leq
L}\|\dot\D_\ell(\r\pa_tu+\r(u\cdot\nabla)u)\|_{L^1_t(L^2)}
\\&\qquad
+\sum_{L\leq\ell}2^{-\ell\varepsilon}
\|\mu(\rho)d\|_{\wt{L}^1_t(\dot B^\varepsilon_{\infty,\infty})}
\\&
\lesssim \|\na u\|_{L^1_t(L^q)}+(\|\p_tu\|_{L^1_t(L^2)}+\|u\cdot\na
u\|_{L^1_t(L^2)})\sqrt{L}\\
&\qquad+2^{-L\e}\bigl(\|\mu(\rho)\|_{L^\infty_t(L^\infty)}\|u\|_{\wt{L}^1_t(\dot
B^{1+\f2q+\varepsilon}_{q,\infty})}+\|\mu(\rho)\|_{\wt{L}^\infty_t(\dot
B^\e_{\infty,\infty})}\|\na u\|_{L^1_t(L^\infty)}\bigr),
\end{aligned}
$$
from which, Lemma \ref{lema.1.1}, \eqref{a.4a} and Proposition
\ref{propb.3}, we infer
$$
\begin{aligned}
\|\mathcal{R}(-\Delta&)^{-\frac{1}{2}}\mathbb{P}\dv
\bigl(2\mu(\rho)d\bigr)\|_{L^1_t(\dot B^0_{\infty,1})}\\
 \leq&
 C\Bigl\{C_1^2+C_1\sqrt{L}+C_1^22^{-L\e}\bigl(1+\|\r_0\|_{B^{\f2q+\e}_{\infty,\infty}}\bigr)
 \exp\bigl(C\|\na u\|_{L^1_t(L^\infty)}\bigr)\Bigr\}.
\end{aligned}
$$
Taking $L=\Bigl[\f{C}{\e\ln 2}\|\na u\|_{L^1_t(L^\infty)}\Bigr]$ in
the above inequality results in \beq\label{a.13}
\begin{aligned}
&\|\mathcal{R}_i(-\Delta)^{-\frac{1}{2}}\mathbb{P}_j\dv
\bigl(2\mu(\rho)d\bigr)\|_{L^1_t(\dot B^0_{\infty,1})}\leq
CC_1^2\bigl(1+\|\r_0\|_{B^{\f2q+\e}_{\infty,\infty}}\bigr)
+\f12\|\na u\|_{L^1_t(L^\infty)}.
\end{aligned}
\eeq Thus thanks to \eqref{b.5af}, we get, by combining \eqref{a.12}
with \eqref{a.13}, that \beno \begin{split} \|u\|_{L^1_t(\dot
B^1_{\infty,1})} \leq &
CC_1^2\bigl(1+\|\r_0\|_{B^{\f2q+\e}_{\infty,\infty}}\bigr)+C\|\mu(\r_0)-1\|_{L^\infty}\|u\|_{L^1_t(\dot
B^1_{\infty,1})}^2+\f12\|\na u\|_{L^1_t(L^\infty)},
\end{split}
\eeno which ensures that
$$
\|u\|_{L^1_t(\dot B^1_{\infty,1})} \le
CC_1^2\bigl(1+\|\r_0\|_{B^{\f2q+\e}_{\infty,\infty}}\bigr)+C\|\mu(\r_0)-1\|_{L^\infty}\|u\|_{L^1_t(\dot
B^1_{\infty,1})}^2,
$$
from which and \eqref{a.7}, we conclude \eqref{a.8}. This completes
the proof of Proposition \ref{propa.lip}.
\end{proof}

\renewcommand{\theequation}{\thesection.\arabic{equation}}
\setcounter{equation}{0}

\section{The blow-up criterion of \eqref{1.1}}

The purpose of this section is to prove a blow-up criterion for
smooth enough solutions of \eqref{1.1}. As a matter of fact, we
shall prove a more general result concerning the propagation of
regularities for \eqref{1.1} which does not require any smallness
assumption on the fluctuation of the viscous coefficient. Toward
this, let $a\eqdefa \f1\r-1$ and $\widetilde\mu(a)\eqdefa
\mu(\frac{1}{1+a}),$ we write \eqref{1.1} as: \beq\label{INSa}
\left\{\begin{array}{l}
\displaystyle\partial_t a+(u\cdot\nabla)a=0\\
\displaystyle\partial_t  u +u\cdot\nabla u+(1+a)
\big\{\nabla\Pi-\dv\big(\widetilde\mu(a)2d\big)\big\}=0\\
\displaystyle\dv\, u=0\\
\displaystyle(a,u)_{|t=0}=(a_0,u_0).
\end{array}
\right. \eeq

The main result can be listed as follows, which is a similar version
of blow-up criterion for hyperbolic systems (\cite{majda}).

\begin{thm}\label{est.app}
{\sl Let $s>1$ and $a_0\in H^{1+s}(\R^2)$ satisfy \beq
0<\frak{m}\leq 1+a_0\leq \frak{M}.\eeq  Let $u_0\in H^s(\R^2)$ be a
solenoidal vector field.  Then there exists a positive time
$T^\ast,$ so that \eqref{INSa} has a unique solution $(a,u)$ with
$a\in\cC([0,T]; H^{1+s}(\R^2)),$  $u\in \cC([0,T];
H^s(\R^2))\cap\wt{L}^1_T(H^{s+2})$ for any $T<T^\ast.$ Moreover, if
$T^\ast$ is the maximum time of existence and $T^\ast<\infty,$ there
holds \beq \int_0^{T^\ast}\|\na u\|_{L^\infty}\,dt'=\infty. \eeq }
\end{thm}

\begin{proof} We first deduce from the standard well-posedness theory
(see \cite{abidi,danchin04} for instance) that \eqref{INSa} has a
unique solution on $[0,T^\ast)$ for some positive time
$T^\ast<\infty.$ Moreover, there holds \beq \label{k.0} \frak{m}\leq
1+a(t,x)\leq \frak{M},\quad\mbox{and}\quad
\|a(t)\|_{L^p}=\|a_0\|_{L^p}\quad\forall\ p\in [2,\infty],\ \
t<T^\ast. \eeq And it follows from the proof of \eqref{thma.2} that
 \beq\label{k.0a} \|u\|_{L^\infty_t(L^2)}^2+\|\na
u\|_{L^2_t(L^2)}^2\leq C\|u_0\|_{L^2}^2\quad\mbox{for}\ \ t<T^\ast.
\eeq
 While we get, by applying $\dot\D_j$ to the continuous
equation of \eqref{INSa} and then taking the $L^2$ inner product of
the resulting equation with $\dot\D_j a,$ that for all $r>0,$ \beno
\begin{split}
\frac12\frac{d}{dt}\|\dot\D_ja\|_{L^2}^2\leq &
\bigl|\int_{\R^2}[\dot\D_j; u]\cdot\na a\ |\ \dot\D_j a\,dx\bigr|,
\end{split}
\eeno applying Lemma \ref{lemd.1} gives \beno
\frac12\frac{d}{dt}\|\dot\D_ja\|_{L^2}^2 \lesssim
c_j^2(t)2^{-2jr}\bigl(\|\na u\|_{L^\infty}\|\r-1\|_{\dot
H^r}+\|\na\r\|_{L^\infty} \|u\|_{\dot H^r}\bigr)\|a\|_{\dot H^r},
\eeno from which, we infer \beno
 \|a\|_{\wt L^\infty_t(\dot H^{r})}\leq \|a_0\|_{\dot
H^{r}}+C\int_0^t\bigl(\|\na u\|_{L^\infty}\|a\|_{\dot H^{r}}+\|\na
a\|_{L^\infty}\|u\|_{\dot H^{r}}\bigr)\,dt'. \eeno Applying
Gronwall's inequality, \eqref{k.0},  and the fact that
\beq\label{k.2} \|\nabla a\|_{L^\infty_t(L^p)}\leq \|\nabla
a_0\|_{L^p}\exp\bigl(\|\nabla
u\|_{L^1_t(L^\infty)}\bigr)\quad{for}\quad \forall \ p\in
[1,\infty],\eeq leads to \beq\label{k.3}
 \|a\|_{\wt L^\infty_t(H^{r})}\leq \bigl(\|a_0\|_{
H^{r}}+C\|\na
a_0\|_{L^\infty}\|u\|_{L^1_t({H}^r)}\bigr)\exp\bigl(C\|\nabla
u\|_{L^1_t(L^\infty)}\bigr). \eeq

On the other hand, applying $\D_q$ to the momentum equation of
\eqref{INSa} and using Bony's decomposition \eqref{bony} in the
inhomogeneous context, one has
\begin{equation}\label{deco.fre}
\begin{aligned}
&\partial_t\D_qu+u\cdot\nabla\D_qu +\D_q\nabla\Pi+\D_q\nabla
T_a\Pi-\dv\bigl((1+a)\widetilde\mu(a)\D_q(2d)\bigr)\\
&=\bigl[\D_q; u\cdot\nabla\bigr]u +\D_qT_{\nabla a}\Pi-\D_q
\cR(a,\nabla\Pi)+R_j,
\end{aligned}
\end{equation}
where \beq\label{k.4}
\begin{aligned}
R_q&=\D_q\bigl[(1+a)\dv(\widetilde\mu(a)2d)\bigr]
-\dv\bigl((1+a)\widetilde\mu(a)\D_q(2d)\bigr)
\\&
= \D_q\bigl[a\,\dv\big((\widetilde\mu(a)-1)2d\bigr)\bigr]
-\dv\bigl[a\,\D_q\bigl((\widetilde\mu(a)-1)2d\bigr)\bigr]
\\&\quad
+\D_q\bigl(a\,\dv(2d)\bigr)-\dv \bigl(a\,\D_q(2d)\bigr)
-\dv\bigl\{(1+a)[\D_q;\widetilde\mu(a)-1] \cdot(2d)\bigr\}
\\&
\eqdefa R_q^1+\dv\,R_q^2,
\end{aligned}
\eeq and $R_j^2\eqdefa-(1+a)[\D_q;\widetilde\mu(a)-1] \cdot(2d),$ $
R_q^1\eqdefa R_q-\dv\,R_q^2.$

Taking $L^2$ inner product of \eqref{deco.fre} with $\D_qu$ and
using $\dv\, u=0,$ we obtain \beno
\begin{split}
\f12&\f{d}{dt}\|\D_qu\|_{L^2}^2+\bigl((1+a)\widetilde\mu(a)\D_q(2d)\ |\ \D_q(2d)\bigr)_{L^2}\\
&=\Bigl(\bigl[\D_q; u\cdot\nabla\bigr]u +\D_qT_{\nabla a}\Pi-\D_q
\cR(a,\nabla\Pi)+R_q\  |\ \D_qu\Bigr)_{L^2}.
\end{split}
\eeno Notice that $\frak{m}\leq(1+a)$ and $0<\underline\mu\leq\mu,$
then we get, by applying standard process (like \cite{DAN}) and
Lemma \ref{lem2.1}, that \beq\label{k.5}
\begin{aligned}
\Vert u\Vert_{\widetilde L^\infty_t(H^s)} +\Vert u\Vert_{\widetilde
L^1_t(H^{s+2})} \lesssim& \Vert
u_0\Vert_{H^s}+\|\D_{-1}u\|_{L^1_t(L^2)}
+\Bigl(\sum_{q\geq-1}2^{2qs}\bigl\Vert[u;\D_q]\cdot\nabla
u\big\Vert_{L^1_t(L^2)}^2\Bigr)^{\frac{1}{2}}
\\&
+\Vert T_{\nabla a}\Pi\Vert_{\widetilde L^1_t(\dH^s)} +\Vert
\cR(a,\nabla\Pi)\Vert_{\widetilde L^1_t(\dH^{s})}
\\&+\Bigl(\sum_{q\geq-1}2^{2qs} \Vert
R_q^1\Vert_{L^1_t(L^2)}^2\Bigr)^{\frac{1}{2}}
+\Bigl(\sum_{q\geq-1}2^{2q(s+1)}\Vert
R_q^2\Vert_{L^1_t(L^2)}^2\Bigr)^{\frac{1}{2}},
\end{aligned}
\eeq where $R_q^1, R_q^2$ are given by \eqref{k.4}. For $s>0,$
applying Lemma \ref{lemd.1} yields \beno \Vert[u;\D_q]\cdot\nabla
u\big\Vert_{L^2}\lesssim c_q(t)2^{-qs}\|\na
u\|_{L^\infty}\|u\|_{H^s}, \eeno from which, we deduce
\beq\label{k.6}
\begin{split}
\Bigl(\sum_{q\geq-1}2^{2qs}\bigl\Vert[u;\D_q]\cdot\nabla
u\big\Vert_{L^1_T(L^2)}^2\Bigr)^{\frac{1}{2}}\lesssim&
\int_0^t\Bigl(\sum_{q\geq-1}2^{2js}\bigl\Vert[u;\D_q]\cdot\nabla
u\big\Vert_{L^2}^2\Bigr)^{\frac{1}{2}}\,dt'\\
\lesssim & \int_0^t\|\nabla u\|_{L^\infty}\|u\|_{H^s}\,dt'.
\end{split}
\eeq Along the same line, we have \beq\label{k.7} \begin{split} &
\Vert T_{\nabla a}\Pi\Vert_{\widetilde L^1_t(\dH^s)} \lesssim
\int_0^t\|\nabla
a\|_{L^\infty}\|\nabla\Pi\|_{\dH^{s-1}}\,dt' \quad\mbox{and}\\
& \Vert \cR(a,\nabla\Pi)\Vert_{\widetilde L^1_t(\dH^{s})} \lesssim
\int_0^t\|a\|_{\dot{B}^s_{\infty,2}}\|\nabla\Pi\|_{L^2}\,dt'
\lesssim \int_0^t\|a\|_{\dH^{s+1}}\|\nabla\Pi\|_{L^2}\,dt'.
\end{split}\eeq
Notice that \beno
\begin{split}
& \D_q\bigl[a\,\dv\big((\widetilde\mu(a)-1)2d\bigr)\bigr]
-\dv\bigl[a\,\D_q\bigl((\widetilde\mu(a)-1)2d\bigr)\bigr]\\
&=[\D_q; a]\dv\big((\widetilde\mu(a)-1)2d\bigr)+[a;
\dv]\D_q\bigl((\widetilde\mu(a)-1)2d\bigr),
\end{split}
\eeno applying Lemma \ref{lemd.1} yields \beno
\begin{split}
&\bigl\|\D_q\bigl[a\,\dv\big((\widetilde\mu(a)-1)2d\bigr)\bigr]
-\dv\bigl[a\,\D_q\bigl((\widetilde\mu(a)-1)2d\bigr)\bigr]\bigr\|\\
&\lesssim c_q(t)2^{-qs}\bigl(\|\nabla
a\|_{L^\infty}\|(\widetilde\mu(a)-1)2d\|_{H^s}
+\|(\widetilde\mu(a)-1)2d\|_{L^\infty}\|a\|_{H^{s+1}}\bigr),
\end{split}
\eeno from which, we deduce, by a similar proof of \eqref{k.6}, that
 \beq\label{k.8}
\begin{aligned}
&\Bigl(\sum_{q\geq-1}2^{2qs}\bigl\Vert
\D_q\bigl[a\,\dv\big((\widetilde\mu(a)-\mu)2d\bigr)\bigr]
-\dv\bigl[a\,\D_q\bigl((\widetilde\mu(a)-\mu)2d\bigr)\bigr]
 \big\|_{L^1_t(L^2)}^2\Big)^{\frac{1}{2}}
\\&
\lesssim \int_0^t\bigl(\|\nabla
a\|_{L^\infty}\|(\widetilde\mu(a)-1)2d\|_{H^s}
+\|(\widetilde\mu(a)-\mu)2d\|_{L^\infty}\|a\|_{H^{s+1}}\bigr)\,dt'
\\&
\lesssim \int_0^t\bigl(\|\nabla a\|_{L^\infty} (\|\nabla
u\|_{L^\infty}\|a\|_{H^s}
+\|\widetilde\mu(a)-1\|_{L^\infty}\|u\|_{H^{s+1}})+\|\nabla
u\|_{L^\infty}\|a\|_{H^{s+1}}\bigr)\,dt'.
\end{aligned}
\eeq While notice that \beno \D_q\bigl(a\,\dv(2d)\bigr)-\dv
\bigl(a\,\D_q(2d)\bigr)=[\D_q;a]\dv(2d)+[a;\dv]\D_j(2d), \eeno a
similar proof of \eqref{k.8} leads to \beno
\begin{aligned}
\Bigl(\sum_{q\geq-1}2^{2qs}\bigl\Vert
\D_q\bigl(a\,&\dv(2d)\bigr)-\dv \bigl(a\,\D_q(2d)\bigr)
\big\|_{L^1_t(L^2)}^2\Bigr)^{\frac{1}{2}} \\
&\lesssim \int_0^t\bigl(\|\nabla a\|_{L^\infty}\|u\|_{H^{s+1}}
+\|\nabla u\|_{L^\infty}\|a\|_{H^{s+1}}\bigr)\,dt'.
\end{aligned}
\eeno Therefore thanks to \eqref{k.4}, we obtain \beq\label{k.9}
\begin{aligned}
\Bigl(\sum_{q\geq-1}2^{2qs}\Vert
R^1_j\|_{L^1_t(L^2)}^2\Bigr)^{\frac{1}{2}} \lesssim &
\int_0^t\bigl(\|\nabla a\|_{L^\infty}(\|\nabla
u\|_{L^\infty}\|a\|_{H^s}\\& +(1+\|a\|_{L^\infty})\|u\|_{H^{s+1}})
+\|\nabla u\|_{L^\infty}\|a\|_{H^{s+1}}\bigr)\,dt'.
\end{aligned}
\eeq It follows the same line that \beq\label{k.10}
\begin{aligned}
\Bigl(\sum_{q\geq-1}2^{2q(s+1)}\Vert
R_j^2\Vert_{L^1_T(L^2)}^2\Bigr)^{\frac{1}{2}} \lesssim&
\int_0^t(1+\|a\|_{L^\infty})\bigl(\|\nabla
a\|_{L^\infty}\|u\|_{H^{s+1}} +\|\nabla
u\|_{L^\infty}\|a\|_{H^{s+1}}\bigr)\,dt'.
\end{aligned}
\eeq It remains to handle the pressure function $\Pi$ in
\eqref{INSa}. Toward this, we get, by taking divergence to the
momentum equation of \eqref{INSa}, that \beq \label{k.11}
\begin{aligned}
\dv\{(1+a)\nabla\Pi\} =&-\dv\{(u\cdot\nabla)u\}
+\dv\{a\,\dv[(\widetilde\mu(a)-1)(2d)]\}
\\&
+\dv\,\dv\{(\widetilde\mu(a)-1)(2d)\} +\dv(a\Delta u),
\end{aligned}
\eeq applying Bony's decomposition \eqref{bony} in the inhomogeneous
contaxt to the right hand side of \eqref{k.11}, we have \beno
\begin{aligned}
\dv\{(1+a)\nabla\Pi\} =&-\dv\{(u\cdot\nabla)u\} +\dv\,
T_a\dv\bigl\{T_{\widetilde\mu(a)-1}(2d)
+R(\widetilde\mu(a)-1,2d)\bigr\}
\\&
+\dv\,R(a,\dv[(\widetilde\mu(a)-1)(2d)])
+\dv\,\dv\,T_{\widetilde\mu(a)-1}(2d)
\\&+\dv\,\dv\,R(\widetilde\mu(a)-1,2d)
+T_{\nabla a}\Delta u +\dv\,R(a,\Delta u),
\end{aligned}
\eeno from which and the fact that $\dv u=0,$ we infer
 \beno
\begin{aligned}
\dv\{(1+a)\nabla\Pi\} =& -\dv\{(u\cdot\nabla)u\} +T_{\nabla
a}\dv\bigl\{T_{\widetilde\mu(a)-\mu^1}(2d)
+R(\widetilde\mu(a)-1,2d)\bigr\}
\\&
+T_{a}\dv\,T_{\nabla\widetilde\mu(a)}(2d)
+T_{a}T_{\nabla\widetilde\mu(a)}\Delta u
+T_{a}\dv\,\dv\,R(\widetilde\mu(a)-1,2d)
\\&
+\dv\,R(a,\dv[(\widetilde\mu(a)-1)(2d)])
+\dv\,T_{\nabla\widetilde\mu(a)}(2d)+T_{\nabla\widetilde\mu(a)}\Delta
u
\\&
+\dv\,\dv\,R(\widetilde\mu(a)-1,2d) +T_{\nabla a}\Delta u
+\dv\,R(a,\Delta u).
\end{aligned}
\eeno taking $L^2$ inner product of the above equation with $\Pi$
and using \eqref{k.0}, we reach  \beno
\begin{aligned}
\|\nabla\Pi\|_{L^2} \lesssim& \|\nabla u\|_{L^\infty}\|u\|_{L^2}
+\|T_{\nabla a}\dv\big\{T_{\widetilde\mu(a)-\mu^1}(2d)
+R(\widetilde\mu(a)-\mu^1,2d)\big\}\|_{\dot H^{-1}}\\
& +\|T_{a}\dv\,T_{\nabla\widetilde\mu(a)}(2d)\|_{\dH^{-1}}
+\|T_{a}T_{\nabla\widetilde\mu(a)}\Delta u\|_{\dH^{-1}}
+\|T_{a}\dv\,\dv\,R(\widetilde\mu(a)-1,2d)\|_{\dot H^{-1}}\\&
+\|R(a,\dv[(\widetilde\mu(a)-1)(2d)])\|_{L^2}
+\|T_{\nabla\widetilde\mu(a)}(2d)\|_{L^2}
+\|T_{\nabla\widetilde\mu(a)}\Delta u\|_{\dot H^{-1}}
\\&
+\|R(\widetilde\mu(a)-1,2d)\|_{\dot H^{1}} +\|T_{\nabla a}\Delta
u\|_{\dot H^{-1}} +\|R(a,\Delta u)\|_{L^2},
\end{aligned}
\eeno which together with standard  para-product estimates
(\cite{BCD}) and \eqref{k.0} implies \beq\label{k.12}
\begin{aligned}
\|\nabla\Pi\|_{L^2} \lesssim& \|\nabla
u\|_{L^\infty}(\|u\|_{L^2}+\|a\|_{H^1}) +\|\nabla a\|_{L^\infty}
(1+\|a\|_{H^1})\|\nabla u\|_{L^2} +\|a\|_{H^2}\|\nabla u\|_{L^2}.
\end{aligned}
\eeq To deal with $\|\nabla\Pi\|_{H^{s-1}},$ we get by acting $\D_q$
to \eqref{k.11} and taking $L^2$ inner product of the resulting
equation with $\D_q\Pi$ that
$$
\begin{aligned}
\|\nabla\Pi\|_{H^{s-1}} \lesssim& \|u\otimes
u\|_{H^s}+\|(1+a)\dv[(\widetilde\mu(a)-1)(2d)]\|_{H^{s-1}}
\\& +\|(1+a)(2d)\|_{H^{s-1}}
+\Bigl(\sum_{q\geq
-1}2^{2q(s-1)}\|[\Delta_q;a]\nabla\Pi\|_{L^2}^2\Big)^{\frac{1}{2}},
\end{aligned}
$$
from which, standard product laws in Sobolev space  and Lemma
\ref{lemd.1}, we obtain
$$
\begin{aligned}
\|\nabla\Pi\|_{H^{s-1}} \lesssim&
(1+\|a\|_{L^\infty}+\|u\|_{L^\infty})\|u\|_{H^s}
+\|a\|_{H^{s}}(\|\na u\|_{L^2}+\|\nabla
u\|_{L^\infty})+\|a\|_{L^\infty}\|u\|_{H^{s+1}}
\\&
+\|a\|_{H^s}\bigl(\|a\|_{H^{1}}\|\nabla
u\|_{L^\infty}+\|a\|_{L^\infty}\|u\|_{H^2}\bigr) +\|\nabla
a\|_{L^\infty}\|\nabla\Pi\|_{H^{s-2}}
+\|\nabla\Pi\|_{L^2}\|a\|_{H^s}.
\end{aligned}
$$
If $s-2\leq0,$ then
$\|\nabla\Pi\|_{H^{s-2}}\lesssim\|\nabla\Pi\|_{L^2}$ otherwise
$$
\|\nabla a\|_{L^\infty}\|\nabla\Pi\|_{H^{s-2}} \leq
\eta\|\nabla\Pi\|_{H^{s-1}} +C\|\nabla
a\|_{L^\infty}^{s-1}\|\nabla\Pi\|_{L^2}.
$$
As a consequence, by taking $\eta$ sufficiently small, we arrive at
 \beq\label{k.13}
\begin{aligned}
\|\nabla\Pi\|_{H^{s-1}} \lesssim& (1+\|u\|_{L^\infty})\|u\|_{H^s}
+\|a\|_{H^{s}}\|\nabla
u\|_{L^\infty}+\|a\|_{H^s}\bigl(\|a\|_{H^{1}}\|\nabla
u\|_{L^\infty}+\|u\|_{H^2}\bigr)
\\&
+\|u\|_{H^{s+1}} +\big(\|\nabla a\|_{L^\infty}
+\bigl\langle\bigl\langle\|\nabla
a\|_{L^\infty}^{s-1}\bigr\rangle\bigr\rangle_{s>2}+\|a\|_{H^s}\big)\|\nabla\Pi\|_{L^2},
\end{aligned}
\eeq where $\bigl\langle\bigl\langle\|\nabla
a\|_{L^\infty}^{s-1}\bigr\rangle\bigr\rangle_{s>2}=\|\nabla
a\|_{L^\infty}^{s-1}$ when $s>2$, and equal to $0$ otherwise.

Therefore, substituting \eqref{k.6}, \eqref{k.7}, \eqref{k.9},
\eqref{k.10}, \eqref{k.12} and \eqref{k.13} into \eqref{k.5}, we
reach \beq\label{k.14}
\begin{aligned}
&\Vert u\Vert_{\widetilde L^\infty_T(H^s)} +\Vert u\Vert_{\widetilde
L^1_T(H^{s+2})} \lesssim \Vert u_0\Vert_{H^s}
+\|u\|_{L^1_t(L^2)}+\int_0^t\|\nabla
u\|_{L^\infty}\|u\|_{H^s}\,dt'\\
&\quad +\int_0^t\|\nabla
a\|_{L^\infty}\bigl\{(1+\|u\|_{L^\infty})\|u\|_{H^s}
+\|a\|_{H^s}\bigl((1+\|a\|_{H^{1}})\|\nabla
u\|_{L^\infty}+\|u\|_{H^2}\bigr)\\
&\quad +\|u\|_{H^{s+1}}\bigr\}\,dt'+\int_0^t\bigl(\|\nabla
a\|_{L^\infty}^2 +\|\nabla
a\|_{L^\infty}^{s}+\|a\|_{H^{s+1}}\bigr)\|\nabla\Pi\|_{L^2}\,dt'\\
&\quad+\int_0^t\bigl\{\|\nabla a\|_{L^\infty}(\|\nabla
u\|_{L^\infty}\|a\|_{H^s}+\|u\|_{H^{s+1}}) +\|\nabla
u\|_{L^\infty}\|a\|_{H^{s+1}}\bigr\}\,dt'.
\end{aligned}
\eeq Thanks to \eqref{k.3}, one has
$$
\begin{aligned}
\|a\|_{L^\infty_t(H^{s+1})} &\leq C(1+\|u\|_{L^1_t(H^{s+1})})
\exp\bigl(C\|\nabla u\|_{L^1_t(L^\infty)}\bigr)
\\&
\leq C\bigl(1+\|u\|_{L^1_t(H^{1})}^{\frac{1}{s+1}}\|u\|_{\widetilde
L^1_t(H^{s+2})}^{\frac{s}{s+1}} \bigr)\exp\bigl(C\|\nabla
u\|_{L^1_t(L^\infty)}\bigr).
\end{aligned}
$$
It follows the same line that \beno \|a\|_{L^\infty_t(H^{s})} \leq
C\bigl(1+\|u\|_{L^1_t(H^{1})}^{\frac{2}{s+1}}\|u\|_{\widetilde
L^1_t(H^{s+2})}^{\frac{s-1}{s+1}} \bigr)\exp\bigl(C\|\nabla
u\|_{L^1_t(L^\infty)}\bigr), \eeno from which and \eqref{k.2}, we
deduce that \beno \begin{split} & \int_0^t\|\nabla
a\|_{L^\infty}\|a\|_{H^s}\|u\|_{H^2}\,dt' \lesssim
\bigl(1+t^{\frac{2}{s+2}}\|u\|_{\widetilde
L^1_t(H^{s+2})}^{\frac{s}{s+2}}\bigr)\|u\|_{L^1_t(H^{1})}^{\frac{s}{s+1}}
\|u\|_{\widetilde L^1_t(H^{s+2})}^{\frac{1}{s+1}}\exp\bigl(C\|\nabla
u\|_{L^1_t(L^\infty)}\bigr),\\
&\int_0^t\bigl(\|\nabla a\|_{L^\infty}^2+\|\nabla
a\|_{L^\infty}^s\bigr) \|a\|_{H^2}\|\nabla u\|_{L^2}\,dt' \lesssim
\sqrt{t}\bigl(1+\|u\|_{L^1_t(H^{1})}^{\frac{s}{s+1}}\|u\|_{\widetilde
L^1_t(H^{s+2})}^{\frac{1}{s+1}} \bigr)\exp\bigl(C\|\nabla
u\|_{L^1_t(L^\infty)}\bigr),\\
&\int_0^t\|a\|_{H^{s+1}}\|\nabla u\|_{L^\infty}\,dt' \lesssim
\bigl(1+\|u\|_{L^1_{\tau}(H^{1})}^{\frac{1}{s+1}}\|u\|_{\widetilde
L^1_{\tau}(H^{s+2})}^{\frac{s}{s+1}} \bigr)\exp\bigl(C\|\nabla
u\|_{L^1_t(L^\infty)}\bigr),
\end{split} \eeno and
$$
\begin{aligned}
\int_0^t\|\nabla u\|_{L^2}&\|a\|_{H^{s+1}}\|a\|_{H^2}\,dt'\\ &
\lesssim \int_0^t\|\nabla u\|_{L^2}
(1+\|u\|_{L^1_{t'}(H^{s+1})})(1+\|u\|_{L^1_t(H^{2})})\exp\bigl(C\|\nabla
u\|_{L^1_{t'}(L^\infty)}\bigr)\,dt'
\\&
\lesssim  \int_0^t\|\nabla u\|_{L^2}
\bigl(1+\|u\|_{L^1_{\tau}(H^{1})}^{\frac{1}{s+1}}\|u\|_{\widetilde
L^1_{\tau}(H^{s+2})}^{\frac{s}{s+1}}
 +\|u\|_{L^1_{\tau}(H^{1})}^{\frac{s}{s+1}}\|u\|_{\widetilde
L^1_{\tau}(H^{s+2})}^{\frac{1}{s+1}}\\
&\qquad +\|u\|_{\widetilde
L^1_{\tau}(H^{s+2})}\bigr)\exp\bigl(C\|\nabla
u\|_{L^1_{t'}(L^\infty)}\bigr)\,dt'.
\end{aligned}
$$
Substituting the above inequalities into \eqref{k.14} and using
Young inequality, we obtain \beq\label{k.15}
\begin{aligned}
\Vert u\Vert_{\widetilde L^\infty_t(H^s)} +\Vert u\Vert_{\widetilde
L^1_t(H^{s+2})} \lesssim& 1+t+ \int_0^t\|\na
u\|_{L^2}\exp\bigl(C\|\nabla
u\|_{L^1_{t'}(L^\infty)}\bigr)\|u\|_{\widetilde
L^1_{\tau}(H^{s+2})}\,dt'\\
& +\int_0^t\bigl(\|\nabla
u\|_{L^\infty}+(1+\sqrt{t'})\exp\bigl(C\|\nabla
u\|_{L^1_{t'}(L^\infty)}\bigr)\bigr)\|u\|_{H^s}\,dt'\\
&+f(t,s,a_0,u_0)\exp\bigl(C\|\nabla u\|_{L^1_{t}(L^\infty)}\bigr),
\end{aligned}
\eeq for $f(t,s,a_0,u_0)$ given by $$
\begin{aligned}
f(t,s,u_0,a_0)\eqdefa 1+t+\Vert
u_0\Vert_{H^s}+\bigl(1+t^{\frac{s+1}{2}}+t^{\frac{s+1}{2s}}\bigr)\|u\|_{L^1_t(H^1)}
+\bigl(t^{\frac{2s+2}{s}}
+t^{\frac{6+s}{2s}}\bigr)\|u\|_{L^1_t(H^1)}^{s+2}.
\end{aligned}
$$
Applying Gronwall's inequality to \eqref{k.15} and using
\eqref{k.0a}, we arrive at
$$
\Vert u\Vert_{\widetilde L^\infty_T(H^s)} +\Vert u\Vert_{\widetilde
L^1_T(H^{s+2})} \le
C\,f(t,s,u_0,a_0)\exp\Big\{C(1+\sqrt{t}+t^{\frac{3}{2}})\exp\bigl\{C\|\nabla
u\|_{L^1_{t}(L^\infty)}\bigr\}\Big\}.
$$
This together with \eqref{k.3} completes the proof of Theorem
\ref{est.app}.
\end{proof}

\renewcommand{\theequation}{\thesection.\arabic{equation}}
\setcounter{equation}{0}

\section{The proof of Theorem \ref{thm1.1}}\label{sect4}

Notice from \cite{desjardins} that with the additional regularity
assumptions that $\na \r_0\in L^r(\Bbb{T}^2)$ for some $r>2,$
Desjardins proved that: there exists some positive time $\tau$ so
that Lions weak solution $(\r,u)$ satisfies $u\in L^2((0,\tau);
H^2(\R^2))$ and $\na\r\in L^\infty((0,\tau);L^{\bar{r}}({\Bbb
T}^2))$ for any $\bar{r}<r.$
   Here with Proposition
\ref{propa.lip}, we shall prove that $\tau=\infty$  and $\bar{r}=r.$

\begin{prop}\label{propu.1}
{\sl Let $(\r,u,\na\Pi)$ be a smooth enough solution of \eqref{1.1}
on $[0,T^\ast).$ Then under the assumptions of Theorem \ref{thm1.1},
 we have \beq\label{u.1}
\begin{aligned}
&\|\na u\|_{L^1_t(L^\infty)}\leq 2CC_1^2\bigl(1+\|\r_0\|_{B^{(2/q)_+}_{\infty,\infty}}\bigr)\eqdefa\frak{C},\\
& \|\na\r\|_{L^\infty_t(L^r)}+\|\Delta
u\|_{L^1_t(L^2)}+\|\nabla\Pi\|_{L^1_t(L^2)} \leq
C\bigl(1+\|\na\r_0\|_{L^r}\bigr)\exp\bigl(C\frak{C}\bigr)\quad
\end{aligned}
\eeq for $t\leq T^\ast$ and  $C_1$ given by \eqref{p.7a}. Here and
in that follows, the uniform constant $C$ may depends on $m, M$ and
$\|\mu'\|_{L^\infty}.$}
\end{prop}

\begin{proof} Under the smallness assumption \eqref{thm1.1av}, we
find that \eqref{a.7} is satisfied. Hence, we get, by applying
Proposition \ref{propa.lip}, that \beq\label{u.1al} \|\na
u\|_{L^\infty_t(L^\infty)}\leq
\|u\|_{L^1_t(\dot{B}^1_{\infty,1})}\leq
2CC_1^2\bigl(1+\|\r_0\|_{B^{\f2q+\e}_{\infty,\infty}}\bigr)
\quad\mbox{for}\quad t<T^\ast\ \mbox{and any}\ \ \e>0. \eeq

On the other hand, by taking $L^2$ inner product of $\D u$ with the
momentum equation of \eqref{1.1}, we obtain \beno
\begin{split}
\|\sqrt{\mu(\r)}\Delta
u\|_{L^2}^2=\int_{\R^2}\bigl(\r\pa_tu+\r(u\cdot\nabla)u
-2\mu'(\r)\nabla\r\cdot d\bigr)\ |\ \D u\,dx,
\end{split}
\eeno from which and \eqref{b.3}, we infer \beq\label{u.1as}
\|\Delta u\|_{L^1_t(L^2)}\leq C\Bigl(
\|\p_tu\|_{L^1_t(L^2)}+\|\na\r\|_{L^\infty_t(L^r)}\|\na
u\|_{L^1_t(L^{n})}+\|u\cdot\na u\|_{L^1_t(L^2)}\Bigr),\eeq where $n$
is determined by $\f1r+\f1{n}=\f12.$ It is easy to check, from the
transport equation of \eqref{1.1} and \eqref{u.1al}, that \beno
\|\na\r\|_{L^\infty_t(L^r)}\leq \|\na\r_0\|_{L^r}\exp\bigl(C\|\na
u\|_{L^1_t(L^\infty)}\bigr)\leq
\|\na\r_0\|_{L^r}\exp(C\frak{C})\quad\mbox{for}\quad t\leq T^{\ast},
\eeno moreover, as $r<\f2{1-2\d},$  \eqref{p.8} ensures that \beno
\|\na u\|_{L^1_t(L^{n})}\leq C\|\w{t'}^{(\f12+\d-\f1n)_-}\na
u\|_{L^2_t(L^2)}\|\w{t'}^{-(\d+\f1r)_-}\|_{L^2_t}\leq CC_1^2, \eeno
so that we deduce from \eqref{u.1as} that \beq\label{u.1af} \|\Delta
u\|_{L^2_t(L^2)}\leq C
\bigl(1+\|\na\r_0\|_{L^r}\bigr)\exp\bigl(C\frak{C}\bigr)\quad\mbox{for}\quad
t<T^\ast.\eeq Finally thanks to the momentum equation of
\eqref{1.1}, one has \beno \|\na\Pi\|_{L^1_t(L^2)}\leq C\Bigl(
\|\p_tu\|_{L^1_t(L^2)}+\|\na\r\|_{L^\infty_t(L^r)}\|\na
u\|_{L^1_t(L^{n})}+\|u\cdot\na u\|_{L^1_t(L^2)}+\|\Delta
u\|_{L^1_t(L^2)}\Bigr),\eeno which along with the proof of
\eqref{u.1af} leads to the estimate of $ \|\na\Pi\|_{L^1_t(L^2)}.$
This completes the proof of Proposition \ref{propu.1}.
\end{proof}

We now turn to the proof of Theorem \ref{thm1.1}.

\begin{proof}[Proof  of Theorem \ref{thm1.1}]  Firstly under the assumption of  \eqref{thma.1} and \eqref{thm1.1av},  Theorem
\ref{est.app} together with Proposition \ref{propu.1} ensures
\eqref{1.1} has a global solution $(\r,u)$ with $\r-1\in
\cC([0,\infty);H^{1+s}(\R^2))$ and $u\in
\cC([0,\infty);H^s(\R^2))\cap
\wt{L}^1_{\mbox{loc}}(\R^+;H^{2+s}(\R^2))$ provided that $\r_0-1\in
H^{1+s}(\R^2),$  $u_0\in H^s(\R^2)$ and $\mu(\cdot)\in
W^{2+[s],\infty}(\R^+)$ for some $s>1.$

To prove the global existence of strong solutions of \eqref{1.1}
without the additional regularity assumption that $\r_0-1\in
H^{1+s}(\R^2)$ and $u_0\in H^s(\R^2)$ for $s>1,$ we denote
$\r_{0,\eta}\eqdefa\r_0\ast j_\eta,$ $u_{0,\eta}\eqdefa u_0\ast
j_\eta,$  and $\mu_\eta=\mu\ast j_\eta,$ where
$j_\eta(|x|)=\eta^{-2}j(|x|/\eta)$ is the standard Friedrich's
mollifier. Then \eqref{1.1} with viscous coefficient $\mu_\eta$ and
with initial data $(\r_{0,\eta},u_{0,\eta})$ has a global solution
$(\r_\eta,u_\eta, \na\Pi_\eta).$ Moreover, Proposition \ref{propu.1}
ensures that $(\r_\eta,u_\eta, \na\Pi_\eta)$ satisfy the uniform
estimates \eqref{thma.2} and \eqref{u.1}. This together with a
standard compactness argument yields the existence part of Theorem
\ref{thm1.1}. For simplicity, we skip the details here.

It remains to prove the uniqueness part of Theorem \ref{thm1.1}.
Indeed let $(\r_i,u_i,\nabla\Pi_i),$ for $i=1,2,$ be two solutions
of \eqref{1.1} so that $\r_i\in
\cC_b([0,T];L^\infty\cap\dot{W}^{1,r}(\R^2)),$ $u_i\in
\cC_b([0,T];L^2(\R^2))\cap L^2((0,T);\dot{H}^1(\R^2))\cap
L^1((0,T);\mbox{Lip}(\R^2)),$ and $\p_tu\in L^2((0,T); L^2(\R^2)),$
we denote by \beq\label{u.2} (\delta \r,\delta u,\nabla\delta\Pi)
\eqdefa (\r_2-\r_1,u_2-u_1,\nabla\Pi_2-\nabla\Pi_1). \eeq Then the
system for $(\delta \r,\delta u,\delta\nabla\Pi)$ reads
\beq\label{u.3} \left\{\begin{array}{l}
\displaystyle \pa_t\delta\rho+u_2\cdot\nabla\delta\r=-\delta u\cdot\nabla\r_1\\
\displaystyle \rho_2\pa_t\delta u+\r_2(u_2\cdot\nabla)\delta u
-2\dv\{\mu(\rho_2)d(\delta u)\}
+\grad\delta\Pi=\d F, \\
\displaystyle \dv\,\delta u = 0, \\
\displaystyle (\delta\r,\delta u)_{|t=0}=(0,0).
\end{array}\right.
\eeq where $\d F$ is determined by
$$
\begin{aligned}
\d F=-\delta\r\pa_tu_1-\delta\r(u_2\cdot\nabla)u_1-\r_1(\delta
u\cdot\nabla)u_1 +2\dv\{(\mu(\rho_2)-\mu(\r_1))d(u_1)\}.
\end{aligned}
$$
Let $2<m<r,$ and $p\eqdefa\f{mr}{r-m},$  we deduce from the
transport equation of \eqref{u.3} that \beq \label{u.6}
\begin{split} \|\delta\r(t)\|_{L^m}\leq& \int_0^t\|\delta
u\|_{L^p}\|\nabla\r_1\|_{L^r}\,dt'
\leq\|\nabla\r_1\|_{L^\infty_t(L^r)}\|\delta
u\|_{L^1_t(L^p)}\\
\leq& Ct^{\f12+\f1p}\|\nabla\r_1\|_{L^\infty_t(L^r)}\|\d
u\|_{L^\infty_t(L^2)}^{\f2p}\|\na \d u\|_{L^2_t(L^2)}^{1-\f2p}.
\end{split}
\eeq Whereas taking $L^2$ inner product $\d u$ with the momentum
equation of \eqref{u.3},  we get \beno
\f12\f{d}{dt}\int_{\R^2}\r_2|\d u|^2\,dx+2\int_{\R^2}\mu(\r_2)d(\d
u):d(\d u)\,dx=\int_{\R^2} \d F\ |\ \d u\,dx, \eeno which leads to
\beq\label{u.7}
\begin{aligned}
\|\delta u\|_{L^\infty_t(L^2)}^2+\|\nabla\delta u\|_{L^2_t(L^2)}^2
\leq&
\|\delta\r\|_{L^\infty_t(L^m)}\int_0^t\bigl(\|\pa_tu_1\|_{L^2}+\|u_2\|_{L^4}\|\nabla
u_1\|_{L^4}\bigr) \|\delta u\|_{L^{\bar{m}}}\,dt'
\\&
+\int_0^t\|\r_1\|_{L^\infty}\|\nabla u_1\|_{L^\infty} \|\delta
u\|_{L^2}^2\,dt'\\
& +\|\delta\r\|_{L^\infty_t(L^m)}\int_0^t\|\nabla
u_1\|_{L^{\bar{m}}} \|\nabla\delta u\|_{L^2}\,dt',
\end{aligned}
\eeq where $\f1m+\f1{\bar{m}}=\f12.$ It follows from \eqref{u.6}
that \beno
\begin{split}
\|\delta\r\|_{L^\infty_t(L^m)}\int_0^t\|\pa_tu_1\|_{L^2} \|\delta
u\|_{L^{\bar{m}}}\,dt'\leq& Ct^{1-\f1r}\|\p_tu_1\|_{L^2_t(L^2)}\|\d
u\|_{L^\infty_t(L^2)}^{2(\f12-\f1r)}\|\na\d
u\|_{L^2_t(L^2)}^{2(\f12+\f1r)}\\
\leq &\eta_1(t)\bigl(\|\d u\|_{L^\infty_t(L^2)}^{2}+\|\na\d
u\|_{L^2_t(L^2)}^{2}\bigr),
\end{split}
\eeno where $\lim_{t\to 0}\eta_1(t)=0.$ The same estimate holds for
$\|\delta\r\|_{L^\infty_t(L^m)}\int_0^t\|u_2\|_{L^4}\|\nabla
u_1\|_{L^4}\|\delta u\|_{L^{\bar{m}}}\,dt'.$

While again notice from \eqref{u.6} that \beno
\begin{split}
\|\delta\r\|_{L^\infty_t(L^m)}&\int_0^t\|\nabla u_1\|_{L^{\bar{m}}}
\|\nabla\delta u\|_{L^2}\,dt'\\
\leq &Ct^{1+\f2q}\|\na u_1\|_{L^2_t(L^{\bar{m}})}^2\|\d
u\|_{L^\infty_t(L^2)}^{\f4p}\|\na
\d u\|_{L^2_t(L^2)}^{2(1-\f2p)}+\f18\|\na \d u\|_{L^2_t(L^2)}^2\\
\leq &\eta_2(t)\|\d u\|_{L^\infty_t(L^2)}^{2}+\f14\|\na\d
u\|_{L^2_t(L^2)}^{2},
\end{split}
\eeno
 where $\eta_2(t)$ satisfies $\lim_{t\to 0}\eta_2(t)=0.$

Hence taking $t_1$ small enough, we infer from \eqref{u.7} that
\beno \|\delta u\|_{L^\infty_t(L^2)}^2\leq
C\int_0^t\|\r_1\|_{L^\infty}\|\nabla u_1\|_{L^\infty} \|\delta
u\|_{L^2}^2\,dt'\quad\mbox{for}\quad  t\leq t_1. \eeno Applying
Gronwall's inequality yields \beno \delta u=0\quad\mbox{for}\quad
t\leq t_1, \eeno from which and \eqref{u.6}, we obtain $ \delta\r=0$
for $ t\leq t_1. $ Finally thanks to the momentum equation of
\eqref{u.3}, we get that $\na \d\Pi=0$ for $t\leq t_1.$ The
uniqueness on the whole time interval $[0,T]$ then follows by a
bootstrap argument. This completes the proof of Theorem
\ref{thm1.1}.
\end{proof}

\renewcommand{\theequation}{\thesection.\arabic{equation}}
\setcounter{equation}{0}

\section{Proof of Corollary \ref{cst}}
In this section, we shall repeat the arguments from Section
\ref{sect2}, Section \ref{sect3} and Section \ref{sect4} to prove
the global well-posedness of \eqref{1.1} in the case of constant
viscosity.

\begin{proof}[Proof of Corollary \ref{cst}] We first deduce, by a similar proof of Theorem 1 of   \cite{AGZ},
that the  system \eqref{1.1} with $\mu(\r)=1$ has a unique local
solution on $[0, T^\ast)$ so that \beq\label{p.0} \begin{split} &
\r-1 \in {\mathcal{C}_{b}}([0, t]; \dot B^1_{2,1}\cap\dot
B^\alpha_{\infty,\infty}(\R^2)) \quad\mbox{and}\quad u\in
{\mathcal{C}_{b}}([0, t]; \dot B^0_{2,1}(\R^2)) \cap L^1([0,t];
\,\dot B^{2}_{2,1}(\R^2)),\\
&\f12\|\sqrt{\r}u(t)\|_{L^2}^2+\|\na
u\|_{L^2_t(L^2)}^2=\f12\|\sqrt{\r_0}u_0\|_{L^2}^2\quad\mbox{and}\quad
m\leq\r(t,x)\leq M, \end{split} \eeq for any $t<T^\ast.$ Then to
complete the proof of Corollary \ref{cst}, we only need to show that
$T^\ast=\infty.$

In fact,  thanks to \eqref{p.0}, we can find some  $t_0\in
(0,T^\ast)$ such that $u(t_0)\in H^1(\R^2).$ Then for $t_0\leq
t<T^\ast,$ we get, by taking the $L^2$ inner product of the momentum
equation of \eqref{1.1} with $\p_tu$, that \beno
\begin{split}
\f12\|\na
u(t)\|_{L^2}^2+&\|\sqrt{\r}\p_tu\|_{L^2((t_0,t);L^2)}^2\\
=&\f12\|\na
u(t_0)\|_{L^2}^2-\int_{t_0}^t\int_{\R^2}\r
u\cdot\na u\ |\ \p_tu\,dx\,dt'\\
\leq &\f12\|\na
u(t_0)\|_{L^2}^2+\f{M^2}{2m}\int_{t_0}^t\|u\|_{L^4}^2\|\na
u\|_{L^4}^2\,dt'+\f{m}2\|\p_tu\|_{L^2((t_0,t);L^2)}^2,
\end{split}
\eeno which along with \eqref{p.0} and $\|a\|_{L^4}^2\leq
C\|u\|_{L^2}\|\na u\|_{L^2}$ implies that for any $\e>0$ and
$t<T^\ast,$ \beq \label{p.1} \|\na
u(t)\|_{L^2}^2+m\|\p_tu\|_{L^2((t_0,t);L^2)}^2 \leq \|\na
u(t_0)\|_{L^2}^2+\f{M^4}{4m^2\e}\int_{t_0}^t\|u\|_{L^2}^2\|\na
u\|_{L^2}^4\,dt'+\e\|\D u\|_{L^2((t_0,t);L^2)}^2. \eeq However, when
$\mu(\r)=1,$ we deduce from the property of linear Stokes system
that \beno
\begin{split}
\|\D u\|_{L^2((t_0,t);L^2)}^2&+\|\na\Pi\|_{L^2((t_0,t);L^2)}^2\leq
\|\r\p_tu\|_{L^2((t_0,t);L^2)}^2+\|\r u\cdot\na
u\|_{L^2((t_0,t);L^2)}^2\\
\leq&
M^2\|\p_tu\|_{L^2((t_0,t);L^2)}^2+\f{M^4}2\int_{t_0}^t\|u\|_{L^2}^2\|\na
u\|_{L^2}^4\,dt'+\f12\|\D u\|_{L^2((t_0,t);L^2)}^2,
\end{split}
\eeno which gives rise to \beq \label{p.2adf} \|\D
u\|_{L^2((t_0,t);L^2)}^2+\|\na\Pi\|_{L^2((t_0,t);L^2)}^2\leq
2M^2\|\p_tu\|_{L^2((t_0,t);L^2)}^2+{M^4}\int_{t_0}^t\|u\|_{L^2}^2\|\na
u\|_{L^2}^4\,dt'. \eeq

Summing up \eqref{p.1} with $2\e\times$\eqref{p.2adf}, we obtain
\beno
\begin{split}
\|\na u(t)\|_{L^2}^2+&m\|\p_tu\|_{L^2((t_0,t);L^2)}^2 +\e\bigl(\|\D
u\|_{L^2((t_0,t);L^2)}^2+\|\na\Pi\|_{L^2((t_0,t);L^2)}^2\bigr)\\
 &\leq \|\na
u(t_0)\|_{L^2}^2+4M^2\e\|\p_tu\|_{L^2((t_0,t);L^2)}^2+C_{m,M}\|u_0\|_{L^2}^2\int_{t_0}^t\|\na
u\|_{L^2}^4\,dt'.
\end{split}
\eeno Taking $\e=\f{m}{8M^2}$ in the above inequality and then
applying Gronwall's inequality leads to \beq \label{p.3}
\begin{split}\|&\na
u(t)\|_{L^2}^2+\f{m}2\|\p_tu\|_{L^2((t_0,t);L^2)}^2
+\f{m}{8M^2}\bigl(\|\D
u\|_{L^2((t_0,t);L^2)}^2+\|\na\Pi\|_{L^2((t_0,t);L^2)}^2\bigr)\\
&\leq \|\na u(t_0)\|_{L^2}^2\exp\bigl\{C_{m,M}\|u_0\|_{L^2}^2\|\na
u\|_{L^2((t_0,t);L^2)}^2\bigr\}\\
&\leq \|\na
u(t_0)\|_{L^2}^2\exp\bigl\{C_{m,M}\|u_0\|_{L^2}^4\bigr\}.
\end{split}
\eeq

On the other hand, we get, by a similar derivation of \eqref{a.3},
that \beq\label{p.4}
\begin{aligned}
\|u\|_{\widetilde L^1((t_0,t);\dot B^{2+\al}_{2,\infty})} \lesssim
&\|u(t_0)\|_{H^1} +\sup_j2^{j\al}\|[\dot\Delta_j\mathbb{P};
u\cdot\nabla]u\|_{L^1((t_0,t);L^2)}
\\&\qquad\qquad\quad
+\sup_j2^{j\al} \|[\dot\Delta_j\mathbb{P};\frac{1}{\rho}]\bigl(\D u
-\nabla\Pi\bigr)\|_{L^1((t_0,t);L^2)}.
\end{aligned}
\eeq The proof of \eqref{a.4} yields \beno
\sup_j2^{j\al}\|[\dot\Delta_j\mathbb{P};
u\cdot\nabla]u\|_{L^1((t_0,t);L^2)}\lesssim \|\na
u\|_{L^2((t_0,t);L^2)}^{2-\al}\|\D u\|_{L^2((t_0,t);L^2)}^\al. \eeno
And it follows form the proof of \eqref{a.5} that \beno
\sup_j2^{j\al} \|[\dot\Delta_j\mathbb{P};\frac{1}{\rho}]\bigl(\D u
-\nabla\Pi\bigr)\|_{L^1((t_0,t);L^2)}\lesssim \sqrt{t-t_0}\
\|\r\|_{L^\infty((t_0,t);\dB^\al_{\infty,\infty})}\|\D
u-\na\Pi\|_{L^2((t_0,t);L^2)}. \eeno Therefore thanks to \eqref{p.3}
and \eqref{p.4}, we conclude that \beq \label{p.5}
\begin{split}
\|u\|_{\widetilde L^1((t_0,t);\dot B^{2+\al}_{2,\infty})}\leq&
C(m,M,\|u(t_0)\|_{H^1})\bigl(1+\sqrt{t-t_0}\
\|\r\|_{L^\infty((t_0,t);\dB^\al_{\infty,\infty})}\bigr)\\
\leq& C(m,M,\|u(t_0)\|_{H^1})\Bigl\{1+\sqrt{t-t_0}\
\|\r(t_0)\|_{\dB^\al_{\infty,\infty}}\exp\bigl(C\|u\|_{L^1((t_0,t);\dB^2_{2,1})}\bigr)\Bigr\}.
\end{split}
\eeq

Now for any positive integer $N$, we get, by applying
\eqref{dyadica},  that \beno
\begin{split}
\|u\|_{L^1((t_0,t);\dB^{2}_{2,1})}\leq&\sum_{\ell\leq
0}2^\ell\|\dot\D_\ell\na u\|_{L^1((t_0,t);L^2)}+\sum_{0<\ell\leq
N}\|\dot\D_\ell\D u\|_{L^1((t_0,t);L^2)}\\
&+\sum_{\ell\geq
N}2^{2\ell}\|\dot\D_\ell u\|_{L^1((t_0,t);L^2)}\\
\lesssim& \sqrt{t-t_0}\ \|\na
u\|_{L^2((t_0,t);L^2)}+\sqrt{N(t-t_0)}\ \|\D
u\|_{L^2((t_0,t);L^2)}\\
&+2^{-N\al}\|u\|_{\widetilde L^1((t_0,t);\dot
B^{2+\al}_{2,\infty})},
\end{split}
\eeno which together with \eqref{p.0},\eqref{p.3} and \eqref{p.5}
implies \beno
\begin{split}
\|u\|_{L^1((t_0,t);\dB^{2}_{2,1})}\leq&
C(m,M,\|\r(t_0)\|_{\dB^\al_{\infty,\infty}},\|u(t_0)\|_{H^1})\Bigl\{1\\
&\qquad+\sqrt{t-t_0}\bigl(1+\sqrt{N}+2^{-N\al}
\exp\bigl(C\|u\|_{L^1((t_0,t);\dB^2_{2,1})}\bigr)\bigr)\Bigr\}.
\end{split}
\eeno Taking
$N=\bigl[\f{C}{\al}\|u\|_{L^1((t_0,t);\dB^{2}_{2,1})}\bigr]$   in
the above inequality results in \beno
\begin{split}
\|u\|_{L^1((t_0,t);\dB^{2}_{2,1})}\leq&
C(m,M,\|\r(t_0)\|_{\dB^\al_{\infty,\infty}},\|u(t_0)\|_{H^1})\Bigl\{1+\sqrt{t-t_0}\bigl(1+\sqrt{\|u\|_{L^1((t_0,t);\dB^{2}_{2,1})}}\bigr)\Bigr\}\\
\leq&
C(m,M,\|\r(t_0)\|_{\dB^\al_{\infty,\infty}},\|u(t_0)\|_{H^1})(1+t-t_0)+\f12\|u\|_{L^1((t_0,t);\dB^{2}_{2,1})},
\end{split}
\eeno from which, we infer \beq\label{p.6}
\|u\|_{L^1((t_0,t);\dB^{2}_{2,1})}\leq
C(m,M,\|\r(t_0)\|_{\dB^\al_{\infty,\infty}},\|u(t_0)\|_{H^1})(1+t-t_0).
\eeq With \eqref{p.6}, it is standard to prove that $T^\ast$ given
at the beginning of the proof equals $\infty.$ This completes the
proof of the corollary.
\end{proof}

\renewcommand{\theequation}{\thesection.\arabic{equation}}
\setcounter{equation}{0}

\section{Proof of Theorem \ref{thm1.2}}

The goal of this section is to present the proof of Theorem
\ref{thm1.2}. To prove the  existence part of Theorem \ref{thm1.2},
we need the following two technical lemmas:

\begin{lem}\label{lemq.9}
{Let $p,q\geq 1$ and $s\in \R$ satisfying $\f1p\leq\f12+\f1q$ and
$\max\bigl(-1,2(\f1q-1)\bigr)<s<1+\f2q.$ Let $v\in \dB^s_{q,2}\cap
\dot{H}^1(\R^2)$ be a solenoidal vector filed.
 Then one
has \beq \label{q.20}
\begin{split}
 \| [\dot\Delta_j\mathbb{P}; v\cdot\nabla]v\|_{L^p}\lesssim
d_j2^{j(1+\f2q-\f2p-s)}\|\na v\|_{L^2}\|v\|_{\dB^{s}_{q,2}}.
\end{split} \eeq}
\end{lem}

\begin{proof} We get, by using Bony's
decomposition \eqref{bony}, that \beq \label{q.7a}
[\dot\Delta_j\mathbb{P}; v\cdot\nabla]v=[\dot\Delta_j\mathbb{P};
T_v\cdot\nabla]v+\dot\D_j\mathbb{P}\bigl(T_{\na v} v+\cR(v,\na
v))-T_{\na\dot\D_jv}v-\cR(v,\na\dot\D_jv). \eeq It is easy to check
that \beno [\dot\Delta_j\mathbb{P};
T_v\cdot\nabla]v(x)=2^{2j}\sum_{|\ell-j|\leq
4}\int_0^1\int_{\R^2}h(2^jz)z\cdot \dot S_{\ell-1}\na v(x+(\theta-1)
z)\dot\D_\ell\na v(x-z)\,dz\,d\theta. \eeno We first deal with the
case when $1<p\leq 2$ in \eqref{q.20}. In this case, applying
H\"older's inequality and the property of the translation invariance
of the Lebesgue measure, we obtain
 \beno
 \begin{split} \|[\dot\Delta_j\mathbb{P};
T_v\cdot\nabla]v\|_{L^p}\leq& 2^{j}\sum_{|\ell-j|\leq
4}\int_0^1\int_{\R^2}|h_1(2^jz)|\| \dot S_{\ell-1}\na v(\cdot+\theta
z)\|_{L^{\bar{p}}}\|\dot\D_\ell\na v(\cdot-z)\|_{L^2}\,dzd\theta\\
\lesssim& 2^{-j}\sum_{|\ell-j|\leq 4}\|\dot S_{\ell-1}\na
v\|_{L^{\bar{p}}}\|\dot\D_\ell\na v\|_{L^2},\end{split} \eeno where
$h_1(z)=zh(z)$ and $\bar{p}$ satisfies $\f1p=\f12+\f1{\bar{p}}.$ As
$\f1p\leq\f12+\f1q$ and $s<1+\f2q,$ we deduce, from Lemma
\ref{lem2.1}, that \beno \|\dot S_{\ell-1}\na
v\|_{L^{\bar{p}}}\lesssim
\sum_{k\leq\ell-1}2^{2k(\f1q-\f1{\bar{p}})}\|\dot\D_k\na
v\|_{L^q}\lesssim c_\ell 2^{\ell[2(1+\f1q-\f1p)-s]}\|
v\|_{\dB^s_{q,2}} \eeno so that one has \beq \label{q.8}
\|[\dot\Delta_j\mathbb{P}; T_v\cdot\nabla]v\|_{L^p}\lesssim
d_j2^{j(1+\f2q-\f2p-s)}\|\na v\|_{L^2}\|v\|_{\dB^{s}_{q,2}}.  \eeq
Along the same line, we have \beq \label{q.9}
\begin{split} &\|\dot\D_j\mathbb{P}\bigl(T_{\na v}
v\bigr)\|_{L^p}\leq\sum_{|\ell-j|\leq 4}\|\dot S_{\ell-1}\na
v\|_{L^{\bar{p}}}\|\dot\D_\ell v\|_{L^2}\lesssim
d_j2^{j(1+\f2q-\f2p-s)}\|\na v\|_{L^2}\|v\|_{\dB^{s}_{q,2}}.
\end{split}\eeq
The same estimate holds for $T_{\na\dot\D_jv}v.$

Whereas for $1<q\leq 2,$ we get, by applying
$\dv\, v=0$ and Lemma
\ref{lem2.1}, that \beq\label{q.9a}
\begin{split}
\|\dot\D_j\mathbb{P}&(\cR(v,\na v))\|_{L^p}\lesssim
2^{j(3-\f2p)}\sum_{\ell\geq j-3}\|\dot\D_\ell
v\|_{L^2}\|\wt{\dot\D}_\ell v\|_{L^2}\\
 &\lesssim2^{j(3-\f2p)}\sum_{\ell\geq j-3} d_\ell
 2^{\ell(\f2q-s-2)}\|\na v\|_{L^2}\|v\|_{\dB^{s}_{q,2}}
 \lesssim d_j2^{j(1+\f2q-\f2p-s)}\|\na
 v\|_{L^2}\|v\|_{\dB^{s}_{q,2}},
\end{split}\eeq
and for $q\geq 2,$  we have \beq\label{q.9b} \begin{split}
\|\dot\D_j\mathbb{P}&(\cR(v,\na v))\|_{L^p}\lesssim
2^{j[2(1+\f1q-\f1p)]}\sum_{\ell\geq j-3}\|\dot\D_\ell
u\|_{L^{q}}\|\wt{\dot\D}_\ell v\|_{L^2}\\
&\lesssim 2^{j[2(1+\f1q-\f1p)]}\sum_{\ell\geq j-3}d_\ell
2^{-\ell(1+s)}\|\na v\|_{L^2}\|v\|_{\dB^{s}_{q,2}} \lesssim
d_j2^{j(1+\f2q-\f2p-s)}\|\na
 v\|_{L^2}\|v\|_{\dB^{s}_{q,2}},
\end{split}
\eeq  where we used the fact that $s>\max\bigl(-1,2(\f1q-1)\bigr).$
This together with \eqref{q.8} and \eqref{q.9} proves \eqref{q.20}
for $1<p\leq 2.$

The case when $2<p$ is much easier. Notice that \beno
 \|[\dot\Delta_j\mathbb{P};
T_v\cdot\nabla]v\|_{L^p} \lesssim 2^{-j}\sum_{|\ell-j|\leq 4}\|\dot
S_{\ell-1}\na v\|_{L^\infty}\|\dot\D\na v\|_{L^p}, \eeno and as
$s<1+\f2q,$ one has \beno \|\dot S_{\ell-1}\na
v\|_{L^\infty}\lesssim c_\ell
2^{\ell(1+\f2q-s)}\|v\|_{\dB^{s}_{q,2}}, \eeno so that \eqref{q.8}
holds for $p>2.$ The same estimate holds for
$\dot\D_j\mathbb{P}\bigl(T_{\na v} v\bigr)$ and $T_{\na\dot\D_jv}v.$
This together with \eqref{q.9a} and \eqref{q.9b}  completes the
proof of \eqref{q.20} for $2<p.$
\end{proof}

\begin{lem}\label{lemq.10}
{Let $p\geq 1$ and $s>-1.$ Let $v\in
\dB^{1+s}_{p,1}\cap\dB^{2+s}_{p,1}\cap H^1(\R^2)$ be a solenoidal
vector filed.
 Then one
has \beq \label{q.9c}
\begin{split}
 \|v\cdot\na v\|_{\dot{B}^s_{p,1}}\lesssim
\|v\|_{L^2}\|v\|_{\dB^{2+s}_{p,1}}+\|\na
v\|_{L^2}\|v\|_{\dB^{1+s}_{p,1}}.
\end{split} \eeq}
\end{lem}

\begin{proof}
Bony's decomposition \eqref{bony} for $v\cdot\na v$ reads  \beno
v\cdot\na v=T_v\cdot\na v+T_{\na v}\cdot v+\cR(v, \na v). \eeno
Applying Lemma \ref{lem2.1} yields \beno
\begin{split}\|\dot\D_j(T_v\cdot\na v)\|_{L^p}\lesssim&
\sum_{|\ell-j|\leq 4}\|\dot S_{\ell-1}v\|_{L^2}\|\dot\D_\ell \na
v\|_{L^{\bar{p}}} \lesssim
d_j2^{-js}\|v\|_{L^2}\|v\|_{\dB^{2+s}_{p,1}},
\end{split}
\eeno where
 $\bar{p}$ satisfies $\f1p=\f12+\f1{\bar{p}}.$

 A similar procedure gives rise to
\beno \begin{split}\|\dot\D_j(T_{\na v}\cdot v)\|_{L^p}\lesssim&
\sum_{|\ell-j|\leq 4}\|\dot S_{\ell-1}\na v\|_{L^2}\|\dot\D_\ell
v\|_{L^{\bar{p}}} \lesssim d_j2^{-js}\|\na
v\|_{L^2}\|v\|_{\dB^{1+s}_{p,1}}.
\end{split}
\eeno Finally as $s>-1,$ by applying $\dv\, v=0$ and  Lemma
\ref{lem2.1}, we get
$$
\|\dot\D_j(\cR(v,\na v)\|_{L^p}\lesssim
2^j \sum_{\ell\geq j-3}\|\dot\D_{\ell} v\|_{L^2}\|\wt{\dot\D}_\ell
v\|_{L^{\bar{p}}} \lesssim d_j2^{-js}\|
v\|_{L^2}\|v\|_{\dB^{2+s}_{p,1}}.
$$
This completes the proof of
\eqref{q.9c}.
\end{proof}

\begin{proof}[Proof  to the existence part of  Theorem
\ref{thm1.2}]  Given initial data $(\r_0,u_0)$ satisfying the
assumptions of Theorem \ref{thm1.2}, we deduce  from \cite{abidi}
that \eqref{1.1} has a  local  solution $(\r,u)$ on $[0, T^\ast)$ so
that \beq\label{loc} \begin{split} & \r-1 \in {\mathcal{C}_{b}}([0,
T]; \dot B^{\frac{2}{p}}_{p,1}(\R^2)), \quad u\in
{\mathcal{C}_{b}}([0, T]; \dot B^{-1+\frac{2}{p}}_{p,1}(\R^2)) \cap
L^1([0,T]; \,\dot B^{1+\frac{2}{p}}_{p,1}(\R^2))\quad\mbox{and}\\
&\f12\|\sqrt{\r}u(T)\|_{L^2}^2+\int_0^T\int_{\R^2}\mu(\r)d(u):d(u)\,dx\,dt'=\f12
\|\sqrt{\r_0}u_0\|_{L^2}^2 \end{split} \eeq for any $T<T^\ast.$

In order to prove that $T^\ast=\infty$ under the nonlinear smallness
condition \eqref{th1.2a}, we write \beq\label{q.2}
\pa_tu+(u\cdot\nabla)u-\Delta
u+\nabla\Pi=(1-\r)\pa_tu+(1-\r)(u\cdot\nabla)u +\dv
\bigl[2(\mu(\rho)-1) d\bigr], \eeq from which and similar derivation
of \eqref{a.3}, we deduce
 for $p\in (1,4)$ and $t\in (0,T^\ast)$ that
\beq\label{q.2a}
\begin{split}
\|u\|_{\widetilde L^\infty_t(\dot B^{-1+\frac{2}{p}}_{p,1})}
&+c\|u\|_{L^1_t(\dot B^{1+\frac{2}{p}}_{p,1})}
 \leq \|u_0\|_{\dot
B^{-1+\frac{2}{p}}_{p,1}} +\sum_{j\in\Z}2^{j(-1+\frac{2}{p})}
\|[\dot\Delta_j\mathbb{P};u\cdot\nabla]u\|_{L^1_t(L^p)}\\
&+\bigl\|\bigl\{(1-\r)\pa_tu+(1-\r)(u\cdot\nabla)u +\dv
\bigl[2(\mu(\rho)-1) d\bigr]\bigr\}\bigr\|_{L^1_t(\dot
B^{1+\frac{2}{p}}_{p,1})}.
\end{split}
\eeq Applying product laws in Besov spaces (\cite{BCD}) yields
\beq\label{pqur}
\begin{aligned}
\|u&\|_{\widetilde L^\infty_t(\dot B^{-1+\frac{2}{p}}_{p,1})}
+c\|u\|_{L^1_t(\dot B^{1+\frac{2}{p}}_{p,1})}\leq \|u_0\|_{\dot
B^{-1+\frac{2}{p}}_{p,1}} +\sum_{j\in\Z}2^{j(-1+\frac{2}{p})}
\|[\dot\Delta_j\mathbb{P};u\cdot\nabla]u\|_{L^1_t(L^p)}\\
&\qquad+C\|\r-1\|_{\widetilde L^\infty_t(\dot
B^{\frac{2}{p}}_{p,1})} \Bigl\{\|\pa_tu\|_{L^1_t(\dot
B^{-1+\frac{2}{p}}_{p,1})} +\|(u\cdot\nabla)u\|_{L^1_t(\dot
B^{-1+\frac{2}{p}}_{p,1})} +\|u\|_{L^1_t(\dot
B^{1+\frac{2}{p}}_{p,1})}\Bigr\}.
\end{aligned}
\eeq However, it following Lemma \ref{lemq.9} and \eqref{loc} that
for $1<p<4,$ \beq\label{q.3}
\begin{split}
\|[\dot\Delta_j\mathbb{P};u\cdot\nabla]u\|_{L^1_t(L^p)}
\lesssim&d_j2^{j(1-\f2p)}\|\na u\|_{L^2_t(L^2)}^2 \lesssim
d_j2^{j(1-\f2p)} \|u_0\|_{L^2}^2, \end{split} \eeq and Lemma
\ref{lemq.10} together with \eqref{loc} ensures that \beq\label{q.4}
\begin{split}
\|(u\cdot\nabla)u\|_{L^1_t(\dot B^{-1+\frac{2}{p}}_{p,1})} \lesssim
& \|u\|_{L^\infty_t(L^2)}\|u\|_{L^1_t(\dB^{1+\f2p}_{p,1})}+\|\na
u\|_{L^2_t(L^2)}\|u\|_{L^2_t(\dB^{\f2p}_{p,1})}\\
\lesssim & \|u_0\|_{L^2}\bigl(\|u\|_{\widetilde L^\infty_t(\dot
B^{-1+\frac{2}{p}}_{p,1})}+ \|u\|_{L^1_t(\dot
B^{1+\frac{2}{p}}_{p,1})}\bigr).\end{split} \eeq Substituting
\eqref{q.3} and \eqref{q.4} into \eqref{pqur} results in \beq
\label{q.5}
\begin{split}
\|u&\|_{\widetilde L^\infty_t(\dot B^{-1+\frac{2}{p}}_{p,1})}
+c\|u\|_{L^1_t(\dot B^{1+\frac{2}{p}}_{p,1})}
 \leq  \|u_0\|_{\dot
B^{-1+\frac{2}{p}}_{p,1}} +C\|u_0\|_{L^2}^2\\
&\quad+C(1+\|u_0\|_{L^2})\|\r-1\|_{\widetilde L^\infty_t(\dot
B^{\frac{2}{p}}_{p,1})} \Bigl\{\f{\|\pa_tu\|_{L^1_t(\dot
B^{-1+\frac{2}{p}}_{p,1})}}{1+\|u_0\|_{L^2}} +\|u\|_{\widetilde
L^\infty_t(\dot B^{-1+\frac{2}{p}}_{p,1})} +\|u\|_{L^1_t(\dot
B^{1+\frac{2}{p}}_{p,1})}\Bigr\}.
\end{split}
\eeq Whereas we infer from \eqref{q.2} and \eqref{q.4} that \beq
\label{q.6}
\begin{split} \|\p_tu&\|_{\widetilde L^\infty_t(\dot
B^{-1+\frac{2}{p}}_{p,1})} +\|\na\Pi\|_{L^1_t(\dot
B^{-1+\frac{2}{p}}_{p,1})}
 \leq  \|u_0\|_{\dot
B^{-1+\frac{2}{p}}_{p,1}}+C\|u_0\|_{L^2}\bigl(\|u\|_{\widetilde
L^\infty_t(\dot B^{-1+\frac{2}{p}}_{p,1})} +\|u\|_{L^1_t(\dot
B^{1+\frac{2}{p}}_{p,1})}\bigr)\\
&+C(1+\|u_0\|_{L^2})\|\r-1\|_{\widetilde L^\infty_t(\dot
B^{\frac{2}{p}}_{p,1})} \Bigl\{\f{\|\pa_tu\|_{L^1_t(\dot
B^{-1+\frac{2}{p}}_{p,1})}}{1+\|u_0\|_{L^2}} +\|u\|_{\widetilde
L^\infty_t(\dot B^{-1+\frac{2}{p}}_{p,1})} +\|u\|_{L^1_t(\dot
B^{1+\frac{2}{p}}_{p,1})}\Bigr\}.
\end{split}
\eeq For $\e$ sufficiently small, we denote \beq\label{q.6a}
\frak{A}(t)\eqdefa\|u\|_{\widetilde L^\infty_t(\dot
B^{-1+\frac{2}{p}}_{p,1})}+\|u\|_{L^1_t(\dot
B^{1+\frac{2}{p}}_{p,1})}
+\f{\e}{1+\|u_0\|_{L^2}}\bigl(\|\pa_tu\|_{L^1_t(\dot
B^{-1+\frac{2}{p}}_{p,1})}+\|\nabla\Pi\|_{L^1_t(\dot
B^{-1+\frac{2}{p}}_{p,1})}\bigr). \eeq Then by summing \eqref{q.5}
with $\f{\e}{1+\|u_0\|_{L^2}}\times$\eqref{q.6}, and  using the
following standard estimate on transport equation \cite{BCD} that
$$
\|\r-1\|_{\widetilde L^\infty_t(\dot B^{\frac{2}{p}}_{p,1})} \leq
\|\r_0-1\|_{\dot B^{\frac{2}{p}}_{p,1}} \exp\bigl(C\|u\|_{L^1_t(\dot
B^{1+\frac{2}{p}}_{p,1})}\bigr),
$$
we obtain \beq\label{q.7} \frak{A}(t) \le C\bigl\{\|u_0\|_{\dot
B^{-1+\frac{2}{p}}_{p,1}}+\|u_0\|_{L^2}^2
+\frak{A}(t)(1+\|u_0\|_{L^2})\|\r_0-1\|_{\dot B^{\frac{2}{p}}_{p,1}}
e^{C\frak{A}(t)}\bigr\}. \eeq In particular if we take $\e_0$ to be
sufficiently small and $C_0$ to be sufficiently large in
\eqref{th1.2a}, one has
$$
C\|\r_0-1\|_{\dot B^{\frac{2}{p}}_{p,1}}\exp\{2C(\|u_0\|_{\dot
B^{-1+\frac{2}{p}}_{p,1}}+\|u_0\|_{L^2}^2)\} \le \frac{1}{4},
$$
which together with \eqref{q.7} ensures that
$$
\frak{A}(t) \leq 2C\bigl(\|u_0\|_{\dot
B^{-1+\frac{2}{p}}_{p,1}}+\|u_0\|_{L^2}^2\bigr)\quad \mbox{for
all}\quad t\in (0, T^\ast).
$$
This in turn proves that $T^\ast=\infty$ under the assumption of
\eqref{th1.2a}, which completes the proof to the existence part of
Theorem \ref{thm1.2}.\end{proof}

To prove the uniqueness part of Theorem \ref{thm1.2}, we first
recall  the following Lemma  from \cite{DM1} (see Proposition 2.1 of
\cite{DM1}):

\begin{lem} \label{lemq.1} {\sl Let $v_0\in \dB^s_{p,1}(\R^2)$ and
$f\in L^1((0,T);\dB^s_{p,1}(\R^2))$ with $p\in [1,\infty]$ and
$s\in\R.$ Let $g, R$ satisfy $\na g\in
L^1((0,T);\dB^s_{p,1}(\R^2)),$  $\p_tR\in
L^1((0,T);\dB^s_{p,1}(\R^2))$ and that the compatibility condition
$g|_{t=0}=\dv v_0$ on $\R^2$. Then the system
\beno\left\{\begin{array}{l} \displaystyle \pa_tv-\Delta v
+\grad Q=f, \\
\displaystyle\dv\, v=g=\dv R, \\
\displaystyle v|_{t=0}=v_0
\end{array}\right.
\eeno has a unique solution $(v, \na Q)$ so that
\beq\label{q.1}\begin{split} \|v\|_{\widetilde L^\infty_t(\dot
B^{s}_{p,1})} +&\|(\partial_tv,\nabla^2v,\nabla Q)\|_{L^1_t(\dot
B^{s}_{p,1})}\\
& \lesssim \|v_0\|_{\dot B^{s}_{p,1}} +\|f\|_{L^1_t(\dot
B^{s}_{p,1})} +\|\na g\|_{L^1_t(\dot B^{s}_{p,1})}
+\|\pa_tR\|_{L^1_t(\dot B^{s}_{p,1})}. \end{split} \eeq}
\end{lem}

\begin{proof} [Proof to the uniqueness part of Theorem \ref{thm1.2}]
This part will  essentially follow the Lagrangian idea from
\cite{DM1}. Yet in \cite{DM1}, the initial density belongs to the
multiplier space $\cM(B^{\f{d}p}_{p,1}(\R^d))$ of
$B^{\f{d}p}_{p,1}(\R^d)$ for $1<p<2d,$ and the viscosity coefficient
$\mu(\r)$ equals some positive constant. Here the initial density
$\r_0$ belongs to $B^{\f2p}_{p,1}(\R^2),$ which is a subspace of
$\cM(B^{\f{2}p}_{p,1}(\R^2)),$ but the viscous coefficient $\mu(\r)$
depends on $\r.$ We remark that our proof here works in general
space dimensions, although we only present here the 2-D case.

Let $(\r,u,\na\Pi)$ be a global solution of \eqref{1.1} obtained in
Theorem \ref{thm1.2}. Then as $u\in L^\infty(\R^+;Lip(\R^2)),$ we
can define the trajectory $X(t,y)$ of $u(t,x)$ by
$$\partial_t X(t,y)=u(t,X(t,y)),\qquad X(0,y)=y,$$
which leads to the following relation between the Eulerian
coordinates $x$ and the Lagrangian coordinates $y$: \beq\label{u.14}
x=X(t,y)=y+\int_0^tu(\tau, X(\tau,y))d\tau. \eeq Moreover, we can
take $T$ so small that \beq\label{u.15} \int_0^T\|\na
u(t,\cdot)\|_{L^\infty}\leq \f12. \eeq Then for $t\leq T,$ $X(t,y)$
is invertible with respect to $y$ variables, and we denote
$Y(t,\cdot)$ to be its inverse mapping. Let \beno
\bar{u}(t,y)\eqdefa u(t,x)=u(t,X(t,y))\quad\mbox{and}\quad
\bar{\Pi}(t,y)\eqdefa \Pi(t,X(t,y)). \eeno Then similar to
\cite{DM1}, one has \beq \label{u.15ag} \bar{u}\in
\wt{L}^\infty(\R^+;\dB^{-1+\f2p}_{p,1}(\R^2)) \quad\mbox{and}\quad
\p^2\bar u, \p_t\bar u, \na\bar{\Pi}\in
L^1(\R^+;\dB^{-1+\f2p}_{p,1}(\R^2)), \eeq and \beq\label{u.16}
\begin{split} &\partial_t
\bar{u}(t,y)=\partial_t u(t,x)+ u(t,x)\nabla u(t,x),\\
&\partial_{x_i} u^j(t,x)= \partial_{y_k}
\bar{u}^j(t,y)\partial_{x_i} y^k \quad\mbox{for}\quad x=X(t,y),\
y=Y(t,x) \end{split} \eeq so that let $A(t,y)\eqdefa (\na
X(t,y))^{-1}=\na_x Y(t,x),$ we have \beq\label{u.17}
 \nabla_x u(t,x)= A(t,y)^T\nabla_y \bar{u}(t,y)\quad\mbox{ and}\quad
\dive u(t,x)=\dive( A(t,y) \bar{u}(t,y)), \eeq and $(\bar{u},
\na_y\bar{\Pi})$ solves
\begin{equation}\label{u.18}
\left\{\begin{array}{l} \displaystyle
\rho_{0}\pa_t\bar{u}-\dv_y(\mu(\rho_{0})d(\bar{u}))
+\grad_y\bar{\Pi}=\dv\bigl(\mu(\rho_{0})(A A^T-Id)d(\bar{u})\bigr)+(Id-A)^T\na_y\bar{\Pi}, \\
\displaystyle \dv\,\bar{u} = \dv\bigl((Id- A)\bar{u}\bigr).
\end{array}\right.
\end{equation}

Now let $(\r_i, u_i, \na\Pi_i),$ $i=1,2,$ be two solutions of
\eqref{1.1} which satisfy the regularity properties listed in
Theorem \ref{thm1.2}. Let $(\bar{u}_i, A_i, \bar{\Pi}_i),$ $i=1,2,$
be defined from \eqref{u.14} to \eqref{u.16}, we denote
$$
(\delta A,\delta\bar u,\nabla\delta\bar\Pi) \eqdefa (A_2-A_1,\bar
u_2-\bar u_1,\nabla\bar\Pi_2-\nabla\bar\Pi_1).
$$
Then it follows from \eqref{u.18} that the system for $(\delta \bar
u, \nabla\delta\bar \Pi)$ reads \beq\label{u.19}
\left\{\begin{array}{l} \displaystyle \pa_t\delta\bar u-\D_y\d\bar u
+\grad_y\delta\bar\Pi=\d\bar F, \\
\displaystyle \dv_y\,\delta\bar u = \na\delta\bar u:(Id-A_2)-\na
u_1:\delta A =\dv_y\bigl((Id-A_2)\delta \bar u-\delta A \bar
u_1\bigr),\\
\displaystyle \d\bar{u}|_{t=0}=0,
\end{array}\right.
\eeq where
$$
\begin{aligned}
\d\bar F=&(1-\r_0)\pa_t\delta\bar
u+\dv_y[(\mu(\rho_{0})-1)\na_y\delta \bar u]- \delta
A^T\nabla_y\bar\Pi_1+ (Id-A_2)^T\grad_y\delta\bar\Pi
\\&
+\dv_y\bigl\{ \mu(\rho_{0})[(A_2 A_2^T-Id) d(\delta \bar u) + (A_2
A_2^T-A_1 A_1^T) d(\bar u_1)]\bigr\}.
\end{aligned}
$$
We first deduce from product laws in Besov spaces (\cite{BCD}) that
\beno \begin{split} \bigl\|(1-\r_0)\pa_t\delta\bar
u+&\dv_y[(\mu(\rho_{0})-1)\na_y\delta \bar
u]\bigr\|_{L^1_t(\dB^{-1+\f2p}_{p,1})}\\
\lesssim&
\|(1-\r_0)\|_{\dB^{\f2p}_{p,1}}\|\pa_t\delta\bar
u\|_{L^1_t(\dB^{-1+\f2p}_{p,1})}+\|(\mu(\rho_{0})-1)\|_{\dB^{\f2p}_{p,1}}\|\na_y\delta
\bar u\|_{L^1_t(\dB^{\f2p}_{p,1})}\\
\lesssim &\|(1-\r_0)\|_{\dB^{\f2p}_{p,1}}\bigl(\|\pa_t\delta\bar
u\|_{L^1_t(\dB^{-1+\f2p}_{p,1})}+\|\delta \bar
u\|_{L^1_t(\dB^{1+\f2p}_{p,1})}\bigr).
\end{split}
\eeno

Before going further, we recall from \cite{DM1,DM2} that under the
assumption of \eqref{u.15}, one has \beq\label{u.19ad}\begin{split}
& \delta A(t)=\big(\int_0^t\na \delta \bar u(\tau)\,d\tau\big) \cdot
\Bigl(\sum_{k\geq1}\sum_{0\leq j\leq k}C_1^j(t)C_2^{k-1-j}(t)\Bigr),\\
&A_i(t)-Id=\sum_{k\geq1}(-1)^k(C_i(t))^k \qquad\mbox{with}\qquad
C_i(t)\eqdefa\int_0^t\na\bar u_i(\tau)\,d\tau.
\end{split}
\eeq Thanks to \eqref{u.19ad}, for $p\in (1,4),$ we get, by applying
product laws in Besov spaces (\cite{BCD}), that
$$
\begin{aligned}
\bigl\|-\delta
A^T\nabla\bar\Pi_1&+(Id-A_2)^T\grad\delta\bar\Pi\bigr\|_{L^1_t((\dB^{-1+\f2p}_{p,1})}\\
&\lesssim \|\delta
A\|_{L^\infty_t(\dB^{\f2p}_{p,1})}\|\nabla\bar\Pi_1\|_{L^1_t(\dB^{-1+\f2p}_{p,1})}
+\|Id-A_2\|_{L^\infty_t(\dB^{\f2p}_{p,1})}\|\grad\delta\bar\Pi\|_{L^1_t(\dB^{-1+\f2p}_{p,1})}
\\&
\lesssim
\|\nabla\bar\Pi_1\|_{L^1_t(\dB^{-1+\f2p}_{p,1})}\|\nabla\delta \bar
u\|_{L^1_t(\dB^{\f2p}_{p,1})} +\|\na\bar
u_2\|_{L^1_t(\dB^{\f2p}_{p,1})}
\|\grad\delta\bar\Pi\|_{L^1_t(\dB^{-1+\f2p}_{p,1})}.
\end{aligned}
$$
To deal with $\dv\bigl\{\mu(\r_0)(A_2A_2^T-Id)d(\delta\bar
u)\bigr\},$ we write
$$
\dv\bigl\{\mu(\r_0)(A_2A_2^T-Id)d(\delta\bar u)\bigr\}
=\dv\bigl\{\mu(\r_0)[(A_2-Id)(A_2-Id)^T+A_2-Id+(A_2-Id)^T]d(\delta
\bar u)\bigr\},
$$
from which, we infer
$$
\begin{aligned}
\bigl\|\dv\bigl\{&\mu(\r_0)(A_2A_2^T-Id)d(\delta\bar
u)\bigr\}\bigr\|_{L^1_t(\dB^{-1+\f2p}_{p,1})} \\
\lesssim&
\bigl(1+\|\mu(\r_0)-1\|_{\dB^{\f2p}_{p,1}}\bigr)\bigl(1+\|A_2-Id\|_{L^\infty_t(\dB^{\f2p}_{p,1})}\bigr)\|A_2-Id\|_{L^\infty_t(\dB^{\f2p}_{p,1})}
\|\nabla\delta\bar u\|_{L^1_t(\dB^{\f2p}_{p,1})}
\\\lesssim&
\bigl(1+\|\na\bar u_2\|_{L^1_t(\dB^{\f2p}_{p,1})}\bigr)\|\na\bar
u_2\|_{L^1_t(\dB^{\f2p}_{p,1})}\|\delta\bar
u\|_{L^1_t(\dB^{1+\f2p}_{p,1})}.
\end{aligned}
$$
Similar estimate holds for
$\dv\bigl\{\mu(\r_0)(A_2A_2^T-A_1A_1^T)d(\bar u_1)\bigr\}.$ As a
consequence, we obtain \beq \label{u.20} \begin{split} \|\d\bar
F\|_{L^1_t(\dB^{-1+\f2p}_{p,1})}\lesssim&
\bigl(\eta_1(t)+\|(1-\r_0)\|_{\dB^{\f2p}_{p,1}}\bigr)\\
&\qquad\times\Bigl(\|\pa_t\delta\bar
u\|_{L^1_t(\dB^{-1+\f2p}_{p,1})}+\|\delta \bar
u\|_{L^1_t(\dB^{1+\f2p}_{p,1})}+\|\grad\delta\bar\Pi\|_{L^1_t(\dB^{-1+\f2p}_{p,1})}\Bigr),
\end{split}
\eeq with $\lim_{t\longrightarrow0}\eta_1(t)=0.$

On the other hand, we deduce from \eqref{u.19} and \eqref{u.19ad}
that \beq\label{u.21}
\begin{split}
\|\dv \d\bar u\|_{L^1_t(\dB^{\f2p}_{p,1})} \lesssim &
\|\na\delta\bar
u\|_{L^1_t(\dB^{\f2p}_{p,1})}\|A_2-Id\|_{L^\infty_t(\dB^{\f2p}_{p,1})}
+\|\na\bar u_1\|_{L^1_t(\dB^{\f2p}_{p,1})}\|\d
A\|_{L^\infty_t(\dB^{\f2p}_{p,1})}\\
\lesssim& \bigl(\|\na\bar u_1\|_{L^1_t(\dB^{\f2p}_{p,1})}+\|\na\bar
u_2\|_{L^1_t(\dB^{\f2p}_{p,1})}\bigr)\|\delta\bar
u\|_{L^1_t(\dB^{1+\f2p}_{p,1})}.
\end{split}
\eeq Along the same line, we get
$$
\begin{aligned}
\bigl\|\pa_t\bigl((Id-A_2)&\delta \bar{u}-\delta A
\bar{u}^1\bigr)\bigr\|_{L^1_t(\dB^{-1+\f2p}_{p,1})}\\ \lesssim&
\|\na \bar{u}_2\delta\bar
u\|_{L^1_t(\dB^{-1+\f2p}_{p,1})}+\|(Id-A_2)\pa_t\delta\bar
u\|_{L^1_t(\dB^{-1+\f2p}_{p,1})}+\|\na\delta \bar u\bar
u_1\|_{L^1_t(\dB^{-1+\f2p}_{p,1})}
\\&
+\Bigl\|\int_0^{t}|\na\delta \bar u\,dt'||\na\bar U_{1,2}| |\bar
u_1|\Bigr\|_{L^1_t(\dB^{-1+\f2p}_{p,1})}
+\Bigl\|\int_0^{t}|\na\delta u|\,dt'|\pa_t\bar
u_1|\Bigr\|_{L^1_t(\dB^{-1+\f2p}_{p,1})}
\\
\lesssim&\|\na \bar u_2\|_{L^1_t(\dB^{\f2p}_{p,1})}\|\delta \bar
u\|_{L^\infty_t(\dB^{-1+\f2p}_{p,1})}
+\|Id-A_2\|_{L^\infty_t(\dB^{\f2p}_{p,1})}\|\pa_t\delta \bar
u\|_{L^1_t(\dB^{-1+\f2p}_{p,1})}
\\&
+\|\na\delta\bar u\|_{L^2_t(\dB^{-1+\f2p}_{p,1})}\|\bar
u_1\|_{L^2_t(\dB^{\f2p}_{p,1})} +\|\na\delta \bar
u\|_{L^1_t(\dB^{\f2p}_{p,1})} \|\na\bar
U_{1,2}\|_{L^2_t(\dB^{-1+\f2p}_{p,1})}\|\bar
u_1\|_{L^2_t(\dB^{\f2p}_{p,1})}
\\&+\|\na \delta\bar u\|_{L^1_t(\dB^{\f2p}_{p,1})} \|\pa_t\bar
u_1\|_{L^1_t(\dB^{-1+\f2p}_{p,1})},
\end{aligned}
$$ which implies
\beq\label{u.22}\begin{split} \bigl\|\pa_t\bigl((Id-A_2)\delta
\bar{u}-&\delta A \bar{u}^1\bigr)\bigr\|_{L^1_t(\dB^{-1+\f2p}_{p,1})}\\
&\lesssim \eta_2(t)\Bigl(\|\delta \bar
u\|_{L^\infty_t(\dB^{-1+\f2p}_{p,1})}+\|\delta \bar
u\|_{L^1_t(\dB^{1+\f2p}_{p,1})} +\|\p_t\d \bar
u\|_{L^1_t(\dB^{-1+\f2p}_{p,1})}\Bigr), \end{split} \eeq where
$\bar{U}_{1,2}$ denotes component of either $\bar u_1$ or $\bar
u_2,$ and $\lim_{t\longrightarrow0}\eta_2(t)=0.$

Thanks to Lemma \ref{lemq.1}, we get, by summing up \eqref{u.20} to
\eqref{u.22}, that \beq\label{u.23} \begin{split}
 \|\delta\bar
u&\|_{\wt{L}^\infty_t(\dB^{-1+\f2p}_{p,1})} +\|(\pa_t\delta \bar u,
\na^2\delta \bar
u,\nabla\delta\bar\Pi)\|_{L^1_t(\dB^{-1+\f2p}_{p,1})} \\
&\leq C( \eta(t)+\|\r_0-1\|_{\dB^{\f2p}_{p,1}})\Bigl(\|\delta\bar
u\|_{\wt{L}^\infty_t(\dB^{-1+\f2p}_{p,1})} +\|(\pa_t\delta \bar u,
\na^2\delta \bar
u,\nabla\delta\bar\Pi)\|_{L^1_t(\dB^{-1+\f2p}_{p,1})}\Bigr),
\end{split}
\eeq
 for some positive function $\eta(t)$ satisfying $\lim_{t\to
 0}\eta(t)=0.$
Uniqueness part of Theorem \ref{thm1.2} on a sufficiently small time
interval $[0,t_1]$ follows \eqref{u.23}.  The whole time uniqueness
then can be obtained by a bootstrap method. This completes the proof
of Theorem \ref{thm1.2}.
\end{proof}

\renewcommand{\theequation}{\thesection.\arabic{equation}}
\setcounter{equation}{0}

\section{Proof of Theorem \ref{thm1.3}}

In this section, we shall present the proof of  Theorem \ref{thm1.3}
by  following  the same line to that of Theorem \ref{thm1.2}. For
this, we first recall the following lemma from \cite{abidi0}:

\begin{lem}\label{lemh.0}[Lemma 2.1 of \cite{abidi0}]
{\sl Let $s<\f2p, (p,r)\in [1,\infty]^2$ be such that
$s+2\inf\bigl(\f1p,\f1{p'}\bigr)>0.$ Let $a\in
\dot B^{\f2p}_{p,\infty}\cap L^\infty(\R^2)$ and $b\in
\dB^s_{p,r}(\R^2).$ We denote
$\la(s)\eqdefa\f{s+2\inf\bigl(\f1p,\f1{p'}\bigr)}{|s|+s+2\inf\bigl(\f1p,\f1{p'}\bigr)},$
then \beno \Vert ab\Vert_{\dot B^{s}_{p,r}} \lesssim \Vert
b\Vert_{\dot B^{s}_{p,r}} \left\{\begin{array}{ll} \Vert
a\Vert_{L^\infty} +\Vert a\Vert_{L^\infty}^{1-\max(0,s)\frac{p}{2}}
\Vert a\Vert_{\dot B^{\frac{2}{p}}_{p,\infty}}^{\max(0,s)\frac{p}{2}}
+\Vert a\Vert_{L^\infty}^{\lambda(s)} \Vert
a\Vert_{\dot B^{\frac{2}{p}}_{p,\infty}}^{1-\lambda(s)}
&\textrm{if $s\neq0$}\\
\Vert a\Vert_{L^\infty} \ln\big(e+\Vert
a\Vert_{\dot B^{\frac{2}{p}}_{p,\infty}}{\Vert
a\Vert^{-1}_{L^\infty}}\big) &\textrm{if $s=0.$}
\end{array}\right.
 \eeno}
\end{lem}

\begin{lem}\label{lemh.1}
{\sl Let $p\in (1,\infty), 0<\e<\f2p.$ Then for $a\in
\dB^{\f2p}_{p,\infty}\cap \dB^{\f2p+\e}_{p,\infty}\cap
L^\infty(\R^2)$ and $b\in \dB^{\f2p-\e}_{p,\infty}\cap
\dB^{\f2p}_{p,\infty}(\R^2),$  one has \beq\label{h.6}
\begin{split}
&\|ab\|_{\dB^{\f2p}_{p,1}}\lesssim
\|a\|_{L^\infty}\|b\|_{\dB^{\f2p}_{p,1}}+\|a\|_{L^\infty}^{\f{p\e}{2+p\e}}\|a\|_{\dB^{\f2p+\e}_{p,\infty}}^{\f2{2+p\e}}
\|b\|_{\dB^{\f2p-\e}_{p,1}}^{\f2{2+p\e}}\|b\|_{\dB^{\f2p}_{p,1}}^{\f{p\e}{2+p\e}},\\
&\|ab\|_{\dB^{\f2p-\e}_{p,1}}\lesssim
\|a\|_{L^\infty}\|b\|_{\dB^{\f2p-\e}_{p,1}}+\|a\|_{L^\infty}^{\f{2\e}{p}}\|a\|_{\dB^{\f2p}_{p,\infty}}^{1-\f{2\e}{p}}
\|b\|_{\dB^{\f2p-\e}_{p,1}}.
\end{split}
\eeq }
\end{lem}

\begin{proof}  Bony's decomposition \eqref{bony} for $ab$ reads
\beno ab=T_ab+T_ba+\cR(a,b). \eeno It follows from para-product
estimate that \beq \label{h.7}
\begin{split}
&\|T_ab+\cR(a,b)\|_{\dB^{\f2p}_{p,1}}\lesssim
\|a\|_{L^\infty}\|b\|_{\dB^{\f2p}_{p,1}}
\quad\mbox{and}\quad\|T_ab+\cR(a,b)\|_{\dB^{\f2p-\e}_{p,1}}\lesssim
\|a\|_{L^\infty}\|b\|_{\dB^{\f2p-\e}_{p,1}}.
\end{split}
\eeq To deal with $T_ba,$ for any integer $M>0,$ we write \beno
\begin{split}
\|\dD_j(T_ba)\|_{L^p}\lesssim \sum_{|\ell-j|\leq 4}\Bigl\{\|\dD_\ell
a\|_{L^p}\sum_{k\leq\ell-M}\|\dD_k b\|_{L^\infty}+\|\dD_\ell
a\|_{L^\infty}\sum_{\ell-M<k\leq\ell}\|\dD_k b\|_{L^p}\Bigr\},
\end{split}
\eeno so that \beno
\begin{split}
\|T_ba\|_{\dB^{\f2p}_{p,1}}\lesssim &\sum_{\substack{\ell\in\Z\\
k\leq\ell-M}}2^{(k-\ell)\e}d_k\|a\|_{\dB^{\f2p+\e}_{p,\infty}}\|b\|_{\dB^{\f2p-\e}_{p,1}}
+\sum_{\substack{\ell\in\Z\\\ell-M<
k\leq\ell}}2^{(\ell-k)\f2p}d_k\|a\|_{L^\infty}\|b\|_{\dB^{\f2p}_{p,1}}\\
\lesssim &
2^{-M\e}\|a\|_{\dB^{\f2p+\e}_{p,\infty}}\|b\|_{\dB^{\f2p-\e}_{p,1}}+2^{\f{2M}p}\|a\|_{L^\infty}\|b\|_{\dB^{\f2p}_{p,1}}.
\end{split}
\eeno Thus choosing
$M=\Bigl[\f{p}{2+p\e}\log_2\f{\|a\|_{\dB^{\f2p+\e}_{p,\infty}}\|b\|_{\dB^{\f2p-\e}_{p,1}}}{\|a\|_{L^\infty}\|b\|_{\dB^{\f2p}_{p,1}}}\Bigr]$
in the above inequality gives rise to \beq\label{h.8}
\|T_ba\|_{\dB^{\f2p}_{p,1}}\lesssim
\bigl(\|a\|_{L^\infty}\|b\|_{\dB^{\f2p}_{p,1}}\bigr)^{\f{p\e}{2+p\e}}\bigl(\|a\|_{\dB^{\f2p+\e}_{p,\infty}}\|b\|_{\dB^{\f2p-\e}_{p,1}}\bigr)^{\f2{2+p\e}}.
\eeq This together with \eqref{h.7} proves the first inequality of
\eqref{h.6}.

Along the same line to proof of \eqref{h.8}, for any positive
integer $M,$ one has \beno \|T_ba\|_{\dB^{\f2p-\e}_{p,1}}\lesssim
\bigl(2^{-M\e}\|a\|_{\dB^{\f2p+\e}_{p,\infty}}+2^{(\f2p-\e)M}\|a\|_{L^\infty}\bigr)\|b\|_{\dB^{\f2p-\e}_{p,1}}.
\eeno Choosing
$M=\Bigl[\f2p\log_2\f{\|a\|_{\dB^{\f2p+\e}_{p,\infty}}}{\|a\|_{L^\infty}}\Bigr]$
in the above inequality leads to \beno
 \|T_ba\|_{\dB^{\f2p-\e}_{p,1}}\lesssim
 \|a\|_{L^\infty}^{\f{2\e}p}\|a\|_{\dB^{\f2p+\e}_{p,\infty}}^{1-\f{2\e}p}\|b\|_{\dB^{\f2p-\e}_{p,1}}.
 \eeno
This together with \eqref{h.7} proves the second inequality of
\eqref{h.6}.
\end{proof}

We now turn to the proof of Theorem \ref{thm1.3}.

\no\begin{proof}[Proof of Theorem \ref{thm1.3}] The proof of Theorem
\ref{thm1.3} essentially follows from that of Theorem \ref{thm1.2}.
For simplicity, we just present the {\it a priori} estimates for
smooth enough solution $(\r, u, \na\Pi)$ of \eqref{1.1}. We first
get, by a similar derivation \eqref{q.2a}, that \beno
\begin{split}
\|u&\|_{\widetilde L^\infty_t(\dot B^{-1+\frac{2}{p}-\e}_{p,1})}
+c\|u\|_{L^1_t(\dot B^{1+\frac{2}{p}-\e}_{p,1})}\\
 &\qquad\lesssim  \|u_0\|_{\dot
B^{-1+\frac{2}{p}-\e}_{p,1}} +\sum_{j\in\Z}2^{j(-1+\frac{2}{p}-\e)}
\|[\dot\Delta_j\mathbb{P};u\cdot\nabla]u\|_{L^1_t(L^p)}\\
&\qquad\qquad+\|(1-\r)\bigl(\pa_tu+u\cdot\nabla
u\bigr)\|_{L^1_t(\dot B^{1+\frac{2}{p}-\e}_{p,1})} +\|(\mu(\rho)-1)
d(u)\|_{L^1_t(\dot B^{\frac{2}{p}-\e}_{p,1})}.
\end{split}
\eeno However, it follows from  Lemma \ref{lemq.9}  that \beno
\begin{split}
\|[\dot\Delta_j\mathbb{P};u\cdot\nabla]u\|_{L^1_t(L^p)}\lesssim&
d_j2^{j(1-\f2p+\e)}\int_0^t\|\na
u(t')\|_{L^2}\|u(t')\|_{\dB^{\f2p-\e}_{p,1}}\,dt'\\
\lesssim& d_j2^{j(1-\f2p+\e)}\int_0^t\|\na
u(t')\|_{L^2}\|u(t')\|_{\dB^{-1+\f2p-\e}_{p,1}}^{\f12}\|u(t')\|_{\dB^{1+\f2p-\e}_{p,1}}^{\f12}\,dt',
\end{split}
\eeno so that one has
 \beno
\begin{split}
&\|u\|_{\widetilde L^\infty_t(\dot B^{-1+\frac{2}{p}-\e}_{p,1})}
+c\|u\|_{L^1_t(\dot B^{1+\frac{2}{p}-\e}_{p,1})}\\
 &\qquad\lesssim  C\Bigl\{\|u_0\|_{\dot
B^{-1+\frac{2}{p}-\e}_{p,1}}+\|(1-\r)\bigl(\pa_tu+u\cdot\nabla
u\bigr)\|_{L^1_t(\dot B^{1+\frac{2}{p}-\e}_{p,1})} +\|(\mu(\rho)-1)
d(u)\|_{L^1_t(\dot
B^{\frac{2}{p}-\e}_{p,1})}\\
&\qquad\qquad+\int_0^t\|\na
u(t')\|_{L^2}^2\|u(t')\|_{\dB^{-1+\f2p-\e}_{p,1}}\,dt'\Bigr\}+\f{c}2\|u\|_{L^1_t(\dot
B^{1+\frac{2}{p}-\e}_{p,1})}.
\end{split}
\eeno Applying Gronwall's inequality  gives rise to \beq\label{h.3}
\begin{split}
\|u\|_{\widetilde L^\infty_t(\dot B^{-1+\frac{2}{p}-\e}_{p,1})}
&+c\|u\|_{L^1_t(\dot B^{1+\frac{2}{p}-\e}_{p,1})} \leq
C\exp\bigl(C\|u_0\|_{L^2}^2\bigr)\Bigl\{\|u_0\|_{\dot
B^{-1+\frac{2}{p}-\e}_{p,1}}\\
&\qquad +\|(\mu(\rho)-1) d(u)\|_{L^1_t(\dot
B^{\frac{2}{p}-\e}_{p,1})} +\|(1-\r)\bigl(\pa_tu+u\cdot\nabla
u\bigr)\|_{L^1_t(\dot B^{-1+\frac{2}{p}-\e}_{p,1})}\Bigr\}.
\end{split}
\eeq Whereas we deduce, by a  similar derivation of  \eqref{q.6},
that
 \beq \label{h.4} \begin{split} \|\p_tu&\|_{\widetilde
L^\infty_t(\dot B^{-1+\frac{2}{p}-\e}_{p,1})}
+\|\na\Pi\|_{L^1_t(\dot
B^{-1+\frac{2}{p}-\e}_{p,1})}\\
 \leq & \|u_0\|_{\dot
B^{-1+\frac{2}{p}-\e}_{p,1}}+C\Bigl\{\|u_0\|_{L^2}\bigl(\|u\|_{\widetilde
L^\infty_t(\dot B^{-1+\frac{2}{p}-\e}_{p,1})} +\|u\|_{L^1_t(\dot
B^{1+\frac{2}{p}-\e}_{p,1})}\bigr)\\
&\qquad+\|(1-\r)\bigl(\pa_tu+u\cdot\nabla u\bigr)\|_{L^1_t(\dot
B^{-1+\frac{2}{p}-\e}_{p,1})} +\|(\mu(\rho)-1) d(u)\|_{L^1_t(\dot
B^{\frac{2}{p}-\e}_{p,1})}\Bigr\}.
\end{split}
\eeq For $\e$ sufficiently small, we denote
$$
\frak{A}_\e(t)\eqdefa\|u\|_{\widetilde L^\infty_t(\dot
B^{-1+\frac{2}{p}-\e}_{p,1})}+\|u\|_{L^1_t(\dot
B^{1+\frac{2}{p}-\e}_{p,1})}
+\f{\e}{1+\|u_0\|_{L^2}}\Bigl(\|\pa_tu\|_{L^1_t(\dot
B^{-1+\frac{2}{p}-\e}_{p,1})}+\|\nabla\Pi\|_{L^1_t(\dot
B^{-1+\frac{2}{p}-\e}_{p,1})}\Bigr).
$$
Then summing \eqref{h.3} with
$\f{\e}{1+\|u_0\|_{L^2}}\times$\eqref{h.4}, we obtain
\beq\label{h.5}
\begin{split}
\frak{A}_\e(t)\leq&
C\exp\bigl\{C\|u_0\|_{L^2}^2\bigr\}\Bigl\{\|u_0\|_{\dot
B^{-1+\frac{2}{p}-\e}_{p,1}} +\|(\mu(\rho)-1) d(u)\|_{L^1_t(\dot
B^{\frac{2}{p}-\e}_{p,1})} \\
&\qquad\qquad\qquad\qquad\qquad\qquad+\|(1-\r)\bigl(\pa_tu+u\cdot\nabla
u\bigr)\|_{L^1_t(\dot B^{-1+\frac{2}{p}-\e}_{p,1})}\Bigr\}.
\end{split}
\eeq Similar estimate holds for $\frak{A}(t)$ defined by
\eqref{q.6a}.

On the other hand, without loss of generality, we may assume that
$-1+\f2p-\e\neq 0,$ applying \eqref{q.4} and Lemma \ref{lemh.0} that
\beno
\begin{split}
 \|(1-\r)\bigl(&\pa_tu+u\cdot\nabla u\bigr)\|_{L^1_t(\dot
B^{-1+\frac{2}{p}-\e}_{p,1})}\\
\lesssim
&\Bigl(\|\r-1\|_{L^\infty_t(L^\infty)}+\|\r-1\|_{L^\infty_t(L^\infty)}^{1-\max(0,-1+\f2p-\e)\f{p}2}\|\r-1\|_{L^\infty_t(\dB^{\f2p}_{p,1})}^{\max(0,-1+\f2p-\e)\f{p}2}\\
&\quad+\|\r-1\|_{L^\infty_t(L^\infty)}^{\la(-1+\f2p-\e)}\|\r-1\|_{L^\infty_t(\dB^{\f2p}_{p,1})}^{1-\la(-1+\f2p-\e)}\Bigr)\|(\p_tu+u\cdot\na
u)\|_{L^1_t(\dB^{-1+\f2p-\e}_{p,1})}\\
\lesssim
&\|\r_0-1\|_{L^\infty}^{\th(\e)}\|\r-1\|_{L^\infty_t(\dB^{\f2p}_{p,1})}^{1-\th(\e)}\Bigl(\|\p_tu\|_{L^1_t(\dB^{-1+\f2p-\e}_{p,1})}\\
&\quad+\|u_0\|_{L^2} (\|u\|_{\widetilde L^\infty_t(\dot
B^{-1+\frac{2}{p}-\e}_{p,1})}+\|u\|_{L^1_t(\dot
B^{1+\frac{2}{p}-\e}_{p,1})})\Bigr),
\end{split}
\eeno where $\la(-1+\f2p-\e)$ is given by Lemma \ref{lemh.0} and
$\th(\e)\eqdefa\min\bigl\{1-\max(0,-1+\f2p-\e)\f{p}2,\la(-1+\f2p-\e)\bigr\}.$
Similar estimate holds for $\|(1-\r)\bigl(\pa_tu+u\cdot\nabla
u\bigr)\|_{L^1_t(\dot B^{-1+\frac{2}{p}}_{p,1})}.$

While we get, by applying Lemma \ref{lemh.1}, that \beno
\begin{split}
\|(\mu(\rho)-1) d(u)\|_{L^1_t(\dot B^{\frac{2}{p}-\e}_{p,1})}
\lesssim &\|u\|_{L^1_t(\dot
B^{1+\frac{2}{p}-\e}_{p,1})}\Bigl(\|\mu(\r)-1\|_{L^\infty_t(L^\infty)}\\
&\quad+\|\mu(\r)-1\|_{L^\infty_t(L^\infty)}^{\f{2\e}{p}}\|\mu(\r)-1\|_{L^\infty_t(\dB^{\f2p}_{p,1})}^{1-\f{2\e}{p}}
\Bigr)\\
\lesssim
&\|\r_0-1\|_{L^\infty}^{\f{2\e}p}\|\r-1\|_{L^\infty_t(\dB^{\f2p}_{p,1})}^{1-\f{2\e}p}\|u\|_{L^1_t(\dot
B^{1+\frac{2}{p}-\e}_{p,1})},
\end{split}
\eeno and
 \beno
\begin{split}
\|(\mu(\rho)&-1) d(u)\|_{L^1_t(\dot B^{\frac{2}{p}}_{p,1})} \lesssim
\|\mu(\r)-1\|_{L^\infty_t(L^\infty)}\|u\|_{L^1_t(\dot
B^{1+\frac{2}{p}}_{p,1})}\\
&+\bigl(\|\mu(\r)-1\|_{L^\infty_t(L^\infty)}\|u\|_{L^1_t(\dot
B^{1+\frac{2}{p}}_{p,1})}\bigr)^{\f{p\e}{2+p\e}}\bigl(\|\mu(\r)-1\|_{L^\infty_t(\dB^{\f2p+\e}_{p,1})}\|u\|_{L^1_t(\dot
B^{1+\frac{2}{p}-\e}_{p,1})}\bigr)^{\f{2}{2+p\e}}
\\
\lesssim
&\|\r_0-1\|_{L^\infty}^{\f{p\e}{2+p\e}}\bigl(\|\r_0-1\|_{L^\infty}^{\f2{2+p\e}}+\|\r-1\|_{L^\infty_t(\dB^{\f2p+\e}_{p,1})}^{\f2{2+p\e}}\bigr)\Bigl(\|u\|_{L^1_t(\dot
B^{1+\frac{2}{p}-\e}_{p,1})}+\|u\|_{L^1_t(\dot
B^{1+\frac{2}{p}}_{p,1})}\Bigr).
\end{split}
\eeno Therefore, we deduce from \eqref{h.5} that \beq\label{h.9}
\begin{split}
\frak{A}(t)+\frak{A}_\e(t)\leq& C\exp\bigl(C\|u_0\|_{L^2}^2\bigr)
\Bigl\{\|u_0\|_{\dB^{-1+\f2p-\e}_{p,1}}
+\|u_0\|_{\dB^{-1+\f2p}_{p,1}} +\|\r_0-1\|_{L^\infty}^{\d(\e)}
\bigl[1\\
&\qquad
+\|\r-1\|_{L^\infty_t(\dB^{\f2p}_{p,1})}
+\|\r-1\|_{L^\infty_t(\dB^{\f2p+\e}_{p,1})}\bigr]
\bigl(\frak{A}(t)+\frak{A}_\e(t)\bigr)\Bigr\}\\
\leq&
C\exp\bigl(C\|u_0\|_{L^2}^2\bigr)
\Bigl\{\|u_0\|_{\dB^{-1+\f2p-\e}_{p,1}}
+\|u_0\|_{\dB^{-1+\f2p}_{p,1}}
+\|\r_0-1\|_{L^\infty}^{\d(\e)}
\bigl[1\\
&\qquad+\|\r_0-1\|_{\dB^{\f2p}_{p,1}}
+\|\r_0-1\|_{\dB^{\f2p+\e}_{p,1}}\bigr]
\bigl(\frak{A}(t)+\frak{A}_\e(t)\bigr)\exp\bigl(C\frak{A}(t)\bigr)\Bigr\}.
\end{split}
\eeq In particular, if $\|\r_0-1\|_{L^\infty}$ is so small that \beq
\label{h.10}
\begin{split}
C\|&\r_0-1\|_{L^\infty}^{\d(\e)}
\bigl(1+\|\r_0-1\|_{\dB^{\f2p}_{p,1}}+\|\r_0-1\|_{\dB^{\f2p+\e}_{p,1}}\bigr)\\
&\qquad\times\exp\bigl(C\|u_0\|_{L^2}^2\bigr)\exp\Bigl\{2C\exp\bigl(C\|u_0\|_{L^2}^2
\bigr)\bigl(\|u_0\|_{\dB^{-1+\f2p-\e}_{p,1}}+\|u_0\|_{\dB^{-1+\f2p}_{p,1}}\bigr)\Bigr\}\leq\f12,
\end{split}
\eeq for $\d(\e)\eqdefa\min\bigl(\th(0),\th(\e),
\f{2\e}{p},\f{p\e}{2+p\e}\bigr),$ we infer from \eqref{h.9} that
\beno \frak{A}(t)+\frak{A}_\e(t)\leq2
C\exp\bigl(C\|u_0\|_{L^2}^2\bigr)\Bigl(\|u_0\|_{\dB^{-1+\f2p-\e}_{p,1}}+\|u_0\|_{\dB^{-1+\f2p}_{p,1}}\Bigr).
\eeno With this {\it a priori} estimate, we complete the proof to
the existence part of Theorem \ref{thm1.3}. The proof to the
uniqueness part is identical to that of Theorem \ref{thm1.2}. One
only needs to use Lemma \ref{lemh.0} and Lemma \ref{lemh.1} rather
than the standard product laws to estimate $\|(1-\r_0)\p_t\d \bar
u\|_{L^1_t(\dB^{-1+\f2p}_{p,1})}$ and $\|(\mu(\r_0)-1)\na_y\bar
u\|_{L^1_t(\dB^{\f2p}_{p,1})}.$ We skip the details here.
\end{proof}

\renewcommand{\theequation}{\thesection.\arabic{equation}}
\setcounter{equation}{0}
\appendix

\setcounter{equation}{0}
\section{Littlewood-Paley analysis}\label{abbendixb}

The proofs  of  Theorem \ref{thm1.1} to Theorem \ref{thm1.3}
requires Littlewood-Paley decomposition. Let us briefly explain how
it may be built in the case $x\in\R^d$ (see e.g. \cite{BCD}). Let
$\varphi$ be a smooth function  supported in the ring
$\mathcal{C}\eqdefa \{
\xi\in\R^d,\frac{3}{4}\leq|\xi|\leq\frac{8}{3}\}$  and $\chi(\xi)$
be a smooth function  supported in the ball $\mathcal{B}\eqdefa \{
\xi\in\R^d,\ |\xi|\leq\frac{4}{3}\}$ such that
\begin{equation*}
 \sum_{j\in\Z}\varphi(2^{-j}\xi)=1 \quad\hbox{for}\quad \xi\neq
 0\quad\mbox{and}\quad \chi(\xi)+ \sum_{q\geq 0}\varphi(2^{-q}\xi)=1\quad\hbox{for \ all }\quad
 \xi\in\R^d.
\end{equation*}
Then for $u\in{\mathcal S}'_h(\R^d)$  (see Definition 1.26 of
\cite{BCD}), which means $u\in\cS'(\R^d)$ and
$\lim_{j\to-\infty}\|\chi(2^{-j}D)u\|_{L^\infty}$ $=0,$ we set
\begin{equation}\label{dydic}
\begin{split}
&\forall\ j\in\Z,\quad \dot\Delta_j
u\eqdefa\varphi(2^{-j}\textnormal{D}) u\quad\mbox{and} \quad \dot
S_j u\eqdefa\chi(2^{-j}D)u,\\
&\forall\ q\geq 0,\quad \D_qu\eqdefa \varphi(2^{-q}\textnormal{D})u,
\quad \D_{-1}u\eqdefa \chi(D)u\quad\mbox{and}\quad S_qu\eqdefa
\sum_{-1\leq q'\leq q-1}\D_{q'}u,
\end{split}
\end{equation}
 we have the formal decomposition
\begin{equation}\label{dyadica}
u=\sum_{j\in\Z}\dot\Delta_j \,u,\quad\forall\,u\in {\mathcal
{S}}'_h(\R^d)\quad \mbox{and}\quad
u=\sum_{q\geq-1}\D_q\,u\quad\forall\, u\in \cS(\R^d).
\end{equation}
 Moreover, the
Littlewood-Paley decomposition satisfies the property of almost
orthogonality:
\begin{equation}\label{Pres_orth}
\dot\Delta_j\dot\Delta_q u\equiv 0 \quad\mbox{if}\quad| j-q|\geq 2
\quad\mbox{and}\quad\dot\Delta_j(\dot S_{q-1}u\dot\Delta_q u) \equiv
0\quad\mbox{if}\quad| j-q|\geq 5.
\end{equation}

We recall now the definition of homogeneous Besov spaces and
Bernstein type inequalities from \cite{BCD}. Similar definitions in
the inhomogeneous context can be found in \cite{BCD}.

\begin{defi}\label{def1.1}[Definition 2.15 of \cite{BCD}]
{\sl   Let $(p,r)\in[1,+\infty]^2,$ $s\in\R$ and $u\in{\mathcal
S}_h'(\R^3),$  we set
$$
\|u\|_{\dB^s_{p,r}}\eqdefa\Big(2^{js}\|\dot\Delta_j
u\|_{L^{p}}\Big)_{\ell ^{r}}.
$$
\begin{itemize}

\item
For $s<\frac{3}{p}$ (or $s=\frac{3}{p}$ if $r=1$), we define $ \dot
B^s_{p,r}(\R^3)\eqdefa \big\{u\in{\mathcal S}_h'(\R^3)\;\big|\; \|
u\|_{\dB^s_{p,r}}<\infty\big\}.$

\item
If $k\in\N$ and $\frac{3}{p}+k\leq s<\frac{3}{p}+k+1$ (or
$s=\frac{3}{p}+k+1$ if $r=1$), then $ \dB^s_{p,r}(\R^3)$ is defined
as the subset of distributions $u\in{\mathcal S}_h'(\R^3)$ such that
$\partial^\beta u\in \dB^{s-k}_{p,r}(\R^3)$ whenever $|\beta|=k.$
\end{itemize}}
\end{defi}

\begin{lem}\label{lem2.1} {\sl Let $\cB$ be a ball   and $\cC$ a ring of $\R^d.$
 A constant $C$ exists so
that for any positive real number $\d,$ any non-negative integer
$k,$ any smooth homogeneous function $\sigma$ of degree $m,$ and any
couple of real numbers $(a, \; b)$ with $ b \geq a \geq 1,$ there
hold
\begin{equation}
\begin{split}
&\Supp \hat{u} \subset \d \mathcal{B} \Rightarrow \sup_{|\alpha|=k}
\|\pa^{\alpha} u\|_{L^{b}} \leq  C^{k+1}
\d^{k+ d(\frac{1}{a}-\frac{1}{b} )}\|u\|_{L^{a}},\\
& \Supp \hat{u} \subset \d \mathcal{C} \Rightarrow C^{-1-k}\d^{
k}\|u\|_{L^{a}}\leq \sup_{|\alpha|=k}\|\partial^{\alpha} u\|_{L^{a}}
\leq
C^{1+k}\d^{ k}\|u\|_{L^{a}},\\
& \Supp \hat{u} \subset \d\mathcal{C} \Rightarrow \|\sigma(D)
u\|_{L^{b}}\leq C_{\sigma, m} \d^{ m+d(\frac{1}{a}-\frac{1}{b}
)}\|u\|_{L^{a}}.
\end{split}\label{2.1}
\end{equation}}
\end{lem}

We also recall  Bony's decomposition  from \cite{Bo}:
\begin{equation}
\begin{split}
&uv=T_u v+R(u,v)=T_u v+T_v u+\cR(u,v), \end{split}\label{bony}
\end{equation}
where
\begin{equation*}
\begin{split}
&T_u v\eqdefa\sum_{j\in \mathbb{Z}}\dot S_{j-1}u\dot\Delta_j
v,\qquad
R(u,v)\eqdefa\sum_{j\in\Z}\dot\Delta_j u \dot S_{j+2}v,\\
&\cR(u,v)\eqdefa\sum_{j\in\Z}\dot\Delta_j u
\dot{\widetilde{\Delta}}_{j}v\quad  \mbox{with}\quad
\dot{\widetilde{\Delta}}_{j}v\eqdefa \sum_{|j'-j|\leq
1}\dot\Delta_{j'}v.
\end{split}
\end{equation*}

To prove Theorem \ref{est.app}, we need the following lemma
concerning the commutator estimates, the proof of which is a
standard application of Basic Littlewood-Paley theory.

\begin{lem}\label{lemd.1}
{\sl Let $s>0,$ $a\in \dot H^{1+s}\cap Lip(\R^2)$ and $b\in
L^\infty\cap \mbox{Lip}\cap H^s(\R^2).$  Then there holds \beno
\begin{split}
&\|[\dot\D_j; a]\na b\|_{L^2}\lesssim c_j2^{-js}\bigl(\|\na
a\|_{L^\infty}\|b\|_{\dot H^s}+\|\na b\|_{L^\infty}\|a\|_{\dot
H^s}\bigr),\\
&\|[\dot\D_j; a]\na b\|_{L^2}\lesssim c_j2^{-js}\bigl(\|\na
a\|_{L^\infty}\|b\|_{\dot H^s}+\| b\|_{L^\infty}\|a\|_{\dot
H^{1+s}}\bigr),\\
&\|[\dot\D_j; a]\na b\|_{L^2}\lesssim c_j2^{-js}\bigl(\|\na
a\|_{L^\infty}\|b\|_{\dot H^s}+\|a\|_{\dot H^{1+s}}\|\na
b\|_{L^2}\bigr).
\end{split}
\eeno }\end{lem}

In  order to obtain a better description of the regularizing effect
of the transport-diffusion equation, we need to use Chemin-Lerner
type spaces $\widetilde{L}^{\lambda}_T(\dot B^s_{p,r}(\R^d))$ from
\cite{BCD}.
\begin{defi}\label{chaleur+}
Let $(r,\lambda,p)\in[1,\,+\infty]^3$ and $T\in]0,\,+\infty]$. We
define $\widetilde{L}^{\lambda}_T(\dot B^s_{p\,r}(\R^d))$ as the
completion of $C([0,T]; \,\cS(\R^d))$ by the norm
$$
\| f\|_{\widetilde{L}^{\lambda}_T(\dot B^s_{p,r})} \eqdefa
\Big(\sum_{j\in\Z}2^{jrs} \Big(\int_0^T\|\dot\Delta_j\,f(t)
\|_{L^p}^{\lambda}\, dt\Big)^{\frac{r}{\lambda}}\Big)^{\frac{1}{r}}
<\infty.
$$
with the usual change if $r=\infty.$  For short, we just denote this
space by $\widetilde{L}^{\lambda}_T(\dot B^s_{p,r}).$
\end{defi}

\noindent {\bf Acknowledgments.} We would like to thank Marius Paicu
for sending us the preprint \cite{HP1}, which motivates us to
improve our original local version of Theorem \ref{thm1.1} to the
present global version. We also thank Zhifei Zhang for profitable
discussions. Part of this work was done when we were visiting
Morningside Center of Mathematics, CAS, in the summer of 2012. We
appreciate the hospitality and the financial support from the
center.  P. Zhang is partially supported by NSF of China under Grant
10421101 and 10931007, the one hundred talents' plan from Chinese
Academy of Sciences under Grant GJHZ200829 and innovation grant from
National Center for Mathematics and Interdisciplinary Sciences.

\end{document}